\documentclass{article}

\title{\huge\textbf{A New Class of Non-Central Dirichlet Distributions}}
\author{
Carlo Orsi\footnote{e-mail: \mail{orsi.carlo@gmail.com}}
\\
\normalsize Ph.D. in Statistics at University of Milano-Bicocca (Milan, Italy),\\
\normalsize Casella Postale 13, Ufficio Postale Fermo, 63900 Fermo (Italy)
}
\usepackage[utf8]{inputenc}
\usepackage[english]{babel}
\usepackage{geometry}
\usepackage{fancyhdr}
\usepackage{hyperref}
\hypersetup{colorlinks=true, linkcolor=black, citecolor=black, urlcolor=black, plainpages=false}
\usepackage{color}
\usepackage{nameref}
\usepackage{graphicx}
\usepackage{subfigure}
\usepackage{float}
\usepackage{amsthm}
\usepackage{amssymb}
\usepackage{amsmath}
\usepackage{verbatim}
\usepackage{lscape}
\usepackage[font=small,labelfont=bf]{caption}
\usepackage{multirow}
\usepackage{wrapfig}
\newcommand{\mail}[1]{\href{mailto:#1}{\texttt{#1}}}

\usepackage[table]{xcolor}
\usepackage{array}
\newcolumntype{C}[1]{>{\centering\let\newline\\\arraybackslash\hspace{0pt}}m{#1}}
\newcolumntype{L}[1]{>{\raggedright\let\newline\\\arraybackslash\hspace{0pt}}m{#1}}
\newcolumntype{R}[1]{>{\raggedleft\let\newline\\\arraybackslash\hspace{0pt}}m{#1}}
\newtheoremstyle{plain}  
  {\topsep}   
  {\topsep}   
  {\itshape}  
  {}       
  {\bfseries} 
  {}         
  {\newline}  
  {}          
\theoremstyle{plain}
\newtheorem{property}{Property}[section]
\newtheorem{proposition}{\textbf{Proposition}}[section]
\newtheorem{definition}{Definition}[section]

\newtheoremstyle{mystyle}  
  {\topsep}   
  {\topsep}   
  {\normalfont}  
  {0pt}       
  {\bfseries} 
  {.}         
  {\newline}  
  {}          
\theoremstyle{mystyle}

\begin{document}
\baselineskip24pt
\maketitle

\begin{abstract}
\baselineskip24pt
In the present paper new light is shed on the non-central extensions of the Dirichlet distribution. Due to several probabilistic and inferential properties and to the easiness of parameter interpretation, the Dirichlet distribution proves the most well-known and widespread model on the unitary simplex. However, despite its many good features, such distribution is inadequate for modeling the data portions next to the vertices of the support due to the strictness of the limiting values of its joint density. To replace this gap, a new class of distributions, called Conditional Non-central Dirichlet, is presented herein. This new model stands out for being a more easily tractable version of the existing Non-central Dirichlet distribution which maintains the ability of this latter to capture the tails of the data by allowing its own density to have arbitrary positive and finite limits.

\vspace*{0.2cm}
\end{abstract}

\section{Introduction}
\label{sec:introduc}

Thanks to its simplicity and good mathematical properties, the Dirichlet distribution has historically represented the first tool for modeling data on the unitary simplex, this latter being the subset of $\mathbb{R}^{\, D}$, $D \geq 1$, given by
\begin{equation}
\mathcal{S}^{\, D}=\left\{\underline{x}=\left(x_1,\, \ldots \, ,x_D\right) \, : \; 0 < x_i <1 \, , \; i=1,\, \ldots \, ,D \; , \; \sum_{i=1}^{D} x_i<1\right\}
\label{eq:d.unit.simpl}
\end{equation}
and corresponding to the Real interval $(0,1)$ for $D=1$. Such data typically consist of components which represent parts of a whole and naturally arise in many areas of application such as medical, biological and social sciences. In this regard, a comprehensive review of the great variety of disciplines in which the above distribution has been employed is provided for example by \cite{NgTiaTan11}. Despite its several statistical properties fundamental for the analysis of this type of data, one of the weakest points of the Dirichlet family lies in the poorness of its parametrization. Amongst others, this fact implies low flexibility of the density at the vertices of Eq.~(\ref{eq:d.unit.simpl}). More precisely, only the $i$-th component of the parameter vector is devoted to modeling the limiting value of the Dirichlet density at the $i$-th vertex of the simplex, this latter being the Real vector whose elements are all equal to zero except for the $i$-th element which is one. This characteristic feature necessarily restricts the applicative potential of such distribution. In fact, the Dirichlet density tipically performs the identification of the tails of the data in an unsatisfactory way by showing limits equal to 0 or $+\infty$ in the cases that positive finite limits would be more appropriate instead. 

Several generalizations of the Dirichlet distribution have been proposed in the literature. However, only a few of them enable to overcome the aforementioned limitation of the Dirichlet by allowing their densities to take on arbitrary positive and finite limits which are potentially apt to properly model the data portions lying in the neighborhood of the vertices of the support. Specifically, the bivariate extension of the Kummer-Beta distribution \cite{BraOroNag11} and the Non-central Dirichlet distribution \cite{SanNagGup06} are right worthy of consideration in the above sense due to the major flexibility achieved by the limiting values of their densities by means of richer parametrizations. As a matter of fact, due to its further parameter, the density of the former model shows a more flexible behavior than the Dirichlet at the vertices of the unit simplex, whereas the density of the latter, despite its analytical hardness and poor interpretability, enables to fully overcome the above mentioned drawback of the Dirichlet thanks to its additional vector of non-centrality parameters.

That said, the present paper sheds new light on the variety of the probabilistic models on the unitary simplex which accommodate for a more sophisticated structure for the limits of the density by proposing a new non-central generalization of the Dirichlet distribution. This novel model, called Conditional Non-central Dirichlet, can be considered as a more tractable analogue of the existing Non-central Dirichlet model. In fact, the former preserves the ability of the latter to potentially capture the tails of the data by letting its density have arbitrary positive and finite limits and, at the same time, offers the advantage of having a density function which is more straightforward and easily handeable than the existing non-central one.

More precisely, our work is organized as follows. Section~\ref{sec:prelim} provides the mathematical tool kit, which essentially hinges on the notions of ascending factorial and generalized hypergeometric function. A glance at the unit simplex distributions involved in the definition and the analysis of the new model is given as well; particular attention is focused on the limiting values of the densities of such models on varying the parameter vector. An in-depth study of the novel model is carried out in Section~\ref{sec:new.ncdir}. Specifically, this latter is properly specified in Subsection~\ref{subsec:cncdir.def.distr}, where definition, distribution and theoretical properties are discussed in detail. In this regard, the Conditional Non-central Dirichlet density is proved to exhibit the same complex mixture type structure as the standard non-central one; however, from its perturbation representation, the greater tractability and interpretability of this new density surprisingly emerge. Then, an investigation of the mixing distribution of the mixture type form of the new Non-central Dirichlet density is reported in Subsection~\ref{subsec:cncdir.mix.distr}. Evidences of the similarities between the new and the standard non-central models are provided in Subsections~\ref{subsec:cncdir.repres} and~\ref{subsec:cncdir.dens.plots}, where some representations of the new model are presented and significant plots of its density are displayed, respectively. In particular, a special focus is given to the case of unitary shape parameters, under which the density of the new model shows the attractive feature of taking on arbitrary positive and finite limits at the vertices of the unit simplex. Subsection~\ref{subsec:cncdir.mom} derives a closed-form expression for the mixed raw moments, whereas the identifiability of the model is tackled in Subsection~\ref{subsec:cncdir.identif}. In Section~\ref{sec:cncdir.examples} an application to a real data set highlights the potential of the new model with respect to the aforementioned alternative ones. In conclusion, Section~\ref{sec:concl} contains some final remarks and Section~\ref{sec:r.funcs} provides an appendix which bears an implementation in the programming enviromnment \texttt{R} of the multiple infinite sum involved in the perturbation representation of the standard Non-central Dirichlet density.  


\section{Preliminaries}
\label{sec:prelim}

This section is intended to briefly review the probabilistic models involved in the present study. However, first of all we shall provide the mathematical tool kit that paves the way for a thorough analysis of the new distribution we are going to define. In this regard, it is essential remembering that
\begin{equation}
\left(a\right)_l=\frac{\Gamma\left(a+l\right)}{\Gamma\left(a\right)}=\left\{\begin{array}{ll} 1 & \mbox{ if } l=0 \\  \\ a\left(a+1\right) \, \ldots \, \left(a+l-1\right) & \mbox{ if } l \in \mathbb{N}\end{array} \right. 
\label{eq:poch.symb}
\end{equation}
is the $l$-th ascending factorial or Pochhammer's symbol of $a > 0$ \cite{JohKemKot05}, where $l \in \mathbb{N}$ indicates the number of terms and $\Gamma\left(\cdot\right)$ denotes the gamma function. Some interesting properties of the above notion are listed in the following. Firstly, by additively decomposing $l$ into $l_1 , \, l_2 \in \mathbb{N}$, in light of Eq.~(\ref{eq:poch.symb}) one has that
\begin{equation}
\left(a\right)_{l_1+l_2}=\frac{\Gamma\left(a+l_1+l_2\right)}{\Gamma\left(a\right)}=\left\{\begin{array}{l} \frac{\Gamma\left(a+l_1\right)}{\Gamma\left(a\right)} \, \frac{\Gamma\left(a+l_1+l_2\right)}{\Gamma\left(a+l_1\right)}=\left(a\right)_{l_1} \, \left(a+l_1\right)_{l_2} \\ \\ \frac{\Gamma\left(a+l_2\right)}{\Gamma\left(a\right)} \, \frac{\Gamma\left(a+l_2+l_1\right)}{\Gamma\left(a+l_2\right)}=\left(a\right)_{l_2} \, \left(a+l_2\right)_{l_1}\end{array} \right. \, .
\label{eq:poch.symb.sum}
\end{equation}
Then, the ratio of two ascending factorials of $a$ can be expressed as
\begin{equation}
\frac{\left(a\right)_{l_1}}{\left(a\right)_{l_2}}=\left\{\begin{array}{ll} \left(a+l_2\right)_{l_1-l_2} & \mbox{ if } l_1 \geq l_2 \\ \\ \left[\left(a+l_1\right)_{l_2-l_1}\right]^{-1} & \mbox{ if } l_1<l_2\end{array} \right. \, ,
\label{eq:poch.symb.ratio}
\end{equation}
whereas the following expansion holds true for the Pochhammer's symbol of a binomial:
\begin{equation}
\left(a+b\right)_{l}=\sum_{j=0}^{l} {l \choose j} \left(a\right)_{l-j} \left(b\right)_j \, .
\label{eq:poch.symb.binom}
\end{equation}

By virtue of the foregoing concept, the generalized hypergeometric function with $p$ upper parameters and $q$ lower parameters, $p, \, q \in \mathbb{N} \cup \{0\}$, can be accordingly defined as
\begin{equation}
_p^{\, }F_q^{}\left(a_1,\, \ldots \, ,a_p;b_1,\, \ldots \, ,b_q;x\right)=\sum_{i=0}^{+\infty} \frac{(a_1)_i \, \ldots \, (a_p)_i}{(b_1)_i \, \ldots \, (b_q)_i} \frac{x^{i}}{i \, !} \, , \quad x \in \mathbb{R} \, ;
\label{eq:fpq}
\end{equation}
for more details on the convergence of the hypergeometric series in Eq.~(\ref{eq:fpq}) as well as for further results and properties of $_p^{\, }F_q^{} \, $, the reader is recommended referring to \cite{SriKar85}. In this regard, the special case of Eq.~(\ref{eq:fpq}) corresponding to $p=0$ and $q=1$, namely 
\begin{equation}
_0^{\, }F_1^{}\left(b; x\right)=\sum_{i=0}^{+\infty} \frac{1}{(b)_i} \frac{x^{i}}{i \, !} \, , \quad x \in \mathbb{R} \, ,
\label{eq:f01}
\end{equation}
plays a prominent role in the analysis of the new model at study. Hence, some features of this latter that make it an extremely regular function are emphasised herein. Specifically, the function in Eq.~(\ref{eq:f01}) is increasing in $x \geq 0$ with values in $\left[1,+\infty\right)$ for any given $b>0$ and decreasing in $b$ for any given $x \geq 0$ as well as its first derivative with respect to $x$. Finally, another important special case of Eq.~(\ref{eq:fpq}) is
\begin{equation}
_1^{\, }F_1^{}\left(a; b; x\right)=\sum_{i=0}^{+\infty} \frac{(a)_i}{(b)_i} \frac{x^{i}}{i \, !} \, , \quad x \in \mathbb{R} \, ,
\label{eq:f11}
\end{equation}
which is otherwise called Kummer's confluent hypergeometric function; this latter satisfies the following transformation:
\begin{equation}
_1^{\, }F_1^{}\left(a; b; x\right)=e^x \, _1^{\, }F_1^{}\left(b-a; b; -x\right) \, ,
\label{eq:kummer.first.teo}
\end{equation}
which is known as Kummer's First Theorem.

That said, denoting independence by $\bot$, we remember that the $D$-dimensional Dirichlet model with vector of shape parameters $\underline{\alpha}=\left(\alpha_1,\, \ldots \, ,\alpha_{D+1}\right)$, $\alpha_i>0$, $i=1,\, \ldots \,,D+1$, denoted by $\mbox{Dir}^{\, D}\left(\underline{\alpha}\right)$, is defined as follows:
\begin{eqnarray}
\lefteqn{\left\{\begin{array}{l} Y_i \stackrel{\bot}{\sim } \chi^{\, 2}_{ 2 \alpha_i} \quad i=1,\, \ldots \,,D+1 \\ \\ Y^+=\sum_{i=1}^{D+1} Y_i  \end{array} \right. \quad \Rightarrow } \nonumber \\
& \Rightarrow & \quad \underline{X}=\left(X_1,\, \ldots \, ,X_D\right)=\left(\frac{Y_1}{Y^+},\, \ldots \, ,\frac{Y_D}{Y^+}\right) \sim \mbox{Dir}^{\, D}\left(\alpha_1,\, \ldots \, , \alpha_{D+1}\right)
\label{eq:dir.def}
\end{eqnarray}
\cite{KotBalJoh00} and its joint density function is given by
\begin{equation}
\mbox{Dir}^{\, D}\left(\underline{x};\underline{\alpha}\right)=\frac{\Gamma\left(\alpha^+\right)}{\prod_{i=1}^{D+1}\Gamma\left(\alpha_i\right)}\left[\prod_{i=1}^{D}x_i^{\alpha_i-1}\right] \left(1-\sum_{i=1}^{D}x_i\right)^{\alpha_{D+1}-1} \, , \qquad \underline{x} \in \mathcal{S}^{\, D} \, ,
\label{eq:dir.dens}
\end{equation}
where $\alpha^+=\sum_{i=1}^{D+1}\alpha_i$ and $\mathcal{S}^{\, D}$ is specified in Eq.~(\ref{eq:d.unit.simpl}). For the ends of this paper it is useful bearing in mind that, in the notation of Eq.~(\ref{eq:dir.def}), the Dirichlet distribution can be also obtained as conditional distribution of $\underline{X}$ given $Y^+$, the former being independent of the latter by virtue of a characterizing property of independent Gamma random variables \cite{Luk55}. The relevance of this property is remarkable; therefore, it is made explicit in the following.

\begin{property}[\cite{Luk55} Characterizing property of independent Gamma random variables]
\label{prope:char.prop.chisq}
Let $Y_1$, $Y_2$ be two nondegenerate and positive random variables and suppose that they are independently distributed. The random variables $Y^+=Y_1+Y_2$ and $X_i=Y_i \, / \, Y^+$, for every $i=1,2$, are independently distributed if and only if both $Y_1$ and $Y_2$ has Gamma distributions with the same scale parameter.
\end{property}
\noindent Clearly, Property~\ref{prope:char.prop.chisq} is valid also in the case that the number of the involved independent Gamma random variables is greater than two. In particular, by setting the common scale parameter of these latter equal to $1/2$, such property holds true also for any finite number of independent Chi-Squared random variables.   

Now let $D$ equal 2. The Dirichlet density, despite the great variety of shapes reachable by it on varying the shape parameters, shows poor flexibility at the vertices of the unit simplex. In this regard, the limiting values of the density of $\left(X_1,X_2\right) \sim \mbox{\normalfont{Dir}}^{\, 2}\left(\alpha_1,\alpha_2,\alpha_3\right)$ are as follows:
\begin{eqnarray}
\lefteqn{\mbox{\normalfont{Dir}}^{\, 2}\left(x_1,x_2;\alpha_1,\alpha_2,\alpha_3\right) \quad \rightarrow} \nonumber\\
& \stackrel{\left(x_1,\, x_2\right) \, \rightarrow \, \left(1,0\right)}{\rightarrow} & \left\{\begin{array}{ll}
\mbox{\normalfont{if} }\, \alpha_2+\alpha_3<2 \, : & + \, \infty\\
\mbox{\normalfont{if} }\, \alpha_2+\alpha_3=2 \, : & \left\{ \begin{array}{ll} \mbox{\normalfont{if} }\, \alpha_2=\alpha_3=1 \, : & \alpha_1\left(\alpha_1+1\right) \\ \mbox{\normalfont{otherwise:} }\, & \not\exists \end{array}\right.\\
\mbox{\normalfont{if} }\, \alpha_2+\alpha_3>2 \, : & 0 
\end{array}\right. \nonumber \\
& \stackrel{\left(x_1,\, x_2\right) \, \rightarrow \, \left(0,1\right)}{\rightarrow} & \left\{\begin{array}{ll}
\mbox{\normalfont{if} }\, \alpha_1+\alpha_3<2 \, : & + \, \infty\\
\mbox{\normalfont{if} }\, \alpha_1+\alpha_3=2 \, : & \left\{ \begin{array}{ll} \mbox{\normalfont{if} }\, \alpha_1=\alpha_3=1 \, : & \alpha_2\left(\alpha_2+1\right) \\ \mbox{\normalfont{otherwise:} }\, & \not\exists \end{array}\right.\\
\mbox{\normalfont{if} }\, \alpha_1+\alpha_3>2 \, : & 0 
\end{array}\right.  \nonumber \\
& \stackrel{\left(x_1,\, x_2\right) \, \rightarrow \, \left(0,0\right)}{\rightarrow} & \left\{\begin{array}{ll}
\mbox{\normalfont{if} }\, \alpha_1+\alpha_2<2 \, : & + \, \infty\\
\mbox{\normalfont{if} }\, \alpha_1+\alpha_2=2 \, : & \left\{ \begin{array}{ll} \mbox{\normalfont{if} }\, \alpha_1=\alpha_2=1 \, : & \alpha_3\left(\alpha_3+1\right) \\ \mbox{\normalfont{otherwise:} }\, & \not\exists \end{array}\right.\\
\mbox{\normalfont{if} }\, \alpha_1+\alpha_2>2 \, : & 0 
\end{array}\right.  \, ,
\label{eq:dir.dens.lims}
\end{eqnarray}
where the symbol $\not\exists \, $ indicates that the corresponding limit does not exist because depends on the direction followed to reach the accumulation point under consideration. From Eq.~(\ref{eq:dir.dens.lims}) it is noticeable that the limiting value at the $i$-th vertex of the simplex, when different from 0 or $+\infty$, depends only on the $i$-th shape parameter under unitary values for the remaining ones. Such state gets worse if $\alpha_i=1$ for every $i=1,2,3$; in fact, in this latter case the bivariate Dirichlet density reduces to the Uniform on $\mathcal{S}^2$ and all its limiting values equal 2. Finally, by Eq.~(\ref{eq:poch.symb}), the mixed raw moment of order $(r_1,r_2)$ of the $\mbox{\normalfont{Dir}}^{\, 2}\left(\alpha_1,\alpha_2,\alpha_3\right)$ distribution can be stated as
\begin{equation}
\mathbb{E}\left(X_1^{\, r_1} \, X_2^{\, r_2}\right)=\frac{\left(\alpha_1\right)_{r_1} \, \left(\alpha_2\right)_{r_2}}{\left(\alpha^+\right)_{r^+}} \, , \qquad \left\{\begin{array}{l} r_1, \, r_2 \in \mathbb{N} \\ \\ r^+=r_1+r_2 \end{array} \right. \, .
\label{eq:dir.mixr1r2mom}
\end{equation}

The bivariate Kummer-Beta distribution is the first generalization of the Dirichlet taken into account herein. This model is briefly recalled in the following for what is of interest in the present analysis; however, for further results and properties, the reader can be referred to \cite{BraOroNag11}. Specifically, a bidimensional random vector is said to follow a bivariate Kummer-Beta distribution with shape parameters $\alpha_1$, $\alpha_2$, $\alpha_3$ and additional parameter $\delta \in \mathbb{R}$, denoted by $\mbox{KB}^{\, 2}\left(\alpha_1,\alpha_2,\alpha_3,\delta\right)$, if its joint density can be expressed as the following infinite mixture of Dirichlet densities weighted by the probabilities obtained by normalizing the terms of the infinite sum in Eq.~(\ref{eq:f11}) where $a=\alpha_3$, $b=\alpha^+$, $x=\delta$:
\begin{eqnarray}
\lefteqn{\mbox{KB}^{\, 2}\left(x_1,x_2;\alpha_1,\alpha_2,\alpha_3,\delta\right)= \quad}\nonumber \\
& = & \sum_{j=0}^{+\infty} \frac{\frac{\left(\alpha_3\right)_j}{\left(\alpha^+\right)_j} \frac{\delta^j}{j!}}{_1F_1\left(\alpha_3;\alpha^+;\delta\right)} \; \mbox{Dir}^{\, 2}\left(x_1,x_2;\alpha_1,\alpha_2,\alpha_3+j\right) \, , \qquad \left(x_1,x_2\right) \in \mathcal{S}^{\, 2} \, . 
\label{eq:kb2.dens.mixt}
\end{eqnarray}
\noindent By simple computations, Eq.~(\ref{eq:kb2.dens.mixt}) can be equivalently stated in terms of the following perturbation of the bivariate Dirichlet density:
$$\mbox{KB}^{\, 2}\left(x_1,x_2;\alpha_1,\alpha_2,\alpha_3,\delta\right)=\mbox{Dir}^{\, 2}\left(x_1,x_2;\alpha_1,\alpha_2,\alpha_3\right) \,  \frac{e^{\, \left(1-x_1-x_2\right)\, \delta}}{_1F_1\left(\alpha_3;\alpha^+;\delta\right)} \, , \quad \left(x_1,x_2\right) \in \mathcal{S}^{\, 2} \, ,$$
which, in turn, by Eq.~(\ref{eq:kummer.first.teo}), can be rewritten in the final form reported by the above cited reference: 
\begin{eqnarray}
\lefteqn{\mbox{KB}^{\, 2}\left(x_1,x_2;\alpha_1,\alpha_2,\alpha_3,\delta\right)= \quad}\nonumber \\
& = & \mbox{Dir}^{\, 2}\left(x_1,x_2;\alpha_1,\alpha_2,\alpha_3\right) \,  \frac{e^{\, -\left(x_1+x_2\right)\, \delta}}{_1F_1\left(\alpha_1+\alpha_2;\alpha^+;-\delta\right)}\, , \quad \left(x_1,x_2\right) \in \mathcal{S}^{\, 2} \, . 
\label{eq:kb2.dens}
\end{eqnarray}
Under unitary shape parameters the behavior of the KB$^{\, 2}$ density at the vertices of the unit simplex turns out to be more flexible than that of the bivariate Dirichlet thanks to its additional parameter $\delta$. More precisely, the limits of the KB$^{\, 2}$ density take on the following values:
\begin{eqnarray*}
\lefteqn{\mbox{\normalfont{KB}}^{\, 2}\left(x_1,x_2;\alpha_1,\alpha_2,\alpha_3,\delta\right) \quad \rightarrow}\\
& \stackrel{\left(x_1,\, x_2\right) \, \rightarrow \, \left(1,0\right)}{\rightarrow} & \left\{\begin{array}{ll}
\mbox{\normalfont{if} }\, \alpha_2+\alpha_3<2 \, : & + \, \infty\\
\mbox{\normalfont{if} }\, \alpha_2+\alpha_3=2 \, : & \left\{ \begin{array}{ll} \mbox{\normalfont{if} }\, \alpha_2=\alpha_3=1 \, : & \frac{\alpha_1\left(\alpha_1+1\right) \, e^{-\delta}}{_1F_1\left(\alpha_1+1;\alpha_1+2;-\delta\right)} \\ \mbox{\normalfont{otherwise:} }\, & \not\exists \end{array}\right.\\
\mbox{\normalfont{if} }\, \alpha_2+\alpha_3>2 \, : & 0 
\end{array}\right. \\
& \stackrel{\left(x_1,\, x_2\right) \, \rightarrow \, \left(0,1\right)}{\rightarrow} & \left\{\begin{array}{ll}
\mbox{\normalfont{if} }\, \alpha_1+\alpha_3<2 \, : & + \, \infty\\
\mbox{\normalfont{if} }\, \alpha_1+\alpha_3=2 \, : & \left\{ \begin{array}{ll} \mbox{\normalfont{if} }\, \alpha_1=\alpha_3=1 \, : & \frac{\alpha_2\left(\alpha_2+1\right) \, e^{-\delta}}{_1F_1\left(1+\alpha_2;2+\alpha_2;-\delta\right)} \\ \mbox{\normalfont{otherwise:} }\, & \not\exists \end{array}\right.\\
\mbox{\normalfont{if} }\, \alpha_1+\alpha_3>2 \, : & 0 
\end{array}\right. \\
& \stackrel{\left(x_1,\, x_2\right) \, \rightarrow \, \left(0,0\right)}{\rightarrow} & \left\{\begin{array}{ll}
\mbox{\normalfont{if} }\, \alpha_1+\alpha_2<2 \, : & + \, \infty\\
\mbox{\normalfont{if} }\, \alpha_1+\alpha_2=2 \, : & \left\{ \begin{array}{ll} \mbox{\normalfont{if} }\, \alpha_1=\alpha_2=1 \, : & \frac{\alpha_3\left(\alpha_3+1\right)}{_1F_1\left(2;2+\alpha_3;-\delta\right)} \\ \mbox{\normalfont{otherwise:} }\, & \not\exists \end{array}\right.\\
\mbox{\normalfont{if} }\, \alpha_1+\alpha_2>2 \, : & 0 
\end{array}\right. \, ,
\end{eqnarray*}
where the symbol $\not\exists \,$ indicates that the corresponding limit does not exist because depends on the direction followed to reach the accumulation point under consideration.

The non-central extension of the Chi-Squared model \cite{JohKotBal95} represents the main ingredient for the definition and the analysis of the existing Non-central Dirichlet distribution, this latter being the second generalization of the Dirichlet considered in the present paper. Specifically, a Non-central Chi-Squared random variable $Y'$ with $g>0$ degrees of freedom and non-centrality parameter $\lambda \geq 0$, denoted by $\chi'^{\,2}_g \left(\lambda \right)$, can be characterized by means of the following mixture representation:
\begin{equation}
Y' \sim \chi'^{\,2}_g \left(\lambda \right) \qquad \Leftrightarrow \qquad Y'\,| \, M \; \sim \; \chi^{\, 2}_{g+2M} \, , \qquad \mbox{where} \; \; M \sim \mbox{Poisson}\left(\lambda/2\right) \, ,
\label{eq:mixrepres.ncchisq}
\end{equation}
the case $\lambda=0$ corresponding to the $\chi^{\, 2}_g$ distribution. Moreover, such a random variable can be additively decomposed into two independent parts, a central one with $g$ degrees of freedom and a purely non-central one with non-centrality parameter $\lambda$, namely
\begin{equation}
Y'=Y+\sum_{j=1}^{M}F_j \, , \quad \mbox{\normalsize{where}} \; \; \;  Y \sim \chi^{\, 2}_g \quad \bot \quad M \sim \mbox{\normalfont{Poisson}}\left(\lambda/2\right) \quad \bot \quad \{F_j \stackrel{\bot}{\sim } \chi^{\, 2}_2 \} \,  .
\label{eq:sumrepres.ncchisq}
\end{equation}
\noindent By virtue of Eq.~(\ref{eq:sumrepres.ncchisq}), the random variable $Y'_{pnc}=\sum_{j=1}^{M}F_j$ is said to have a Purely Non-central Chi-Squared distribution with non-centrality parameter $\lambda$. Indeed, we shall denote it by $\chi'^{\,2}_0 \left(\lambda \right)$, its number of degrees of freedom being equal to zero \cite{Sie79}. By Eq.~(\ref{eq:mixrepres.ncchisq}), the density function $f_{Y'}$ of $Y' \sim \chi'^{\,2}_g \left(\lambda \right)$ can be expressed as
\begin{equation}
f_{Y'}\left(y;g,\lambda\right)=\sum_{i=0}^{+\infty}\frac{e^{-\frac{\lambda}{2}}\left(\frac{\lambda}{2}\right)^i}{i!}
\frac{y^{\frac{g+2i}{2}-1} \, e^{-\frac{y}{2}}}{\Gamma\left(\frac{g+2i}{2}\right)2^{\frac{g+2i}{2}}}, \quad y>0 \, ,
\label{eq:dens.ncchisq}
\end{equation}
i.e. as the infinite series of the $\chi^2_{g+2i}$ densities, $i \in \mathbb{N} \cup \{0\}$, weighted by the probabilities of $M \sim \mbox{\normalfont{Poisson}}\left(\lambda/2\right)$. In this regard, the case $g=2$ is of prominent interest in the present setting; in fact, in this case the density in Eq.~(\ref{eq:dens.ncchisq}) exhibits a flexible limiting behavior on varying the non-centrality parameter by allowing its limit at $0$ to take on the following expression:
\begin{equation}
\lim_{y \rightarrow 0^+} f_{Y'}\left(y; 2,\lambda\right)=\frac{1}{2} \, e^{-\frac{\lambda}{2}} \, .
\label{eq:lim0.dens.ncchisq}
\end{equation}
Finally, the Non-central Chi-Squared distribution is reproductive with respect to both the number of degrees of freedom and the non-centrality parameter; specifically:
\begin{equation}
Y'_i \stackrel{\bot}{\sim } \chi'^{\, 2}_{g_i}(\lambda_i) \quad i=1,\, \ldots \, ,m  \qquad \Rightarrow \qquad \left\{\begin{array}{l} Y'^{+}=\sum_{i=1}^m Y'_i \sim \chi'^{\, 2}_{g^+}(\lambda^+) \\ \\ g^+=\sum_{i=1}^m g_i \, , \; \lambda^+=\sum_{i=1}^m \lambda_i \end{array} .\right.
\label{eq:ncchisq.reprod}
\end{equation}

That said, the $D$-dimensional Non-central Dirichlet model with vector of shape parameters $\underline{\alpha}=\left(\alpha_1,\, \ldots \, ,\alpha_{D+1}\right)$ and vector of non-centrality parameters $\underline{\lambda}=\left(\lambda_1,\, \ldots \, ,\lambda_{D+1}\right)$, $\lambda_i \geq 0$, $i=1,\, \ldots \,,D+1$, denoted by $\mbox{NcDir}^{\, D}\left(\underline{\alpha},\underline{\lambda}\right)$, can be easily defined by replacing the $Y_i \,$'s by $Y'_i \stackrel{\bot}{\sim } \chi'^{\,2}_{2 \alpha_i} \left(\lambda_i \right)$ in Eq.~(\ref{eq:dir.def}) as follows:
\begin{eqnarray}
\lefteqn{\left\{\begin{array}{l} Y'_i \stackrel{\bot}{\sim } \chi'^{\,2}_{2 \alpha_i} \left(\lambda_i \right) \quad i=1,\, \ldots \, ,D+1 \\ \\ Y'^{+}=\sum_{i=1}^{D+1} Y'_i  \end{array} \right. \Rightarrow } \nonumber \\
& \Rightarrow & \quad \underline{X}'=\left(X'_1,\, \ldots \, ,X'_D\right)=\left(\frac{Y'_1}{Y'^{+}},\, \ldots \, ,\frac{Y'_D}{Y'^{+}}\right) \sim \mbox{NcDir}^{\, D}\left(\underline{\alpha},\underline{\lambda}\right)
\label{eq:ncdir.def}
\end{eqnarray}
\noindent \cite{SanNagGup06}. The density function of the $\mbox{NcDir}^{\, D}\left(\underline{\alpha},\underline{\lambda}\right)$ distribution can be simply derived by following the next arguments; in this regard, let:
\begin{eqnarray}
\lefteqn{\underline{M}=\left(M_1,\, \ldots \, ,M_{D+1}\right) \sim \mbox{\normalfont{Multi-Poisson}}^{\, D+1}(\underline{\lambda} \, / \, 2) \qquad \Leftrightarrow}\nonumber \\
& \qquad \Leftrightarrow \qquad & M_i \stackrel{\bot}{\sim } \mbox{\normalfont{Poisson}}(\lambda_i \, / \, 2) \, , \; \quad  i=1,\, \ldots \, ,D+1 \; .
\label{eq:multipois.def}
\end{eqnarray}
Then, in the notation of Eq.~(\ref{eq:ncdir.def}), by Eq.~(\ref{eq:mixrepres.ncchisq}) the conditional distribution of $Y'_i$ given $\underline{M}$ is of the type $\chi^{\, 2}_{2 \alpha_i+2 M_i}$, $i=1,\, \ldots \, ,D+1$; therefore, by Eq.~(\ref{eq:dir.def}):
\begin{equation}
\underline{X}' \sim \mbox{NcDir}^{\, D}\left(\underline{\alpha},\underline{\lambda}\right) \qquad \Leftrightarrow \qquad \left\{\begin{array}{l} \left. \underline{X}' \, \right| \, \underline{M}  \sim \mbox{Dir}^{\, D}\left(\underline{\alpha}+\underline{M}\right) \\ \\ \underline{M} \sim \mbox{\normalfont{Multi-Poisson}}^{\, D+1}(\underline{\lambda} \, / \, 2) \end{array} \right. \, ,
\label{eq:ncdir.mixt.repres}
\end{equation}
which is the mixture representation of the NcDir distribution. Hence, the joint density of $\underline{X}' \sim \mbox{NcDir}^{\, D}\left(\underline{\alpha},\underline{\lambda}\right)$ can be stated as
\begin{equation}
\mbox{NcDir}^{\, D}\left(\underline{x};\underline{\alpha},\underline{\lambda}\right)=\sum_{\underline{j} \in \mathbb{N}_0^{\, D+1}} \left[\Pr\left(\underline{M}=\underline{j}\right) \cdot \mbox{Dir}^{\, D}\left(\underline{x};\underline{\alpha}+\underline{j}\right)\right] \, , \qquad \underline{x} \in \mathcal{S}^{\, D} \, ,
\label{eq:ncdir.dens}
\end{equation}
i.e. as the multiple infinite series of the $\mbox{Dir}^{\, D}\left(\underline{\alpha}+\underline{j}\right)$ densities, $\underline{j} \in \mathbb{N}_0^{\, D+1}$, weighted by the joint probabilities of the random vector $\underline{M}$ defined in Eq.~(\ref{eq:multipois.def}). Moreover, the function in Eq.~(\ref{eq:ncdir.dens}) can be equivalently expressed in terms of perturbation of the corresponding central case as follows:
\begin{eqnarray}
\lefteqn{\mbox{NcDir}^{\, D}\left(\underline{x};\underline{\alpha},\underline{\lambda}\right)=\mbox{Dir}^{\, D}\left(\underline{x};\underline{\alpha}\right) \quad \cdot} \nonumber \\
& \cdot & e^{-\frac{\lambda^+}{2}} \,  \Psi_2^{\, (D+1)}\left[\alpha^+;\underline{\alpha};\frac{\lambda_1}{2}x_1,\, \ldots \, ,\frac{\lambda_D }{2}x_D,\frac{\lambda_{D+1}}{2}\left(1-\sum_{i=1}^{D}x_i\right)\right] \, , \qquad \underline{x} \in \mathcal{S}^{\, D} \qquad
\label{eq:ncdir.perturb.dens}
\end{eqnarray}
where 
\begin{equation}
\Psi_2^{\, (m)}\left[a;b_1,\, \ldots \, ,b_m;x_1,\, \ldots \, ,x_m\right]=\sum_{j_1,\, \ldots \, , \, j_m= \, 0}^{+\infty}\frac{(a)_{j_1+\, \ldots \, +j_m}}{(b_1)_{j_1} \, \ldots \, (b_m)_{j_m}} \, \frac{x_1^{\, j_1}}{j_1!} \, \ldots \, \frac{x_m^{\, j_m}}{j_m!}
\label{eq:ncdir.perturb}
\end{equation}
is the $m$-dimensional ($m>2$) generalization of the Humbert's confluent hypergeometric function $\Psi_2$ \cite{SriKar85}. Unfortunately, the perturbation representation of the Non-central Dirichlet density in Eq.~(\ref{eq:ncdir.perturb.dens}) highlights the uneasy tractability and interpretability of this latter. Indeed, regardless of the constant term, the Dirichlet density is perturbed by a function in $D+1$ variables given by the sum of the multiple power series in Eq.~(\ref{eq:ncdir.perturb}), that, as far as we know, cannot be reduced to a more easily tractable analytical form. In this regard, the code of a routine implemented in the programming environment \texttt{R} for the computation of the function in Eq.~(\ref{eq:ncdir.perturb}) with $m=3$ is provided in the appendix contained in Section~\ref{sec:r.funcs}.\\
\noindent Now let $D$ equal 2. By simple computations it is easy to see that, despite its poor tractability from a mathematical standpoint, the limiting values of the NcDir$^{\, 2}$ density interestingly take on the following simple expressions:
\begin{eqnarray*}
\lefteqn{\mbox{\normalfont{NcDir}}^{\, 2}\left(x_1,x_2;\alpha_1,\alpha_2,\alpha_3,\lambda_1,\lambda_2,\lambda_3\right) \quad \rightarrow}\\
& \stackrel{\left(x_1,\, x_2\right) \, \rightarrow \, \left(1,0\right)}{\rightarrow} & \left\{\begin{array}{ll}
\mbox{\normalfont{if} }\, \alpha_2+\alpha_3<2 \, : & + \, \infty\\
\mbox{\normalfont{if} }\, \alpha_2+\alpha_3=2 \, : & \left\{ \begin{array}{ll} \mbox{\normalfont{if} }\, \alpha_2=\alpha_3=1 \, : & e^{\, -\frac{\lambda_2+\lambda_3}{2}} \cdot \\ & \left[\left(\frac{\lambda_1}{2}\right)^2 +\left(\alpha_1+1\right)\left(\alpha_1+\lambda_1\right)\right]  \\ \mbox{\normalfont{otherwise:} }\, & \not\exists \end{array}\right.\\
\mbox{\normalfont{if} }\, \alpha_2+\alpha_3>2 \, : & 0 
\end{array}\right. \\
& \stackrel{\left(x_1,\, x_2\right) \, \rightarrow \, \left(0,1\right)}{\rightarrow} & \left\{\begin{array}{ll}
\mbox{\normalfont{if} }\, \alpha_1+\alpha_3<2 \, : & + \, \infty\\
\mbox{\normalfont{if} }\, \alpha_1+\alpha_3=2 \, : & \left\{ \begin{array}{ll} \mbox{\normalfont{if} }\, \alpha_1=\alpha_3=1 \, : & e^{\, -\frac{\lambda_1+\lambda_3}{2}} \cdot \\ & \left[\left(\frac{\lambda_2}{2}\right)^2 +\left(\alpha_2+1\right)\left(\alpha_2+\lambda_2\right)\right] \\ \mbox{\normalfont{otherwise:} }\, & \not\exists \end{array}\right.\\
\mbox{\normalfont{if} }\, \alpha_1+\alpha_3>2 \, : & 0 
\end{array}\right. \\
& \stackrel{\left(x_1,\, x_2\right) \, \rightarrow \, \left(0,0\right)}{\rightarrow} & \left\{\begin{array}{ll}
\mbox{\normalfont{if} }\, \alpha_1+\alpha_2<2 \, : & + \, \infty\\
\mbox{\normalfont{if} }\, \alpha_1+\alpha_2=2 \, : & \left\{ \begin{array}{ll} \mbox{\normalfont{if} }\, \alpha_1=\alpha_2=1 \, : & e^{\, -\frac{\lambda_1+\lambda_2}{2}} \cdot \\ & \left[\left(\frac{\lambda_3}{2}\right)^2 +\left(\alpha_3+1\right)\left(\alpha_3+\lambda_3\right)\right] \\ \mbox{\normalfont{otherwise:} }\, & \not\exists \end{array}\right.\\
\mbox{\normalfont{if} }\, \alpha_1+\alpha_2>2 \, : & 0 
\end{array}\right.
\end{eqnarray*}
Hence, thanks to its richer parametrization based on three non-centrality parameters, the density of the bivariate Non-central Dirichlet distribution enables to overcome the aforementioned limitations of the Dirichlet by allowing its own limits at the vertices of the unit simplex to have arbitrary positive and finite values under unitary shape parameters. This feature, which results from the flexibility of the limit at 0 of the Non-central Chi-Squared density previously noticed in Eq.~(\ref{eq:lim0.dens.ncchisq}), makes the NcDir model potentially apt to properly capture the data portions having values next to the vertices of the support.


\section{A new Non-central Dirichlet model}
\label{sec:new.ncdir}
 

\subsection{Definition and distribution}
\label{subsec:cncdir.def.distr}

The above recalled characterizing property of independent Gamma random variables is no longer valid in the non-central setting. This fact covers a key role in the present paper and, more specifically, is assumed as the starting point in the derivation of the new family of non-central generalizations of the Dirichlet distribution we are interested in. Indeed, in the notation of Eq.~(\ref{eq:ncdir.def}), a new Non-central Dirichlet model can be achieved by conditioning the existing one $\underline{X}'$ on the sum $Y'^+$ of the $D+1$ independent Non-central Chi-Squared random variables involved in its definition. Such a model is thus distributed according to the conditional distribution of $\underline{X}'$ given $Y'^+$; moreover, as $\underline{X}'$ is not independent of $Y'^+$, this distribution must be different from the $\mbox{NcDir}^{\, D}\left(\underline{\alpha},\underline{\lambda}\right)$ except for $\underline{\lambda}=\underline{0}$, this latter case corresponding to the Dirichlet. In this regard, the conditional density of $\underline{X}'$ given $Y'^+=y$ can be obtained for any fixed $y>0$ by mixing the conditional density of $\underline{X}'$ given $(\underline{M}, Y'^+=y)$ with respect to the conditional probability mass function of $\underline{M}$ given $Y'^+=y$. As $\underline{X}'$ is conditionally independent of $Y'^+$ given $\underline{M}$ by Property~\ref{prope:char.prop.chisq}, in light of Eq.~(\ref{eq:ncdir.mixt.repres}) we have that $\underline{X}'\left| \right.(\underline{M}, Y'^+=y) \sim \mbox{Dir}^{\, D}\left(\underline{\alpha}+\underline{M}\right)$, while a direct application of Bayes' Theorem shows that:
\begin{eqnarray}
\lefteqn{\Pr\left(\left. \underline{M}=\underline{j} \; \right| \; Y'^+=y \right)=} \nonumber \\
& = & \Pr\left(\left. M_1=j_1 \, , \, \ldots \, , \, M_{D+1}=j_{D+1} \; \right| \; Y'^+=y \right) \qquad \left\{\begin{array}{l} j_i \in \mathbb{N}_0 \, , \quad i=1,\, \ldots \, ,D+1 \\ j^+=\sum_{i=1}^{D+1} j_i \end{array}  \right. \nonumber \\
& = & \frac{1}{_0F_1\left(\alpha^+; \frac{y \, \lambda^+}{4}\right)} \, \frac{1}{\left(\alpha^+\right)_{j^+}}  \prod_{i=1}^{D+1}\frac{\left(\frac{y \, \lambda_{i}}{4}\right)^{j_i}}{j_i \, !} \; . 
\label{eq:cncdir.mix.distr}
\end{eqnarray}
Therefore, by incorporating the conditioning value $y$ in the $\lambda_{i}$'s, without loss of generality we are led to the following definitions.

\begin{definition}[$D$-variate Conditional Non-central Dirichlet distribution]
\label{def:cncdir.def}
The $D$-variate Conditional Non-central Dirichlet distribution with vector of shape parameters $\underline{\alpha}=\left(\alpha_1,\, \ldots \, ,\alpha_{D+1}\right)$ and vector of non-centrality parameters $\underline{\lambda}=\left(\lambda_1,\, \ldots \, ,\lambda_{D+1}\right)$, denoted by $\mbox{\normalfont{CNcDir}}^{\, D}\left(\underline{\alpha},\underline{\lambda}\right)$, is the distribution of the $D$-dimensional random vector $\underline{X}''=\left(X''_1, \, \ldots \, , X''_D\right)$ defined as follows:
\begin{eqnarray}
\lefteqn{\left\{\begin{array}{l} Y'_i \stackrel{\bot}{\sim } \chi'^{\,2}_{2 \alpha_i} \left(\lambda_i \right) \quad i=1,\, \ldots \, ,D+1 \\ \\ Y'^{+}=\sum_{i=1}^{D+1} Y'_i  \end{array} \right. \quad \Rightarrow } \nonumber \\
& \Rightarrow & \quad \underline{X}''=\left.\left(\frac{Y'_1}{Y'^{+}},\, \ldots \, ,\frac{Y'_D}{Y'^{+}}\right) \, \right| Y'^+=1 \quad \sim \;  \mbox{\normalfont{CNcDir}}^{\, D}\left(\underline{\alpha},\underline{\lambda}\right) \, .
\label{eq:cncdir.def}
\end{eqnarray}
\end{definition}

\begin{definition}[$(D+1)$-variate Mixture Weight distribution]
\label{def:multivar.mw}
The $(D+1)$-variate Mixture Weight distribution with shape parameter $\alpha^+$ and vector of non-centrality parameters $\underline{\lambda}=\left(\lambda_1,\, \ldots \, ,\lambda_{D+1}\right)$, denoted by $\mbox{\normalfont{MW}}^{\, D+1}\left(\alpha^+,\underline{\lambda}\right)$, is the distribution of the $(D+1)$-dimensional random vector $\underline{N}=\left(N_1, \, \ldots \, , N_{D+1}\right)$ obtained as follows:
\begin{eqnarray}
\lefteqn{\left\{\begin{array}{l} \underline{M} \sim \mbox{\normalfont{Multi-Poisson}}^{\, D+1}(\underline{\lambda} \, / \, 2) \\ \\ Y'_i \stackrel{\bot}{\sim } \chi'^{\,2}_{2 \alpha_i} \left(\lambda_i \right) \, , \; \; i=1,\, \ldots \, ,D+1 \\ \\   Y'^{+}=\sum_{i=1}^{D+1} Y'_i  \end{array} \right. \quad \Rightarrow}\nonumber \\
& \qquad \qquad \qquad \Rightarrow \qquad & \underline{N}=\left. \underline{M} \; \right| \, Y'^+=1 \quad \sim \; \mbox{\normalfont{MW}}^{\, D+1}\left(\alpha^+,\underline{\lambda}\right) \; .
\label{eq:multivar.mw}
\end{eqnarray}
\end{definition}
\noindent The special case of Eq.~(\ref{eq:multivar.mw}) where $\underline{\lambda}$ is set equal to $\underline{0}$ leads to a random vector degenerate at $\underline{0}$. By Eq.~(\ref{eq:cncdir.mix.distr}) and in view of the foregoing arguments, the joint probability mass function of the $\mbox{\normalfont{MW}}^{\, D+1}\left(\alpha^+,\underline{\lambda}\right)$ distribution takes on the following form: 
\begin{eqnarray}
\lefteqn{\Pr\left(\underline{N}=\underline{j}\right) \qquad =}\nonumber \\
& = & \frac{1}{_0F_1\left(\alpha^+; \frac{\lambda^+}{4}\right)} \, \frac{1}{\left(\alpha^+\right)_{j^+}} \prod_{i=1}^{D+1}\frac{\left(\frac{\lambda_{i}}{4}\right)^{j_i}}{j_i \, !} \, , \qquad \left\{\begin{array}{l} \underline{j}=\left(j_1,	\, \ldots \, ,j_{D+1}\right) \in \mathbb{N}_0^{\, D+1} \\ \\ j^+=\sum_{i=1}^{D+1} j_i \end{array} \right. .
\label{eq:multidim.peso.prob}
\end{eqnarray}
\noindent Therefore, the density function of $\underline{X}'' \sim \mbox{\normalfont{CNcDir}}^{\, D}\left(\underline{\alpha},\underline{\lambda}\right)$ can be accordingly expressed as follows:
\begin{equation}
\mbox{CNcDir}^{\, D}\left(\underline{x};\underline{\alpha},\underline{\lambda}\right)=\sum_{\underline{j} \in \mathbb{N}_0^{\, D+1}} \left[\Pr\left(\underline{N}=\underline{j}\right) \cdot \mbox{Dir}^{\, D}\left(\underline{x};\underline{\alpha}+\underline{j}\right)\right] \, , \qquad \underline{x} \in \mathcal{S}^{\, D} \, ,
\label{eq:cncdir.dens}
\end{equation}
i.e. as the multiple infinite series of the $\mbox{Dir}^{\, D}\left(\underline{\alpha}+\underline{j}\right)$ densities, $\underline{j} \in \mathbb{N}_0^{\, D+1}$, weighted by the joint probabilities of the random vector $\underline{N}$ in Eq.~(\ref{eq:multidim.peso.prob}). Hence, the $\mbox{\normalfont{CNcDir}}$ density shows the same mixture type form as the $\mbox{\normalfont{NcDir}}$ one; the only difference lies in the mixing distribution, given by the $\mbox{MW}$ for the former and the $\mbox{\normalfont{Multi-Poisson}}$ for the latter. In this regard, an overview of the properties of the multivariate Mixture Weight distribution is deferred to the next subsection. More interestingly, by simple computations it is easy to see that the CNcDir density turns out to be equivalently stated in terms of the following perturbation of the corresponding central case:
\begin{eqnarray}
\lefteqn{\mbox{CNcDir}^{\, D}\left(\underline{x};\underline{\alpha},\underline{\lambda}\right) \quad = \quad \mbox{Dir}^{\, D}\left(\underline{x};\underline{\alpha}\right) \cdot} \nonumber \\
& \cdot & \frac{\left[\prod_{i=1}^{D} \, _0F_1\left(\alpha_i;\frac{\lambda_i}{4} x_i\right)\right]  \, _0F_1\left[\alpha_{D+1};\frac{\lambda_{D+1}}{4}\left(1-\sum_{i=1}^{D}x_i\right)\right]}{_0F_1\left(\alpha^+;\frac{\lambda^+}{4}\right)} \, , \qquad \underline{x} \in \mathcal{S}^D \, .
\label{eq:cncdir.perturb.dens}
\end{eqnarray}
From Eq.~(\ref{eq:cncdir.perturb.dens}) the major tractability and interpretability of the $\mbox{\normalfont{CNcDir}}$ density over the $\mbox{\normalfont{NcDir}}$ one prove evident. Contrarily to the standard non-central case where the perturbing factor of the Dirichlet density shows an inner structure that is too complex to be handled analytically, with reference to the density of the conditional model the perturbing effect can be clearly noticed for each value of the parameter vector. More precisely, in Eq.~(\ref{eq:cncdir.perturb.dens}), regardless of the constant term, the Dirichlet density is perturbed by the product of $D+1$ generalized hypergeometric functions $_0F_1$ which show perfectly symmetric behaviors. Specifically, as $\lambda_i$ gets higher, $i=1,\, \ldots \, , D+1$, the corresponding function $_0F_1$ gives more weight to the tail of the Dirichlet density relative to the $i$-th vertex of the unit simplex; moreover, the less is the corresponding shape parameter $\alpha_i$, the larger is the extent of this phenomenon.

The Conditional Non-central Dirichlet model shares with the Dirichlet and the existing Non-central Dirichlet some important properties that are discussed below.

The first one is the closure under permutation, an apparently trivial property which expresses the fact that the $\mbox{\normalfont{CNcDir}}$ distribution treats all its components in a completely symmetric way.
\begin{proposition}[Closure under permutation]
\label{propo:cncdir.permut}
Let $\underline{X}''=\left(X''_1, \, \ldots \, , X''_D\right) \sim \mbox{\normalfont{CNcDir}}^{\, D}\left(\underline{\alpha},\underline{\lambda}\right)$ where $\underline{\alpha}=(\alpha_1,\, \ldots \, ,\alpha_D,\alpha_{D+1})$,  $\underline{\lambda}=(\lambda_1,\ldots,$ $\lambda_D,\lambda_{D+1})$ and $(\underline{k})=(k_{(1)},\, \ldots \,, k_{(D)},k_{(D+1)})$ be a permutation of $(1,\, \ldots \, ,D,D+1)$. Then, $\underline{X}''_{(\underline{k})}=(X''_{k_{(1)}}, \, \ldots \, ,X''_{k_{(D)}}) \sim \mbox{\normalfont{CNcDir}}^{\, D} (\underline{\alpha}_{(\underline{k})},\underline{\lambda}_{(\underline{k})})$ where $\underline{\alpha}_{(\underline{k})}=(\alpha_{k_{(1)}},\, \ldots \, ,\alpha_{k_{(D)}},$ $\alpha_{k_{(D+1)}})$ and $\underline{\lambda}_{(\underline{k})}=(\lambda_{k_{(1)}},\, \ldots \, ,\lambda_{k_{(D)}},\lambda_{k_{(D+1)}})$.
\end{proposition}
\begin{proof}
The proof directly follows from the definition of the model at study in Eq.~(\ref{eq:cncdir.def}) by applying the permutation $(\underline{k})$ to the $Y'_i$'s.
\end{proof}

Then, like the Dirichlet and the Non-central Dirichlet, the present model satisfies the aggregation property, which is crucial to the derivation of the results established in the following.
\begin{proposition}[Aggregation property]
\label{propo:cncdir.aggreg.prop}
Let $\underline{X}''=\left(X''_1, \, \ldots \, , X''_D\right) \sim \mbox{\normalfont{CNcDir}}^{\, D} \left(\underline{\alpha},\underline{\lambda}\right)$ where $\underline{\alpha}=\left(\alpha_1, \, \ldots \, ,\alpha_D, \alpha_{D+1}\right)$, $\underline{\lambda}=\left(\lambda_1,\ldots,\right.$ $\left.\lambda_D,\lambda_{D+1}\right)$. Let $\mathcal{P}=\left\{A_1,\, \ldots \, ,A_m, A_{m+1}\right\}$ be a partition of $\left\{1,\, \ldots \, ,D,D+1\right\}$, $1 \leq m \leq D$. Then:
\begin{equation}
\left(\sum_{i \in A_1} X''_i, \, \ldots \, , \sum_{i \in A_m} X''_i\right) \sim \mbox{\normalfont{CNcDir}}^{\, m}\left(\sum_{i \in A_1} \alpha_i, \, \ldots \, , \sum_{i \in A_{m+1}} \alpha_i,\sum_{i \in A_1} \lambda_i, \, \ldots \, , \sum_{i \in A_{m+1}} \lambda_i\right)
\label{eq:cncdir.aggreg.prop}
\end{equation}
\end{proposition}
\begin{proof}
The proof directly ensues from the definition of the model at study in Eq.~(\ref{eq:cncdir.def}) and from the reproductive property of the Non-central Chi-Squared distribution in Eq.~(\ref{eq:ncchisq.reprod}) by applying the partition $\mathcal{P}$ to the $Y'_i$'s.
\end{proof}

\noindent By suitably resorting to Proposition~\ref{propo:cncdir.aggreg.prop}, the $k$-dimensional marginals of the $D$-dimensional Conditional Non-central Dirichlet model can be simply achieved for every $k=1,\ldots,D-1$. Specifically, like the Dirichlet and the existing Non-central Dirichlet models, the Conditional Non-central Dirichlet is closed under marginalization and simple relationships hold between the parameters of the joint and the marginal distributions. In particular, the univariate marginals are of Conditional Doubly Non-central Beta (CDNcB) type, this latter being the analogue of the CNcDir on the Real interval $(0,1)$. In this regard, a first analysis and an in-depth study of the CDNcB distribution are given in \cite{OngOrs15} and \cite{Ors21}, respectively. In the following, only the univariate and bivariate marginals are made explicit but as far as the investigation of the $k$-dimensional marginals, $k>2$, is concerned, similar results apply.

\begin{proposition}[Univariate and bivariate marginals]
\label{propo:cncdir.margs}
Let $\underline{X}'' \sim \mbox{\normalfont{CNcDir}}^{\, D}\left(\underline{\alpha},\underline{\lambda}\right)$. Then, for every $p,\, q=1, \, \ldots \, ,D$, $p \neq q$:
\begin{equation}
X''_p \sim \mbox{\normalfont{CDNcB}}\left(\alpha_p, \alpha^+-\alpha_p,\lambda_p, \lambda^+-\lambda_p\right) \, ,
\label{eq:cncdir.unidim.marg}
\end{equation}
\begin{equation}
\left(X''_p,X''_q\right) \sim \mbox{\normalfont{CNcDir}}^{\, 2}\left(\alpha_p,\alpha_q,\alpha^+-(\alpha_p+\alpha_q),\lambda_p,\lambda_q,\lambda^+-(\lambda_p+\lambda_q)\right) \, . 
\label{eq:cncdir.bidim.marg}
\end{equation}
\end{proposition}
\begin{proof}
The proof of Eq.~(\ref{eq:cncdir.unidim.marg}) ensues from Eq.~(\ref{eq:cncdir.aggreg.prop}) by setting $m=1$ and taking $A_1=\left\{p\right\}$, $A_2=\left\{1,\, \ldots \, ,p-1,p+1,\, \ldots \, ,D,D+1\right\}$ for every $p=1,\, \ldots \, ,D$, whereas the proof of Eq.~(\ref{eq:cncdir.bidim.marg}) follows from Eq.~(\ref{eq:cncdir.aggreg.prop}) by setting $m=2$ and taking $A_1=\left\{p\right\}$, $A_2=\left\{q\right\}$, $A_3=\left\{1,\, \ldots \, ,D,D+1\right\}-\left\{p,q\right\}$, for every $p,q=1,\, \ldots \, ,D$, $p \neq q$.
\end{proof}

Another consequence of the aggregation property is about the distribution of the sum of the components of a CNcDir random vector.
\begin{proposition}[Sum of the components]
\label{propo:cncdir.comps.sum}
Let $\underline{X}''=\left(X''_1, \, \ldots \, , X''_D\right) \sim \mbox{\normalfont{CNcDir}}^{\, D}\left(\underline{\alpha},\underline{\lambda}\right)$ where $\underline{\alpha}=\left(\alpha_1,\, \ldots \, ,\alpha_D,\alpha_{D+1}\right)$,  $\underline{\lambda}=\left(\lambda_1,\ldots, \right.$ $\left.\lambda_D,\lambda_{D+1}\right)$. Then, $
X''^+=\sum_{j=1}^{D} X''_j \sim \mbox{\normalfont{CDNcB}}\left(\sum_{j=1}^{D} \alpha_j, \alpha_{D+1},\sum_{j=1}^{D} \lambda_j, \lambda_{D+1}\right)$ .
\end{proposition}
\begin{proof}
The proof ensues from Proposition~\ref{propo:cncdir.aggreg.prop} by setting $m=1$ and taking $A_1=\left\{1, \, \ldots \, ,D\right\}$, $A_2=\left\{D+1\right\}$.
\end{proof}

We end this subsection pointing out that, unlike the standard Non-central Dirichlet case, the Conditional Non-central Dirichlet distribution is closed under conditioning after normalization.

\begin{proposition}[Closure under normalized conditioning]
\label{propo:cncdir.normal.condit}
Let $\underline{X}''=\left(X''_1, \, \ldots \, , X''_D\right) \sim \mbox{\normalfont{CNcDir}}^{\, D}\left(\underline{\alpha},\underline{\lambda}\right)$ where $\underline{\alpha}=\left(\alpha_1,\, \ldots \, ,\alpha_{D+1}\right)$, $\underline{\lambda}=\left(\lambda_1,\, \ldots \, ,\right.$ $\left.\lambda_{D+1}\right)$. Let $\underline{X}''_1=\left(X''_1, \, \ldots \, ,X''_k\right)$ and $\underline{X}''_2=\left(X''_{k+1},\, \ldots \, ,X''_{D}\right)$, $1 \leq k \leq D-1$. Then, the normalized conditional distribution of $\underline{X}''_2$ given $\underline{X}''_1=\underline{x}_1$ is:
\begin{eqnarray}
\lefteqn{\left. \frac{\underline{X}''_2}{1-x_1^+} \right | \, \underline{X}''_1=\underline{x}_1=\left(x_1,\, \ldots \, ,x_k\right), \; x_1^+=\sum_{j=1}^{k}x_j \quad \sim}\nonumber \\
& \sim \quad & \mbox{\normalfont{CNcDir}}^{\, D-k}\left(\underline{\alpha}_{2},\underline{\lambda}_{2}\left(1-x_1^+\right)\right) \, , \qquad \left\{\begin{array}{l} \underline{\alpha}_2=\left(\alpha_{k+1},\, \ldots \, ,\alpha_{D+1}\right) \\ \\ \underline{\lambda}_2=\left(\lambda_{k+1},\, \ldots \, ,\lambda_{D+1}\right) \end{array} \right. \, .
\label{eq:cncdir.normal.condit}
\end{eqnarray}
\end{proposition}
\begin{proof}
By setting $m=k$ and taking $A_j=\left\{j\right\}$ for every $j=1,\, \ldots \, ,k$,  $A_{k+1}=\left\{k+1,\, \ldots \, ,\right.$ $\left. D+1\right\}$ in Eq.~(\ref{eq:cncdir.aggreg.prop}), one has $\underline{X}''_1 \sim \mbox{\normalfont{CNcDir}}^{\, k}\left(\underline{\alpha}_1,\alpha_2^+,\underline{\lambda}_1,\lambda_2^+\right)$, where $\underline{\alpha}_1=\left(\alpha_1,\, \ldots \, ,\alpha_k\right)$, $\underline{\lambda}_1=\left(\lambda_1,\, \ldots \, ,\lambda_k\right)$, $\alpha_2^+=\sum_{i=k+1}^{D+1}\alpha_i$ and $\lambda_2^+=\sum_{i=k+1}^{D+1} \lambda_i$. For every $\underline{x}_2=\left(x_{k+1},\, \ldots \, ,\right.$ $\left.x_D\right)$, the conditional density of $\underline{X}''_2$ given $\underline{X}''_1=\underline{x}_1$ admits the following expression:
\begin{eqnarray*}
\lefteqn{f_{\underline{X}''_2 \, | \, \underline{X}''_1= \, \underline{x}_1}\left(\underline{x}_2;\underline{\alpha}_2,\underline{\lambda}_2\right) \quad =} \\
& = & \frac{\Gamma\left(\alpha_2^+\right)}{\prod_{i=k+1}^{D+1}\Gamma\left(\alpha_i\right)} \left[\prod_{i=k+1}^{D} \left(\frac{x_i}{1-x_1^+}\right)^{\alpha_i-1} \right] \left(1-\frac{x_2^+}{1-x_1^+}\right)^{\alpha_{D+1}-1} \\
& \cdot & \frac{\left[\prod_{i=k+1}^{D} \, _0F_1\left(\alpha_i;\frac{\lambda_i}{4}x_i\right) \right] \cdot \,  _0F_1\left[\alpha_{D+1};\frac{\lambda_{D+1}}{4}\left(1-x^+\right)\right]}{\left(1-x_1^+\right)^{D-k} \, _0F_1\left[\alpha_2^+;\frac{\lambda_2^+}{4}\left(1-x_1^+\right)\right]} \; .
\end{eqnarray*}
Now consider the $(D-k)$-dimensional random vector $\underline{Y}$ defined by the following linear transformation:
$$\underline{Y}=h\left(\underline{X}''_2\right)=\frac{\underline{X}''_2}{1-x_1^+} \quad \Leftrightarrow \quad \left\{\begin{array}{l} \underline{X}''_2=h^{-1}\left(\underline{Y}\right)=\left(1-x_1^+\right) \, \underline{Y} \\ \\ |J|=\mbox{detJac} \, h^{-1}\left(\underline{Y}\right)=(1-x_1^+)^{D-k} \end{array} \right. \, ;$$
note that the conditional density of this latter given $\underline{X}''_1=\underline{x}_1$ takes on the form of
\begin{eqnarray}
\lefteqn{f_{\underline{Y}\, |\, \underline{X}''_1= \, \underline{x}_1}\left(\underline{y};\underline{\alpha}_2,\underline{\lambda}_2\left(1-x_1^+\right)\right)=}\nonumber \\
& = &  f_{\underline{X}''_2 \, | \, \underline{X}''_1=\, \underline{x}_1}\left(\left(1-x_1^+\right)\underline{y};\underline{\alpha}_2,\underline{\lambda}_2\right) \cdot |J|=\quad \mbox{\normalfont{Dir}}^{\, D-k}\left(\underline{y};\underline{\alpha}_2\right) \cdot \nonumber \\
& \cdot & \frac{\prod_{i=k+1}^{D} \, _0F_1\left[\alpha_i;\frac{\lambda_i}{4}\left(1-x_1^+\right)y_i\right] \cdot \, _0F_1\left[\alpha_{D+1};\frac{\lambda_{D+1}}{4}\left(1-x_1^+\right)\left(1-y^+\right)\right]}{_0F_1\left[\alpha_2^+;\frac{\lambda_2^+}{4}\left(1-x_1^+\right)\right]} \, , \quad
\label{eq:cncdir.condit.dens.dim1}
\end{eqnarray}
where $y^+=\sum_{i=k+1}^{D}y_i$. Clearly, Eq.~(\ref{eq:cncdir.condit.dens.dim1}) corresponds to the density of the distribution in Eq.~(\ref{eq:cncdir.normal.condit}).
\end{proof}


\subsection{Mixing distribution}
\label{subsec:cncdir.mix.distr}

The $(D+1)$-variate Mixture Weight distribution specified in Eq.~(\ref{eq:multidim.peso.prob}) plays the role of mixing distribution in the mixture type form of the CNcDir$^{\, D}$ density. Such model represents the generalization to $D+1$ dimensions $(D>1)$ of the $\mbox{\normalfont{MW}}^{\, 2}$, this latter being the mixing distribution in the mixture type form of the CDNcB density. Hence, the properties of the $\mbox{\normalfont{MW}}^{\, D+1}$ distribution illustrated in the present subsection are extensions of the ones of the bivariate case, an in-depth study of which is provided in \cite{Ors21}.

First of all, the $\mbox{\normalfont{MW}}^{\, D+1}$ distribution is closed under permutation and the non-centrality parameters of the permuted random vector are obtained by applying the same permutation to the original ones.
\begin{proposition}[Closure under permutation]
\label{propo:mix.distr.permut}
Let $\underline{N}=\left(N_1,\, \ldots \, ,N_{D+1}\right) \sim \mbox{\normalfont{MW}}^{\, D+1}\left(\alpha^+,\underline{\lambda}\right)$ where $\underline{\lambda}=\left(\lambda_1,\, \ldots \, ,\lambda_{D+1}\right)$ and $(\underline{k})=\left(k_{(1)},\, \ldots \, ,k_{(D+1)}\right)$ be a permutation of $\left\{1,\, \ldots \, ,D+1\right\}$. Then, $\underline{N}_{(\underline{k})}=\left(N_{k_{(1)}}, \, \ldots \, , \right.$  $\left. N_{k_{(D+1)}} \right) \sim \mbox{\normalfont{MW}}^{\, D+1}(\alpha^+,\underline{\lambda}_{(\underline{k})})$ where $\underline{\lambda}_{(\underline{k})}=(\lambda_{k_{(1)}},\, \ldots \, ,\lambda_{k_{(D+1)}})$.
\end{proposition}
\begin{proof}
The proof directly follows from the definition of the model at study in Eq.~(\ref{eq:multivar.mw}) by applying the permutation $(\underline{k})$ to the $Y'_i$'s.
\end{proof}
\noindent Proposition~\ref{propo:mix.distr.permut} fully expresses the symmetric nature of the $\mbox{\normalfont{MW}}^{\, D+1}$ distribution, implying that the particular choice of its components does not affect any of the following results.

The distribution under consideration is not closed under marginalization, but it is closed under conditioning.

\begin{proposition}[Marginals and conditionals]
\label{propo:mix.distr.marg.cond}
Let $\underline{N}=\left(N_1,\, \ldots \, ,N_{D+1}\right) \sim \mbox{\normalfont{MW}}^{\, D+1}\left(\alpha^+,\underline{\lambda}\right)$. Then, for every $m=1,\, \ldots \, ,D$, the marginal probability mass function of $\left(N_1,\, \ldots \, ,N_m\right)$ is
\begin{eqnarray*}
\lefteqn{\Pr\left( \, \left(N_1,\, \ldots \, ,N_m\right)=\left(j_1,\, \ldots \, ,j_m\right) \, \right) \quad =} \\
& = & \frac{1}{\left(\alpha^+\right)_{j^+}} \left[\prod_{i=1}^{m}\frac{\left(\frac{\lambda_{j_i}}{4}\right)^{j_i}}{j_i \, !} \right] \, \frac{_0F_1\left(\alpha^++j^+; \frac{\lambda^+-\sum_{i=1}^{m}\lambda_{j_i}}{4}\right)}{_0F_1\left(\alpha^+; \frac{\lambda^+}{4}\right)} \, , \quad \left\{\begin{array}{l}  \left(j_1,\, \ldots \, ,j_m\right) \in \mathbb{N}_0^{\, m} \\ \\  j^+=\sum_{i=1}^{m}j_i \end{array}\right.
\end{eqnarray*}
and the conditional distribution of $\left(N_1,\, \ldots \, ,N_m\right)$ given $\left(N_{m+1},\, \ldots \, ,N_{D+1}\right)$ is
$$
\left. \left(N_1,\, \ldots \, ,N_m\right) \, \right | \, \left(N_{m+1}, \, \ldots \, ,N_{D+1}\right) \; \sim \; \mbox{\normalfont{MW}}^{\, m}\left(\alpha^++\sum_{i=m+1}^{D+1}N_i \, ,\lambda_1,\, \ldots \, ,\lambda_m\right) \, .$$
\end{proposition}
\begin{proof}
The proof is straightforward and is omitted.
\end{proof}

The sum of the components of $\underline{N} \sim \mbox{\normalfont{MW}}^{\, D+1}\left(\alpha^+,\underline{\lambda}\right)$ belongs to the univariate Mixture Weight family.

\begin{proposition}[Sum of the components]
\label{propo:mix.distr.somma}
Let $\underline{N}=\left(N_1,\, \ldots \, ,N_{D+1}\right) \sim \mbox{\normalfont{MW}}^{\, D+1}\left(\alpha^+,\underline{\lambda}\right)$. Then:
\begin{equation}
N^+=\sum_{i=1}^{D+1}N_i \; \sim \; \mbox{\normalfont{MW}}^{\, 1}\left(\alpha^+,\lambda^+\right) \, .
\label{eq:multidim.peso.somma}
\end{equation}
\end{proposition}
\begin{proof}
By employing the probability mass function of the sum of the components of the $\mbox{\normalfont{Multi-Poisson}}^{\, D+1}$ distribution in Eq.~(\ref{eq:multipois.def}), for every $s \in \mathbb{N}_0$ the probabilities of $N^+$ take on the following form:
\begin{eqnarray*}
\lefteqn{\Pr\left(N^+=s\right)=}\\
& = & \sum_{\sum_{i=1}^{D}s_i \leq s} \Pr\left(N_1=s_1,\, \ldots \, ,N_D=s_D,N_{D+1}=s-\sum_{i=1}^{D}s_i\right)=\\
& = & \frac{1}{\left(\alpha^+\right)_{s} \, _0F_1\left(\alpha^+; \frac{\lambda^+}{4}\right)} \sum_{\sum_{i=1}^{D}s_i \leq s}   \left\{\left[\prod_{i=1}^{D}\frac{\left(\frac{\lambda_{i}}{4}\right)^{s_i}}{s_i \, !}\right] \, \frac{\left(\frac{\lambda_{D+1}}{4}\right)^{s-\sum_{i=1}^{D}s_i}}{\left(s-\sum_{i=1}^{D}s_i\right) !}\right\}=\\
& = & \frac{e^{\frac{\lambda^+}{4}}}{\left(\alpha^+\right)_{s} \, _0F_1\left(\alpha^+; \frac{\lambda^+}{4}\right)} \sum_{\sum_{i=1}^{D}s_i \leq s}   \left\{\left[\prod_{i=1}^{D} e^{-\frac{\lambda_{i}}{4}} \, \frac{\left(\frac{\lambda_{i}}{4}\right)^{s_i}}{s_i \, !}\right]  \frac{e^{-\frac{\lambda_{D+1}}{4}}\left(\frac{\lambda_{D+1}}{4}\right)^{s-\sum_{i=1}^{D}s_i}}{\left(s-\sum_{i=1}^{D}s_i\right) !}\right\}\\
& = & \frac{1}{_0F_1\left(\alpha^+; \frac{\lambda^+}{4}\right)} \frac{1}{\left(\alpha^+\right)_{s}} \frac{\left(\frac{\lambda^+}{4}\right)^s}{s \, !}  \qquad s \in \mathbb{N}_0 \, ,
\end{eqnarray*}
which corresponds to the special case of Eq.~(\ref{eq:multidim.peso.prob}) where $D+1$ is set equal to 1.
\end{proof}

Despite the differences existing between their joint probability mass functions, the $\mbox{\normalfont{MW}}^{\, D+1}$ and the $\mbox{\normalfont{Multi-Poisson}}^{\, D+1}$ share the same conditional distribution given the sum of their components.

\begin{proposition}[Conditional distribution given the sum of the components]
\label{propo:mix.distr.cond.sum}
Let $\underline{N}=\left(N_1,\, \ldots \, ,N_{D+1}\right) \sim \mbox{\normalfont{MW}}^{\, D+1}\left(\alpha^+,\underline{\lambda}\right)$ where $\underline{\lambda}=\left(\lambda_1, \, \ldots \, , \lambda_{D+1}\right)$ and $N^+=\sum_{i=1}^{D+1}N_i$. Then:
\begin{equation}
\left. \left(N_1,\, \ldots \, ,N_{D}\right) \, \right| \, N^+ \sim \mbox{\normalfont{Multinomial}}^{\, D}\left(N^+,\frac{\lambda_1}{\lambda^+},\, \ldots \, ,\frac{\lambda_D}{\lambda^+}\right) \, .
\label{eq:multidim.peso.cond.sum}
\end{equation}
\end{proposition}
\begin{proof}
Let $\left(j_1,\, \ldots \, ,j_D\right) \in \mathbb{N}_0^{\, D}$, $s \in \mathbb{N}_0$ and $\sum_{i=1}^{D}j_i=j^+ \leq s$; then, the proof ensues from Eqs.~(\ref{eq:multidim.peso.prob}) and~(\ref{eq:multidim.peso.somma}) by carrying out the following computations:
\begin{eqnarray*}
\lefteqn{\Pr\left( \, \left.\left(N_1,\, \ldots \, ,N_{D}\right)=\left(j_1,\, \ldots \, ,j_D\right) \, \right | \, N^+=s \right)=}\\
& = & \frac{\Pr\left(N_1=j_1,\, \ldots \, ,N_D=j_D,N_{D+1}=s-j^+\right)}{\Pr\left(N^+=s\right)}=\\
& = & \frac{1}{\left(\alpha^+\right)_{s} \, _0F_1\left(\alpha^+; \frac{\lambda^+}{4}\right)}    \left[\prod_{i=1}^{D}\frac{\left(\frac{\lambda_{i}}{4}\right)^{j_i}}{j_i \, !}\right] \, \frac{\left(\frac{\lambda_{D+1}}{4}\right)^{s-j^+}}{\left(s-j^+\right) !} \left[\frac{1}{_0F_1\left(\alpha^+; \frac{\lambda^+}{4}\right)} \frac{1}{\left(\alpha^+\right)_{s}} \frac{\left(\frac{\lambda^+}{4}\right)^s}{s \, !}\right]^{-1}\\
& = & {s \choose j_1,\, \ldots \, ,j_D} \left[\prod_{i=1}^{D}\left(\frac{\lambda_{i}}{\lambda^+}\right)^{j_i}\right] \left(1-\sum_{i=1}^{D}\frac{\lambda_{i}}{\lambda^+}\right)^{s-j^+} \, ,
\end{eqnarray*}
this latter being the probability mass function of the distribution in Eq.~(\ref{eq:multidim.peso.cond.sum}).
\end{proof}

Finally, the $\mbox{\normalfont{MW}}^{\, D+1}$ distribution can be simulated by first generating the random variable $N^+$ according to Eq.~(\ref{eq:multidim.peso.somma}) and then the $D$-variate Multinomial distribution in Eq.~(\ref{eq:multidim.peso.cond.sum}). In particular, the former can be easily simulated by using, for example, the Inverse-Transform method \cite{LawKel00}. Specifically, after generating a realization $u$ from a Uniform random variable on the Real interval $(0,1)$, an exact realization $s \in \mathbb{N}_0$ from the random variable $N^+$ is obtainable by taking the generalized inverse of its distribution function $F$ in $u$, i.e. $F^-\left(u\right)=\inf \{x \, | \, F\left(x\right) \geq u\}$.

\subsection{Representations}
\label{subsec:cncdir.repres}

In the present subsection a number of fundamental representations of the Conditional Non-central Dirichlet model are provided, disclosing its remarkable tractability.

The first one is the mixture representation, which summarizes the arguments leading to Definitions~\ref{def:cncdir.def} and~\ref{def:multivar.mw}.

\begin{proposition}[Mixture representation]
\label{propo:cncdir.mixt.repres}
Let $\underline{X}'' \sim$ $\mbox{\normalfont{CNcDir}}^{\, D}\left(\underline{\alpha},\underline{\lambda}\right)$ and $\underline{N} \sim \mbox{\normalfont{MW}}^{\, D+1}\left(\alpha^+,\underline{\lambda}\right)$. Then:
\begin{equation}
\left. \underline{X}'' \, \right| \, \underline{N} \sim \mbox{\normalfont{Dir}}^{\, D}\left(\underline{\alpha}+\underline{N}\right) \; .
\label{eq:cncdir.mixt.repres}
\end{equation}
\end{proposition}
\begin{proof}
The proof is straightforward and is omitted.
\end{proof}

\noindent Proposition~\ref{propo:cncdir.mixt.repres} proves significant in several respects. Indeed, the above representation constitutes a valid practical method for generating from a Conditional Non-central Dirichlet random vector. Specifically, the algorithm to obtain a realization from the $\mbox{\normalfont{CNcDir}}^{\, D}$ model using such a representation requires to first generate from the $\mbox{MW}^{\, D+1}$ distribution and then from the conditional distribution in Eq.~(\ref{eq:cncdir.mixt.repres}). In particular, the former operation can be accomplished by following the lines made explicit in the previous subsection. In this regard, an interesting restatement of the combination of the two aforementioned generation steps is provided by the following result, which shows that, conditionally on $N^+$, the density of the model of interest is given by a mixture of Dirichlet densities weighted by the Multinomial probabilities in Eq.~(\ref{eq:multidim.peso.cond.sum}).  

\begin{proposition}[Conditional density given $N^+$]
\label{propo:cncdir.cond.distr.mplus}
Let $\underline{X}'' \sim \mbox{\normalfont{CNcDir}}^{\, D}\left(\underline{\alpha},\underline{\lambda}\right)$, $\underline{N} \sim \mbox{\normalfont{MW}}^{\, D+1}\left(\alpha^+,\underline{\lambda}\right)$ and $N^+=\sum_{i=1}^{D+1}N_i$. Then, the conditional density of $\underline{X}''$ given $N^+$ is
\begin{eqnarray}
\lefteqn{f_{\left.\underline{X}'' \, \right| \, N^+}\left(\underline{x};\underline{\alpha}, \underline{\lambda}\right)=\sum_{j^+ \, \leq \, N^+}\left[\mbox{\normalfont{Multinomial}}^{\, D}\left(j_1,\, \ldots \, ,j_D;\, N^+,\frac{\lambda_1}{\lambda^+},\, \ldots \, ,\frac{\lambda_D}{\lambda^+}\right) \cdot \right.}\nonumber \\
& \cdot & \left. \mbox{\normalfont{Dir}}^{\, D}\left(\underline{x};\alpha_1+j_1,\, \ldots \, ,\alpha_D+j_D,\alpha_{D+1}+N^+-j^+\right)\right] \qquad \left\{\begin{array}{l} \left(j_1,\, \ldots \, ,j_D\right) \in \mathbb{N}_0^{\, D} \\ \\ j^+=\sum_{i=1}^{D}j_i \end{array} \right.
\label{eq:cncdir.cond.distr.nplus}
\end{eqnarray}
\end{proposition}
\begin{proof}
The proof ensues from the statement 
\begin{eqnarray*}
f_{\left.\underline{X}'' \, \right| \, N^+}\left(\underline{x};\underline{\alpha}, \underline{\lambda}\right) & = & \sum_{j^+ \, \leq \, N^+} \left[\, \Pr\left(\left. \, \left(N_1, \, \ldots \, ,N_D\right)=\left(j_1,\, \ldots \, ,j_D\right) \, \right| \, N^+\right) \right. \cdot\\
& \cdot & \left. f_{\left.\underline{X}'' \, \right| \, \left(\, N_1, \, \ldots \, ,N_D, N^+\right)}\left(\underline{x};\underline{\alpha}, \underline{\lambda}, \left(j_1, \, \ldots \, ,j_D\right)\right) \, \right] \, ,
\end{eqnarray*}
where $\left(j_1,\, \ldots \, ,j_D\right) \in \mathbb{N}_0^{\, D}$, $j^+=\sum_{i=1}^{D}j_i$, by using Eq.~(\ref{eq:multidim.peso.cond.sum}) and noting that Eq.~(\ref{eq:cncdir.mixt.repres}) can be rearranged as follows:
$$\left. \underline{X}'' \, \right| \, \left(\, N_1, \, \ldots \, ,N_D,N^+\right) \sim \mbox{\normalfont{Dir}}^{\, D}\left(\alpha_1+N_1,\, \ldots \, ,\alpha_D+N_D,\alpha_{D+1}+N^+-\sum_{i=1}^{D}N_i\right) \; .$$
\end{proof}

Finally, the mixture representation in Eq.~(\ref{eq:cncdir.mixt.repres}) leads to the following representation of the CNcDir model in terms of unconditional composition.

\begin{proposition}[Unconditional composition type representation]
\label{propo:cncdir.uncond.compos.repres}
Let $\underline{N} \sim \mbox{\normalfont{MW}}^{\, D+1}\left(\alpha^+,\underline{\lambda}\right)$ and $\left. Z'_i \, \right| \, \underline{N} \stackrel{\bot}{\sim } \chi^{\,2}_{2\alpha_i+2N_i}$, $i=1,\ldots,D+1$ with $Z'^+=\sum_{i=1}^{D+1} Z'_i$. Then:
\begin{itemize}
\item[i) ] $\left(\frac{Z'_1}{Z'^+},\, \ldots \, ,\frac{Z'_D}{Z'^+}\right) \sim \mbox{\normalfont{CNcDir}}^{\, D}\left(\underline{\alpha},\underline{\lambda}\right)$ ,
\item[ii) ] $Z'_i=Z_i+\sum_{j=1}^{N_i}F_j$ , $i=1,\, \ldots \, ,D+1$, where $Z_i \sim \chi^{\, 2}_{2\alpha_i}$, $\underline{N}$, $\{F_j \stackrel{\bot}{\sim } \chi^{\, 2}_2 \}$ are mutually independent.
\end{itemize}
\end{proposition}
\begin{proof}
The proof is straightforward and is omitted. 
\end{proof}

\noindent The issue of generating from the Conditional Non-central Dirichlet model can be also addressed by resorting to \textit{i)} in Proposition~\ref{propo:cncdir.uncond.compos.repres}. Specifically, the simulation methods based on the mixture representation and on the above unconditional composition type representation differ only in the way that the Dirichlet distribution is generated: directly by the former and via Chi-Squareds by the latter. Moreover, from \textit{ii)} in Proposition~\ref{propo:cncdir.uncond.compos.repres} it is noticeable that the $Z'_i$'s play the same role as the Non-central Chi-Squared random variables involved in the definition of the standard Non-central Dirichlet model. Hence, the present representation guides us to view the CNcDir model as an analogue of the NcDir one. Indeed, like the $Y'_i$'s in  Eq.~(\ref{eq:sumrepres.ncchisq}), the $Z'_i$'s can be additively decomposed into two independent parts, a central one and a purely non-central one. Moreover, by resorting to the $Z'_i$'s, the following generalization of the independence relationship stated in Property~\ref{prope:char.prop.chisq} holds true.

\begin{proposition}[Conditional independence]
\label{propo:cncdir.condit.indep}
Let $\underline{X}'' \sim \mbox{\normalfont{CNcDir}}^{\, D}\left(\underline{\alpha},\underline{\lambda}\right)$ and $\underline{N}=\left(N_1, \, \ldots \, ,N_{D+1}\right) \sim \mbox{\normalfont{MW}}^{\, D+1}\left(\alpha^+,\underline{\lambda}\right)$ with $N^+=\sum_{i=1}^{D+1}N_i$. Furthermore, let $\left(Z'_1, \, \ldots \, ,Z'_{D+1}\right)$ be defined as in Proposition~\ref{propo:cncdir.uncond.compos.repres} with $Z'^+=\sum_{i=1}^{D+1} Z'_i$. Then, $\underline{X}''$ and $Z'^+$ are conditionally independent given $N^+$.
\end{proposition}
\begin{proof}
The ingredients of the representation of the CNcDir$^{\, D}$ model in Proposition~\ref{propo:cncdir.uncond.compos.repres} can be analogously defined as follows:
$$Z'_i \left. \, \right| \, \left(N_1, \, \ldots \, , N_D,N^+\right) \; \stackrel{\bot}{\sim } \; \left\{\begin{array}{ll} \chi^2_{2 \alpha_i+2 N_i} &  , \; i=1, \, \ldots \, , D \\ \\ \chi^2_{2 \alpha_{D+1}+2 \left(N^+-\sum_{r=1}^{D}N_r\right)} & , \; i=D+1 \end{array}\right. \; ;$$
hence, by Eq.~(\ref{eq:ncchisq.reprod}):
\begin{equation}
\left.Z'^+ \, \right| \, \left(N_1, \, \ldots \, , N_D,N^+\right) \stackrel{d}{=} \left. Z'^+ \, \right| \, N^+ \sim \chi^2_{2\alpha^++2N^+} \, ,
\label{eq:zplus.cond.chisq}
\end{equation}
where $\stackrel{d}{=}$ denotes equality in distribution. By Property~\ref{prope:char.prop.chisq}, $\underline{X}''$ and $Z'^+$ are independent conditionally on $\left(N_1, \, \ldots \, , N_D,N^+\right)$ and therefore, given this latter, the joint distribution of $\left(\underline{X}'',Z'^+\right)$ factores into the marginal distributions of $\underline{X}''$ and $Z'^+$. That said, the proof follows by noting that the joint density of $\left. \left(\underline{X}'',Z'^+\right) \, \right| \, N^+$ factores into the marginal densities of $\left. \underline{X}'' \, \right| \, N^+$ and $\left. Z'^+ \, \right| \, N^+$; indeed, under Eq.~(\ref{eq:zplus.cond.chisq}):
\begin{eqnarray*}
f_{\left.\left(\underline{X}'',\, Z'^+\right)\, \right| \, N^+}\left(\underline{x},z\right) & = & \sum_{j^+ \leq N^+} \left[\, \Pr\left(\left. \, \left(N_1, \, \ldots \, ,N_D\right)=\left(j_1,\, \ldots \, ,j_D\right) \, \right| \, N^+\right) \cdot \right. \\
& \cdot & \left. f_{\left.\left(\underline{X}'',\, Z'^+\right) \, \right| \, \left(\, N_1, \, \ldots \, ,N_D,N^+ \, \right)}\left(\underline{x},z; \, \left(j_1, \, \ldots \, ,j_D\right) \, \right) \, \right]=\\
& = & \sum_{j^+ \leq N^+} \left[\, \Pr\left(\left. \, \left(N_1, \, \ldots \, ,N_D\right)=\left(j_1,\, \ldots \, ,j_D\right) \, \right| \, N^+\right)\cdot \right. \\
& \cdot & \left.  f_{\left. \underline{X}'' \, \right| \, \left(N_1, \, \ldots \, ,N_D,N^+\right)}\left(\underline{x},\left(j_1, \, \ldots \, ,j_D\right) \, \right) \, \right]  \cdot f_{\left. Z'^+ \, \right| \, \left(N_1, \, \ldots \, ,N_D,N^+\right)}\left(z\right)\\
& = & f_{\left.\underline{X}'' \, \right| \, N^+}\left(\underline{x}\right) \cdot f_{\left.Z'^+ \, \right| \, N^+}\left(z\right) \, ,
\end{eqnarray*}
where $\left(j_1,\, \ldots \, ,j_D\right) \in \mathbb{N}_0^{\, D}$, $j^+=\sum_{i=1}^{D}j_i$ and the density $f_{\left. X'' \, \right| \, N^+}$ is given in Eq.~(\ref{eq:cncdir.cond.distr.nplus}).
\end{proof}

\subsection{Density plots}
\label{subsec:cncdir.dens.plots}

In this subsection we shall focus on the case $D=2$ in order to make it possible the graphical depiction of the joint density of the model of interest.

A key advantage of the $\mbox{\normalfont{CNcDir}}^{\, 2}$ distribution over the bivariate Dirichlet is the much larger variety of shapes reachable by the density of the former on varying the non-centrality parameters. In this regard, some significant plots of the bivariate $\mbox{\normalfont{CNcDir}}$ density are displayed in Figure~\ref{fig:GrafCNc_o} for selected values of the parameter vector.

\begin{figure}[ht]
 \centering
 \subfigure
   {\includegraphics[width=6.3cm]{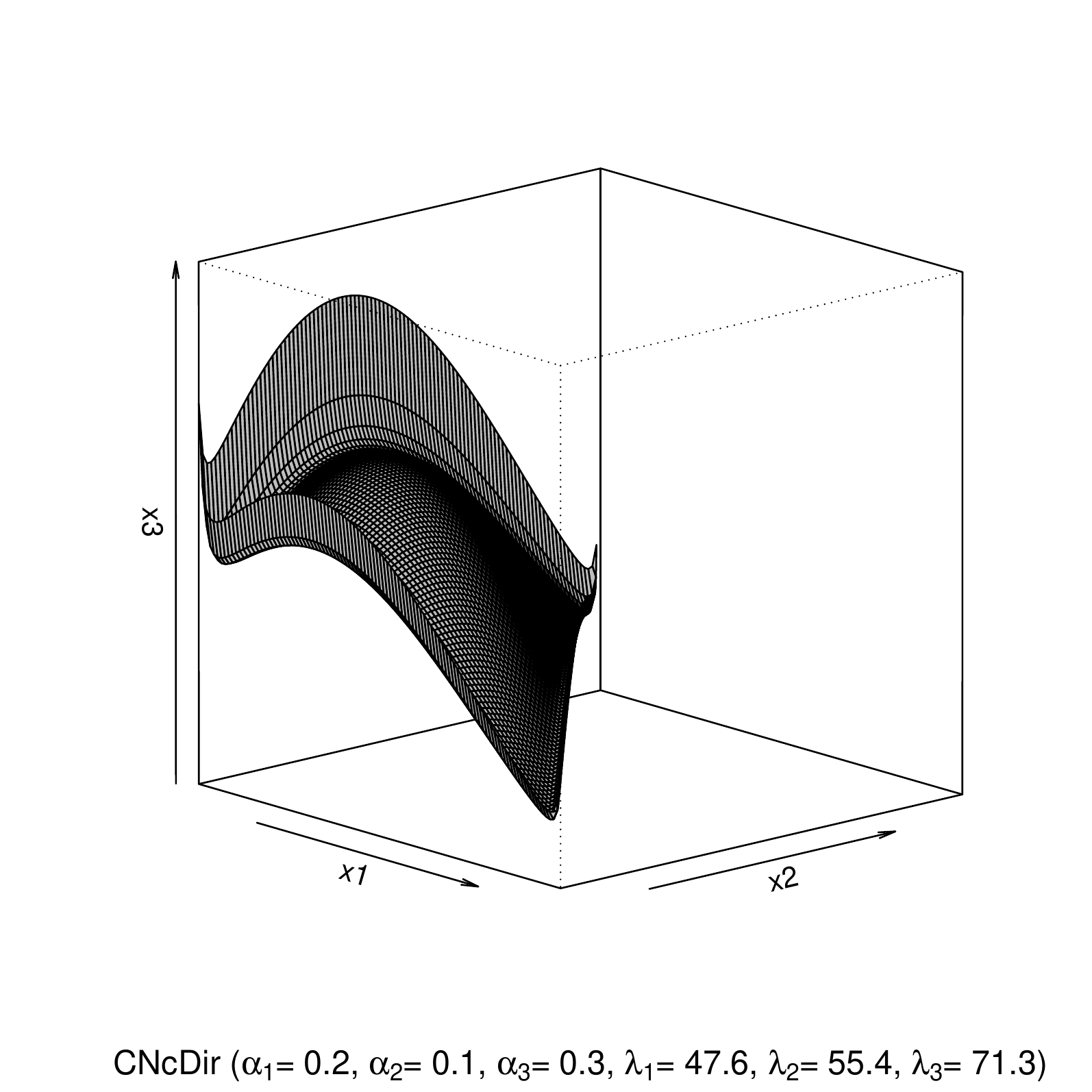}}
 \hspace{5mm}
 \subfigure
   {\includegraphics[width=6.3cm]{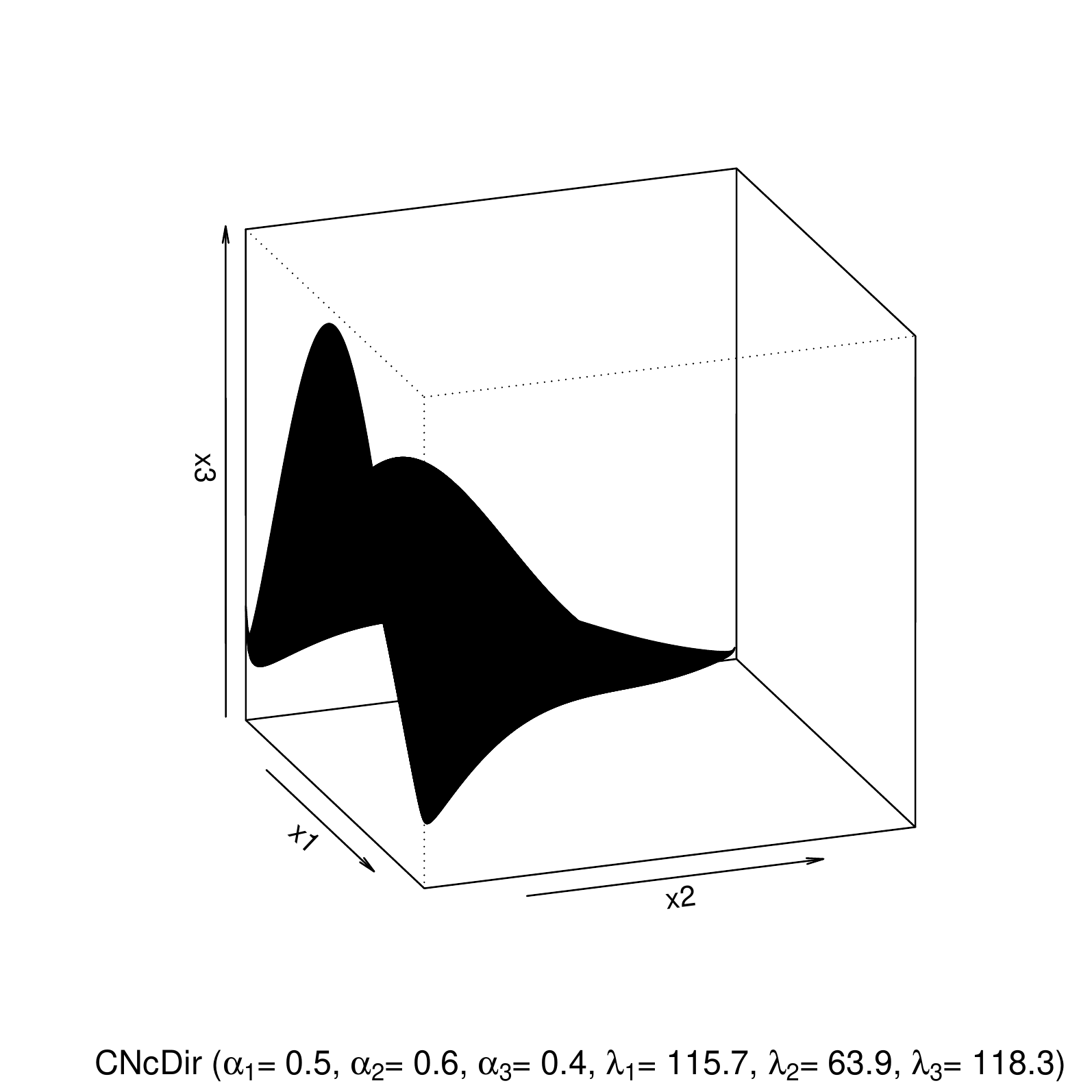}}
 \vspace{2mm}
 \subfigure
   {\includegraphics[width=6.3cm]{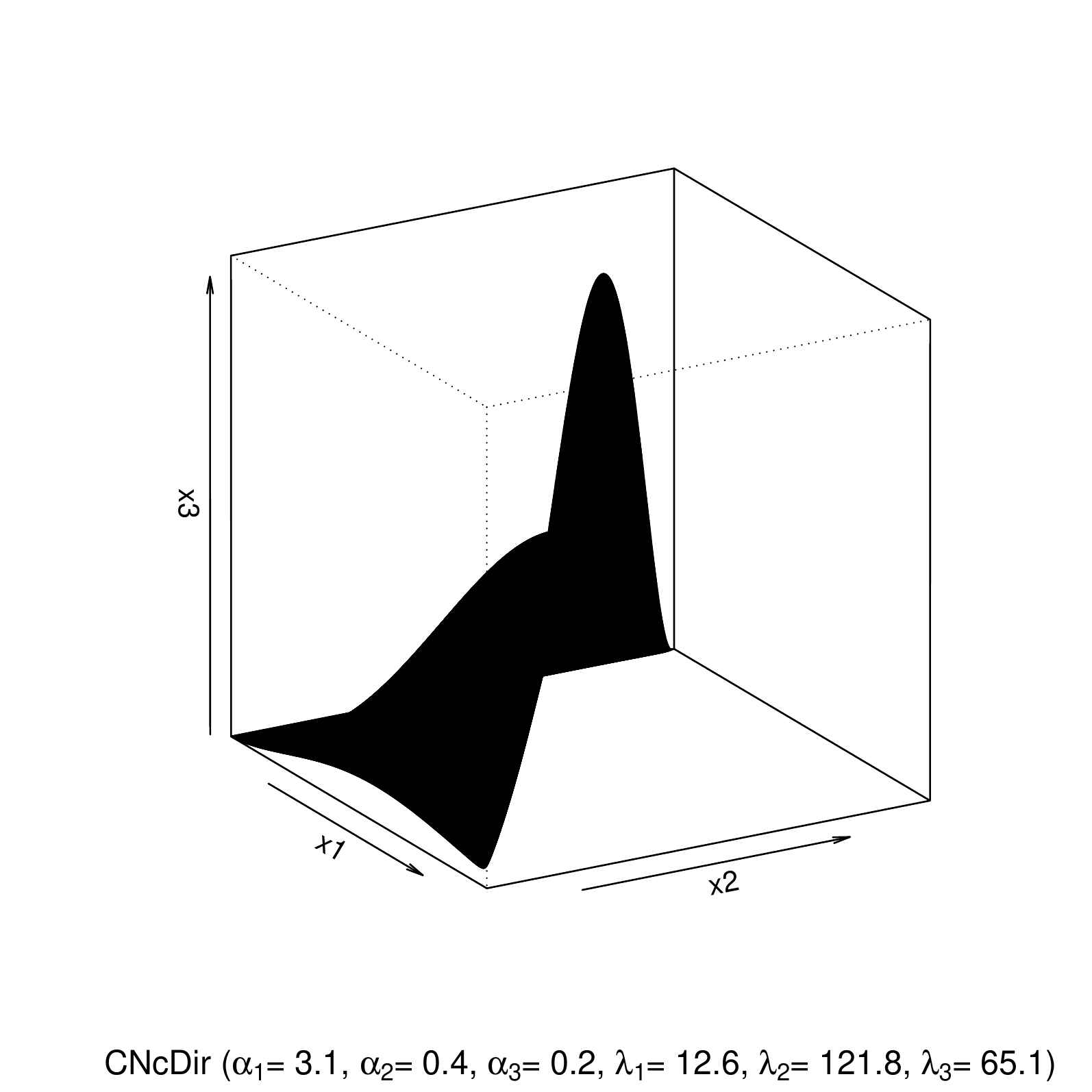}}
 \hspace{5mm}
 \subfigure
   {\includegraphics[width=6.3cm]{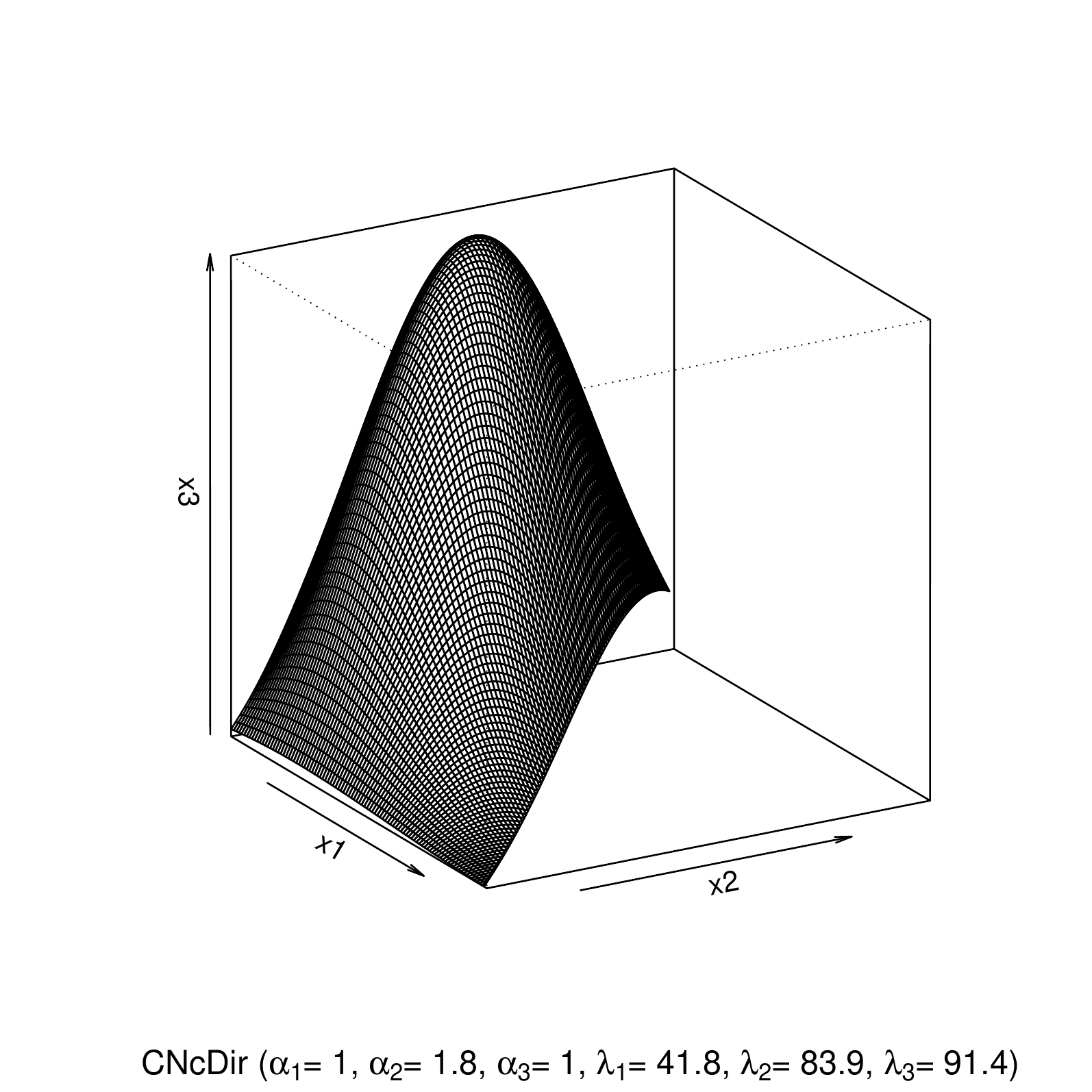}}
 \caption{Plots of the density of $\left(X''_1,X''_2\right) \sim \mbox{\normalfont{CNcDir}}^{\, 2}\left(\alpha_1,\alpha_2,\alpha_3,\lambda_1,\lambda_2,\lambda_3\right)$ for selected values of the parameter vector.}
 \label{fig:GrafCNc_o}
\end{figure}

\noindent Specifically, the main attraction of the CNcDir$^{\, 2}$ distribution lies in the following feature. When at least two shape parameters are set equal to one, the three non-centrality parameters control the height of the density at the vertices of the unit simplex by allowing the limits of this latter to have arbitrary positive and finite values. That said, the limiting values of the bivariate Conditional Non-central Dirichlet density are established herein.

\begin{proposition}[Limiting values of the bivariate density]
\label{propo:cncdir.dens.lims}
Let $\left(X''_1,X''_2\right) \sim \mbox{\normalfont{CNcDir}}^{\, 2}\left(\alpha_1,\alpha_2,\alpha_3,\lambda_1,\lambda_2,\lambda_3\right)$. Then:
\begin{eqnarray*}
\lefteqn{\mbox{\normalfont{CNcDir}}^{\, 2}\left(x_1,x_2;\alpha_1,\alpha_2,\alpha_3,\lambda_1,\lambda_2,\lambda_3\right) \quad \rightarrow}\\
& \stackrel{\left(x_1,\, x_2\right) \, \rightarrow \, \left(1,0\right)}{\rightarrow} & \left\{\begin{array}{ll}
\mbox{\normalfont{if} }\, \alpha_2+\alpha_3<2 \, : & + \, \infty\\
\mbox{\normalfont{if} }\, \alpha_2+\alpha_3=2 \, : & \left\{ \begin{array}{ll} \mbox{\normalfont{if} }\, \alpha_2=\alpha_3=1 \, : & \alpha_1\left(\alpha_1+1\right) \frac{_0F_1\left(\alpha_1;\frac{\lambda_1}{4}\right)}{_0F_1\left(\alpha_1+2;\frac{\lambda^+}{4}\right)} \\ \mbox{\normalfont{otherwise:} }\, & \not\exists \end{array}\right.\\
\mbox{\normalfont{if} }\, \alpha_2+\alpha_3>2 \, : & 0 
\end{array}\right. \\
& \stackrel{\left(x_1,\, x_2\right) \, \rightarrow \, \left(0,1\right)}{\rightarrow} & \left\{\begin{array}{ll}
\mbox{\normalfont{if} }\, \alpha_1+\alpha_3<2 \, : & + \, \infty\\
\mbox{\normalfont{if} }\, \alpha_1+\alpha_3=2 \, : & \left\{ \begin{array}{ll} \mbox{\normalfont{if} }\, \alpha_1=\alpha_3=1 \, : & \alpha_2\left(\alpha_2+1\right) \frac{_0F_1\left(\alpha_2;\frac{\lambda_2}{4}\right)}{_0F_1\left(\alpha_2+2;\frac{\lambda^+}{4}\right)} \\ \mbox{\normalfont{otherwise:} }\, & \not\exists \end{array}\right.\\
\mbox{\normalfont{if} }\, \alpha_1+\alpha_3>2 \, : & 0 
\end{array}\right. \\
& \stackrel{\left(x_1,\, x_2\right) \, \rightarrow \, \left(0,0\right)}{\rightarrow} & \left\{\begin{array}{ll}
\mbox{\normalfont{if} }\, \alpha_1+\alpha_2<2 \, : & + \, \infty\\
\mbox{\normalfont{if} }\, \alpha_1+\alpha_2=2 \, : & \left\{ \begin{array}{ll} \mbox{\normalfont{if} }\, \alpha_1=\alpha_2=1 \, : & \alpha_3\left(\alpha_3+1\right) \frac{_0F_1\left(\alpha_3;\frac{\lambda_3}{4}\right)}{_0F_1\left(\alpha_3+2;\frac{\lambda^+}{4}\right)} \\ \mbox{\normalfont{otherwise:} }\, & \not\exists \end{array}\right.\\
\mbox{\normalfont{if} }\, \alpha_1+\alpha_2>2 \, : & 0 
\end{array}\right. \\
\end{eqnarray*}
where the symbol $\not\exists \, $ indicates that the corresponding limit does not exist because depends on the direction followed to reach the accumulation point under consideration.
\end{proposition}
\begin{proof}
The proof ensues from Eqs.~(\ref{eq:dir.dens.lims}) and~(\ref{eq:cncdir.perturb.dens}) by observing that, in the notation of Eq.~(\ref{eq:f01}), $_0F_1\left(b;0\right)=1$.
\end{proof}
\noindent The surface representing the density function of the CNcDir$^{\, 2}$ distribution is depicted in Figure~\ref{fig:GrafCNc} for shape parameters equal one and selected values of the non-centrality parameters. The above mentioned feature of the density at study can be clearly detected from these plots.     

\begin{figure}[ht]
 \centering
 \subfigure
   {\includegraphics[width=6.3cm]{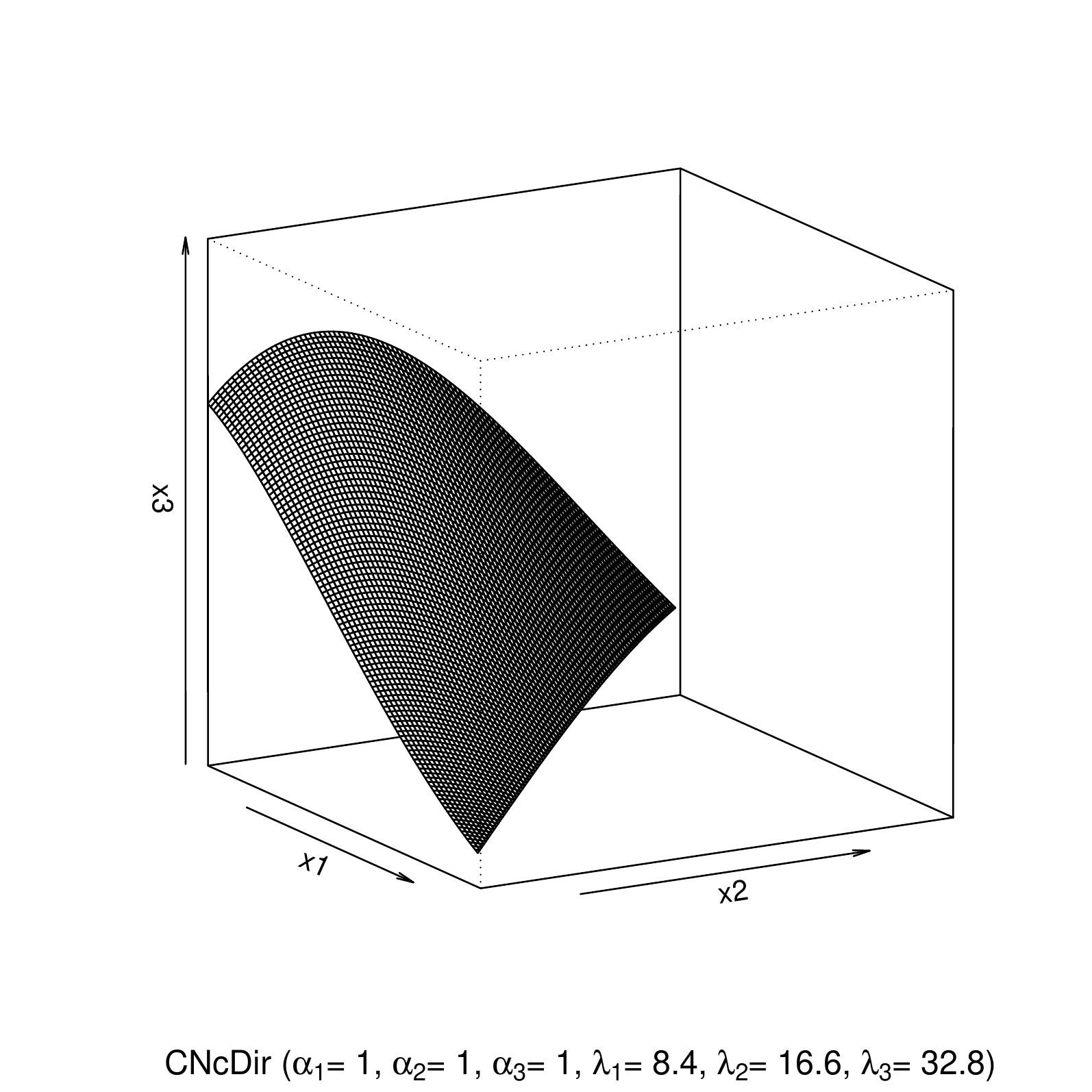}}
 \hspace{5mm}
 \subfigure
   {\includegraphics[width=6.3cm]{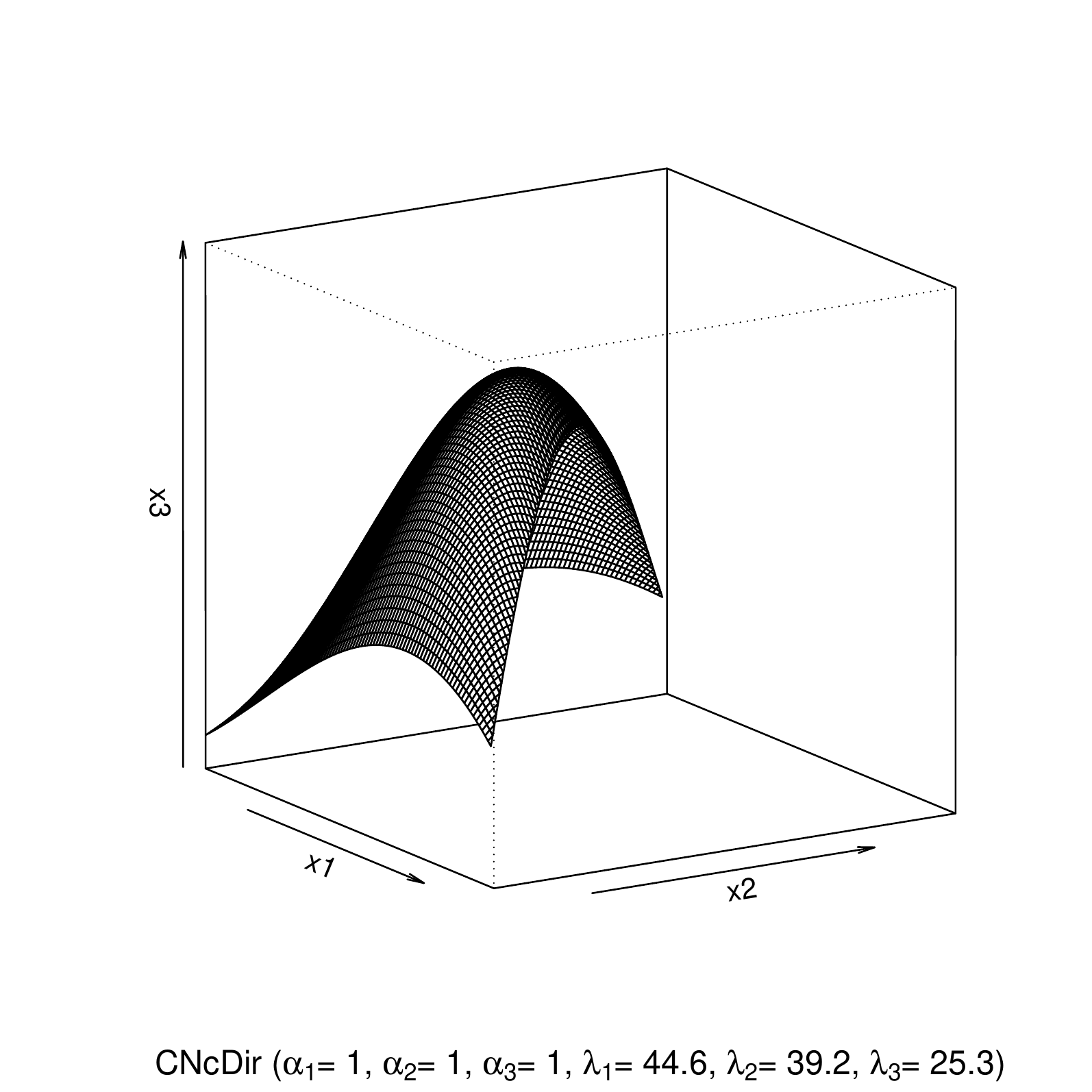}}
 \vspace{2mm}
 \subfigure
   {\includegraphics[width=6.3cm]{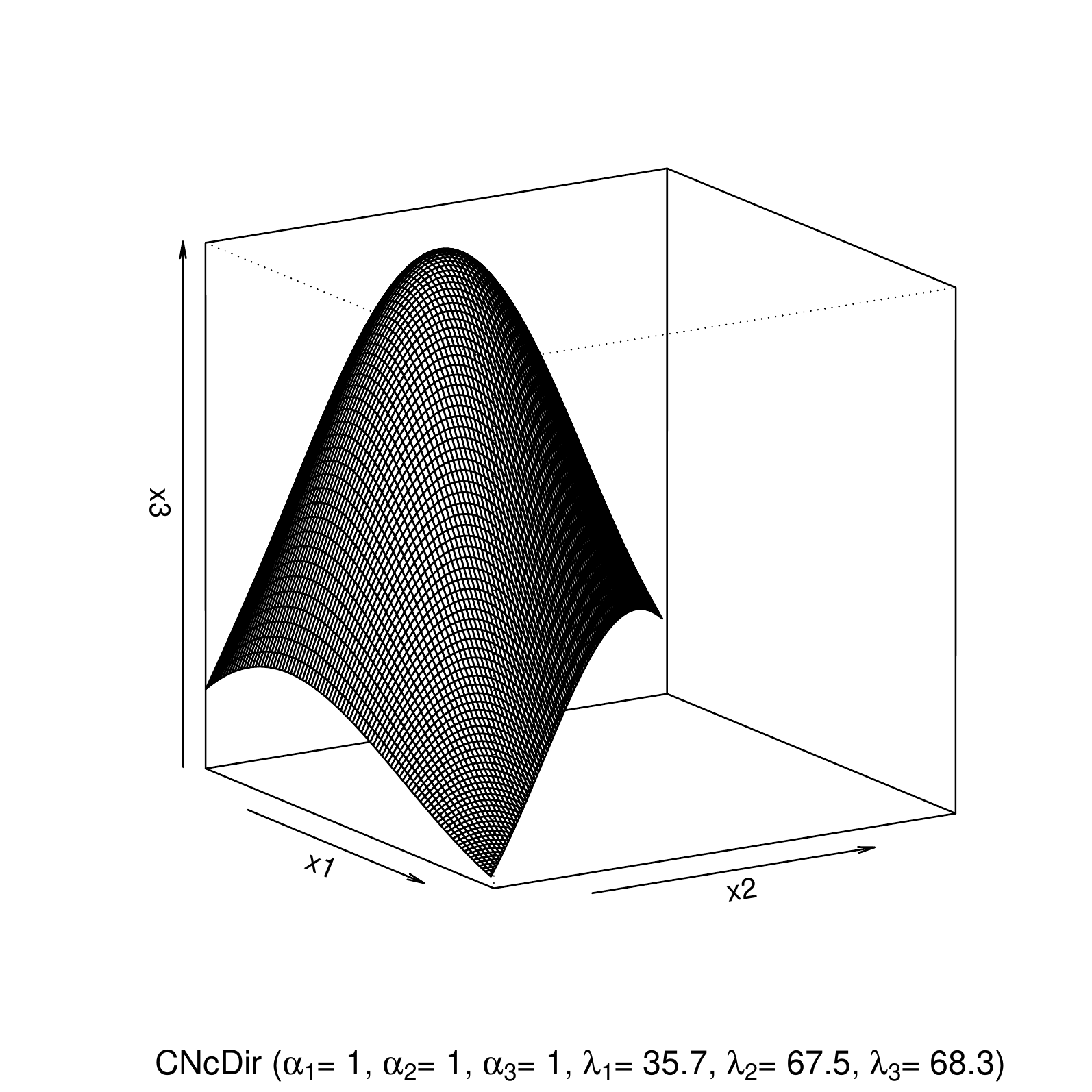}}
 \hspace{5mm}
 \subfigure
   {\includegraphics[width=6.3cm]{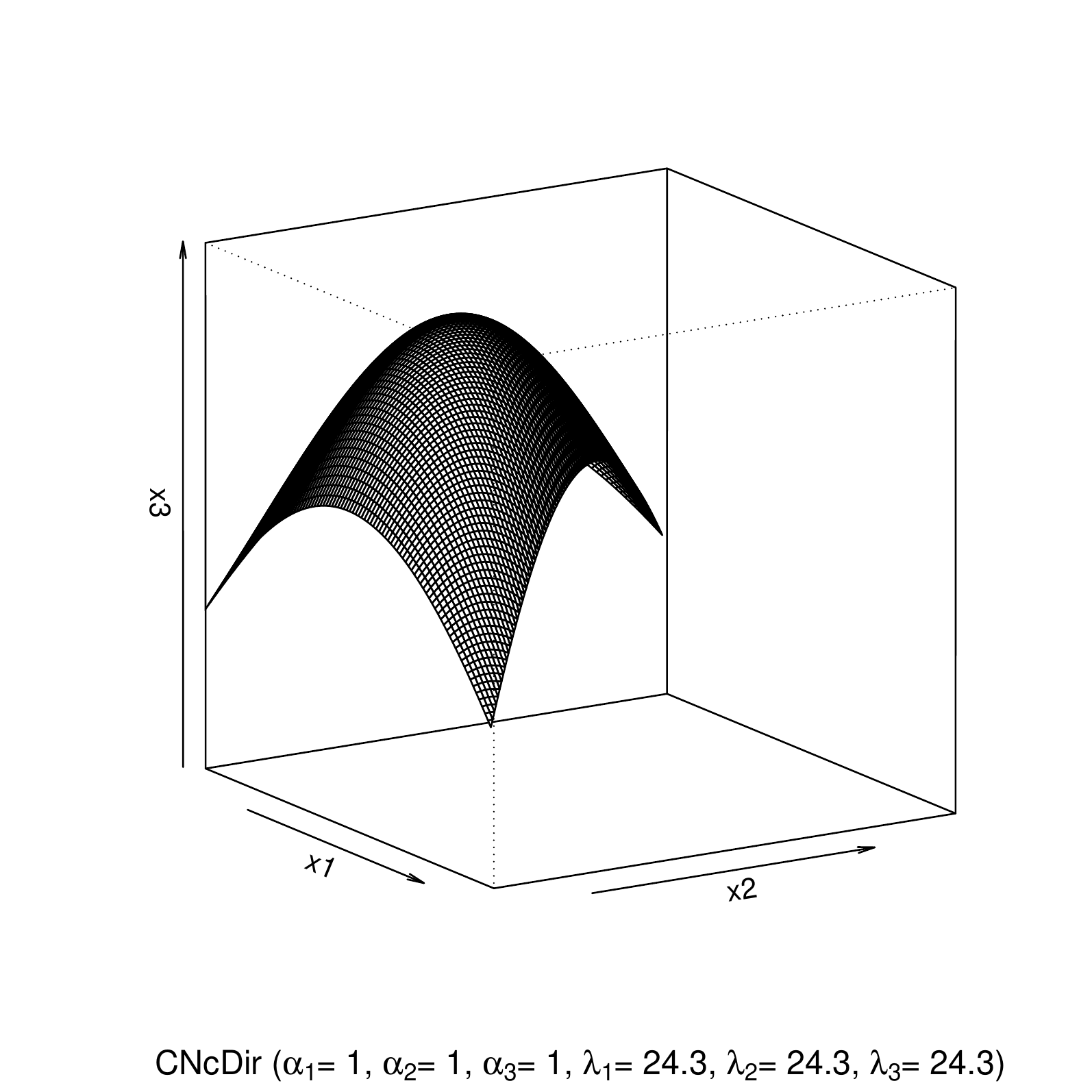}}
 \caption{Plots of the density of $\left(X''_1,X''_2\right) \sim \mbox{\normalfont{CNcDir}}^{\, 2}\left(\alpha_1,\alpha_2,\alpha_3,\lambda_1,\lambda_2,\lambda_3\right)$ for $\alpha_i=1$, $i=1,2,3$ and selected values of $\lambda_1$, $\lambda_2$, $\lambda_3$.}
 \label{fig:GrafCNc}
\end{figure}

All the elements that enable us to affirm that the CNcDir distribution preserves the applicative potential of the NcDir one are finally at hand. Indeed, there exist deep similarities between the varieties of shapes reachable by the two densities in the case that, under the same values for the shape parameters, the non-centrality parameters of the Conditional model are assigned much higher values than those used for the existing one. In light of the mixture representations of the two models, in fact, the differences between such distributions can be entirely explained by the differences existing between their respective mixing distributions. In light of Eq.~(\ref{eq:multidim.peso.cond.sum}), these latter are completely due to the differences between the sums of their components $N^+$ and $M^+$, the probability mass functions of which depend on $\underline{\lambda}$ by means of $\lambda^+$. By virtue of the arguments put forward in \cite{Ors21} (Section 4) with reference to the comparison between the Conditional and the standard Doubly Non-central Beta distributions, the link between the $\lambda^+$ values of $N^+$ and $M^+$ can be understood by comparing their modes as functions of $\lambda^+$. Specifically, for a given value of the mode of $M^+$, one can find the value $\lambda^+$ relative to $N^+$ which produces similar mode; as shown in the above cited reference, this $\lambda^+$ value is much larger than the one relative to $M^+$ and gets even larger when $\alpha^+$ increases.

\subsection{Moments}
\label{subsec:cncdir.mom}

For the sake of simplicity, in all the present subsection we shall focus on the case $D=2$; however, by virtue of Eq.~(\ref{eq:cncdir.bidim.marg}), the results derived herein also apply for any two-dimensional marginal of the CNcDir$^{\, D}$ model with $D>2$.

By analogy with the form of the density in Eq.~(\ref{eq:cncdir.dens}) and by Eq.~(\ref{eq:dir.mixr1r2mom}), for every $r_1, \, r_2 \in \mathbb{N}$, the mixed raw moment of order $(r_1, r_2)$ of $\left(X''_1,X''_2\right) \sim \mbox{\normalfont{CNcDir}}^{\, 2}\left(\alpha_1,\alpha_2,\alpha_3,\lambda_1,\lambda_2,\lambda_3\right)$ can be stated as
\begin{eqnarray}
\lefteqn{\mathbb{E}\left[\left(X_1''\right)^{ \, r_1} \left(X_2''\right)^{ \, r_2} \, \right]=  \qquad \qquad \qquad \qquad \qquad \qquad \left\{\begin{array}{l} r^+=r_1+r_2 \\ j^+=j_1+j_2+j_3 \end{array}\right. }\nonumber \\
& = &  \sum_{j_1, \, j_2, \, j_3 \, = \, 0}^{+\infty} \left\{ \, \Pr\left[ \; \left(N_1,N_2,N_3\right)=\left(j_1,j_2,j_3\right) \; \right] \cdot \frac{\left(\alpha_1+j_1\right)_{r_1} \, \left(\alpha_2+j_2\right)_{r_2}}{\left(\alpha^+ +j^+\right)_{r^+}} \, \right\}  \, ,
\label{eq:cncdir.mixr1r2mom.def}
\end{eqnarray}
i.e. as the multiple infinite series of the mixed raw moments of order $(r_1,r_2)$ of the $\mbox{Dir}^{\, 2}\left(\alpha_1+j_1,\alpha_2+j_2,\alpha_3+j_3\right)$ distributions weighted by the $\mbox{MW}^{\, 3}\left(\alpha^+,\lambda_1,\lambda_2,\lambda_3\right)$ probabilities. Moreover, by Eq.~(\ref{eq:poch.symb.sum}), Eq.~(\ref{eq:cncdir.mixr1r2mom.def}) can be equivalently expressed as the following doubly infinite sum of generalized hypergeometric functions with one numerator parameter and two denominator parameters:
\begin{eqnarray}
\lefteqn{\mathbb{E}\left[\left(X_1''\right)^{ \, r_1} \left(X_2''\right)^{ \, r_2} \, \right]=\frac{\left(\alpha_1\right)_{r_1} \left(\alpha_2\right)_{r_2}}{\left(\alpha^+\right)_{r^+}} \cdot \sum_{j_2, \, j_3 \, = \, 0}^{+\infty} \left[ \, \frac{\left(\alpha_2+r_2\right)_{j_2}}{\left(\alpha_2\right)_{j_2}  \, \left(\alpha^++r^+\right)_{j_2+j_3}} \right. \cdot}\nonumber \\
& \cdot & \left. \frac{\left(\frac{\lambda_2}{4}\right)^{j_2}}{j_2 !} \frac{\left(\frac{\lambda_3}{4}\right)^{j_3}}{j_3 !} \cdot \, \frac{_1^{\, }F^{\, }_2\left(\alpha_1+r_1;\alpha_1,\alpha^++r^++j_2+j_3;\frac{\lambda_1}{4}\right)}{_0F_1\left(\alpha^+; \frac{\lambda^+}{4}\right)} \, \right]\, .
\label{eq:cncdir.mixr1r2mom.def2}
\end{eqnarray}
Unfortunately, the above formulas are computationally cumbersome. Hence, in the following we shall provide an interesting general expression for the quantity at study that allows the computation of this latter to be reduced from the doubly infinite sum in Eq.~(\ref{eq:cncdir.mixr1r2mom.def2}) to a surprisingly simple form given by a doubly finite sum.

\begin{proposition}[Mixed raw moments of the bivariate distribution]
\label{propo:cncdir.momr1r2}
The mixed raw moment of order $(r_1,r_2)$ of $\left(X''_1,X''_2\right) \sim \mbox{\normalfont{CNcDir}}^{\, 2}\left(\alpha_1,\alpha_2,\alpha_3,\lambda_1,\lambda_2,\lambda_3\right)$ admits the following expression:
\begin{eqnarray}
\lefteqn{\mathbb{E}\left[\left(X''_1\right)^{r_1} \left(X''_2\right)^{r_2}\right] \quad = \quad \frac{\left(\alpha_1\right)_{r_1}  \left(\alpha_2\right)_{r_2}}{\left(\alpha^+\right)_{r^+}}  \; \cdot \qquad \quad \left\{\begin{array}{l} r_1,r_2 \in \mathbb{N}, \, r^+=r_1+r_2 \\ j^+=j_1+j_2 \end{array}\right.} \nonumber\\
& \cdot & \sum_{j_1=0}^{r_1} \sum_{j_2=0}^{r_2} \frac{{r_1 \choose j_1} {r_2 \choose j_2} \left(\frac{\lambda_1}{4}\right)^{j_1} \left(\frac{\lambda_2}{4}\right)^{j_2}}{\left(\alpha_1\right)_{j_1} \left(\alpha_2\right)_{j_2} \left(\alpha^++r^+\right)_{j^+}} \frac{_0F_1\left(\alpha^++r^++j^+;\frac{\lambda^+}{4}\right)}{_0F_1\left(\alpha^+;\frac{\lambda^+}{4}\right)} \, .
\label{eq:cncdir.momr1r2}
\end{eqnarray}
\end{proposition}
\begin{proof}
Let $\left(L_1,L_2\right)$ have a $\mbox{Multinomial}^{\, 2}\left(N^+,\theta_1,\theta_2\right)$ distribution conditionally on $N^+$ with $\theta_i=\lambda_i \, / \, \lambda^+$, $i=1,2$. By Eq.~(\ref{eq:cncdir.cond.distr.nplus}), one has:
\begin{eqnarray}
\lefteqn{\mathbb{E}\left[\left. \left(X''_1\right)^{r_1} \left(X''_2\right)^{r_2} \right| \, N^+\right]=}\nonumber \\
& = & \int_{\left(x_1,\, x_2\right) \, \in \, \mathcal{S}^2} x_1^{\, r_1} x_2^{\, r_2} \cdot f_{\left.\left(X''_1,\, X''_2\right) \, \right| \, N^+}\left(x_1,x_2;\alpha_1,\alpha_2,\alpha_3,\lambda_1,\lambda_2,\lambda_3\right) \, dx_1 \, dx_2=\nonumber \\
& = & \sum_{l_1=0}^{N^+} \sum_{l_2=0}^{N^+- \, l_1} \frac{\left(\alpha_1+l_1\right)_{r_1} \left(\alpha_2+l_2\right)_{r_2} }{\left(\alpha^++N^+\right)_{r^+}} \, {N^+ \choose l_1 \; l_2} \; \theta_1^{\, l_1} \,  \theta_2^{\, l_2} \left(1-\theta_1-\theta_2\right)^{N^+- \, l_1-l_2} = \nonumber \\
& = & \frac{\mathbb{E}\left[\left. \left(\alpha_1+L_1\right)_{r_1} \left(\alpha_2+L_2\right)_{r_2} \right| \, N^+\right]}{\left(\alpha^++N^+\right)_{r^+}} \, ;
\label{eq:mom.dim1}
\end{eqnarray}
in light of Eq.~(\ref{eq:poch.symb.binom}):
$$\left(\alpha_i+L_i\right)_{r_i}=\left[\left(\alpha_i-1\right)+\left(L_i+1\right)\right]_{r_i}=\sum_{j_i=0}^{r_i} {r_i \choose j_i} \left(\alpha_i-1\right)_{r_i-j_i} \left(L_i+1\right)_{j_i} \, , \quad \; i=1,2 \, ,$$
so that one obtains:
\begin{eqnarray}
\lefteqn{\mathbb{E}\left[\left. \left(\alpha_1+L_1\right)_{r_1} \left(\alpha_2+L_2\right)_{r_2} \right| \, N^+\right]=} \nonumber \\
& = & \sum_{j_1=0}^{r_1} \sum_{j_2=0}^{r_2} {r_1 \choose j_1} {r_2 \choose j_2} \left(\alpha_1-1\right)_{r_1-j_1} \left(\alpha_2-1\right)_{r_2-j_2} \mathbb{E}\left[\left. \left(L_1+1\right)_{j_1} \left(L_2+1\right)_{j_2} \right| \, N^+\right] \, , \nonumber \\
\label{eq:mom.dim1b}
\end{eqnarray}
where:
\begin{eqnarray}
\lefteqn{\mathbb{E}\left[\left. \left(L_1+1\right)_{j_1} \left(L_2+1\right)_{j_2} \right| \, N^+\right]=} \nonumber \\
& = & \sum_{l_1=0}^{N^+} \sum_{l_2=0}^{N^+-\, l_1} \left(l_1+1\right)_{j_1} \left(l_2+1\right)_{j_2} \, {N^+ \choose l_1 \; l_2} \; \theta_1^{\, l_1} \,  \theta_2^{\, l_2} \left(1-\theta_1-\theta_2\right)^{N^+- \, l_1- \, l_2} \, .
\label{eq:mom.dimb2}
\end{eqnarray}
By Eq.~(\ref{eq:poch.symb}), for every $j_i=0,\, \ldots \, ,r_i$, $i=1,2$:
\begin{equation}
\left(l_i+1\right)_{j_i}=\frac{\Gamma\left(l_i+j_i+1\right)}{\Gamma\left(l_i+1\right)}=\frac{\left(l_i+j_i\right) !}{l_i !}={l_i+j_i \choose l_i} \, j_i! \; ;
\label{eq:mom.dimb3}
\end{equation}
under Eq.~(\ref{eq:mom.dimb3}), Eq.~(\ref{eq:mom.dimb2}) can be thus rewritten as follows: 
\begin{eqnarray}
\lefteqn{\mathbb{E}\left[\left. \left(L_1+1\right)_{j_1} \left(L_2+1\right)_{j_2} \right| \, N^+\right]=} \nonumber \\
& = & j_1 ! \; j_2 ! \, \sum_{l_1=0}^{N^+} {l_1 +j_1 \choose l_1} {N^+ \choose l_1} \left(1-\theta_1\right)^{N^+-\, l_1} \theta_1^{\, l_1}  \cdot \nonumber \\
& \cdot & \sum_{l_2=0}^{N^+-\, l_1} {l_2 +j_2 \choose l_2} {N^+-\, l_1 \choose l_2} \left(1-\frac{\theta_2}{1-\theta_1}\right)^{\left(N^+-\, l_1\right)- \, l_2} \left(\frac{\theta_2}{1-\theta_1}\right)^{\, l_2}  \, .
\label{eq:mom.dim1c}
\end{eqnarray}
In carrying out the prove, reference must be made to Ljunggren's Identity, namely
\begin{equation}
\sum_{k=0}^{n} {\alpha+k \choose k} {n \choose k} \left(x-y\right)^{n-k} y^k=\sum_{k=0}^{n} {\alpha \choose k} {n \choose k} \, x^{n-k} \, y^k \, , 
\label{eq:ljunggren.id}
\end{equation}
which is (3.18) in \cite{Gou72}. By the special case of Eq.~(\ref{eq:ljunggren.id}) where $k=l_2$, $n=N^+- \, l_1$, $\alpha=j_2$, $x=1$, $y=\theta_2/(1-\theta_1)$, the final sum in Eq.~(\ref{eq:mom.dim1c}) can be equivalently expressed as
$$
\sum_{l_2=0}^{N^+-\, l_1} {j_2 \choose l_2} \, {N^+-\, l_1 \choose l_2} \left(\frac{\theta_2}{1-\theta_1}\right)^{l_2} \, ,$$
so that, for every $j_i=0,\, \ldots \, ,r_i$, $i=1,2$, Eq.~(\ref{eq:mom.dim1c}) can be restated as
\begin{eqnarray}
\lefteqn{\mathbb{E}\left[\left. \left(L_1+1\right)_{j_1} \left(L_2+1\right)_{j_2} \right| \, N^+\right]=} \nonumber \\
& = & j_1 ! \, j_2 ! \, \sum_{l_1=0}^{N^+} \sum_{l_2=0}^{N^+-\, l_1} {j_2 \choose l_2} {l_1+j_1 \choose l_1} {N^+ \choose l_1 \; l_2}  \,  \theta_1^{\, l_1} \,  \theta_2^{\, l_2} \left(1-\theta_1\right)^{N^+- \, l_1-\, l_2}= \nonumber \\
& = & j_1 ! \, j_2 ! \, \sum_{l_2=0}^{N^+} {j_2 \choose l_2} \, {N^+ \choose l_2} \,  \theta_2^{\, l_2} \cdot \nonumber \\
& \cdot & \sum_{l_1=0}^{N^+-\, l_2}  {l_1+j_1 \choose l_1} \, {N^+-l_2 \choose l_1}  \left(1-\theta_1\right)^{(N^+- \, l_2)-\, l_1} \,  \theta_1^{\, l_1}   \, .
\label{eq:mom.dim2a}
\end{eqnarray}
Under the special case of Eq.~(\ref{eq:ljunggren.id}) where $k=l_1$, $n=N^+- \, l_2$, $\alpha=j_1$, $x=1$, $y=\theta_1$, Eq.~(\ref{eq:mom.dim2a}) can be rewritten in the form of
\begin{eqnarray}
\lefteqn{\mathbb{E}\left[\left. \left(L_1+1\right)_{j_1} \left(L_2+1\right)_{j_2} \right| \, N^+\right]=} \nonumber \\
& = & j_1 ! \, j_2 ! \, \sum_{l_2=0}^{N^+} {j_2 \choose l_2} {N^+ \choose l_2} \,  \theta_2^{\, l_2} \sum_{l_1=0}^{N^+-\, l_2} {j_1 \choose l_1} {N^+-\, l_2 \choose l_1}   \,  \theta_1^{\, l_1} \, ,
\label{eq:mom.dim2b}
\end{eqnarray}
so that, by Eqs.~(\ref{eq:mom.dim1}),~(\ref{eq:mom.dim1b}) and~(\ref{eq:mom.dim2b}), conditionally on $N^+$, the mixed raw moment of interest takes on the following expression:
\begin{eqnarray}
\lefteqn{\mathbb{E}\left[\left. \left(X''_1\right)^{r_1} \left(X''_2\right)^{r_2} \right| \, N^+ \right]=}\nonumber \\ 
& = & \frac{1}{\left(\alpha^++N^+\right)_{r^+}} \sum_{j_1=0}^{r_1} \sum_{j_2=0}^{r_2} {r_1 \choose j_1} j_1 ! \, {r_2 \choose j_2} j_2 ! \left(\alpha_1-1\right)_{r_1-j_1} \left(\alpha_2-1\right)_{r_2-j_2} \cdot \nonumber \\
& \cdot & \sum_{l_2=0}^{N^+} {j_2 \choose l_2} {N^+ \choose l_2} \,  \theta_2^{\, l_2} \sum_{l_1=0}^{N^+-\, l_2} {j_1 \choose l_1} {N^+-\, l_2 \choose l_1}   \,  \theta_1^{\, l_1} \, ;
\label{eq:mom.dim2c}
\end{eqnarray}
applying the law of iterated expectations to Eq.~(\ref{eq:mom.dim2c}) finally leads to:
\begin{eqnarray}
\lefteqn{\mathbb{E}\left[\left(X''_1\right)^{r_1} \left(X''_2\right)^{r_2}\right]=}\nonumber \\ 
& = & \frac{\Gamma\left(\alpha^+\right)}{_0F_1\left(\alpha^+;\frac{\lambda^+}{4}\right)} \sum_{j_1=0}^{r_1} \sum_{j_2=0}^{r_2} {r_1 \choose j_1} j_1 ! \, {r_2 \choose j_2} j_2 ! \left(\alpha_1-1\right)_{r_1-j_1} \left(\alpha_2-1\right)_{r_2-j_2} \cdot \nonumber \\
& \cdot & \sum_{n=0}^{+\infty} \frac{\left(\frac{\lambda^+}{4}\right)^n}{n ! \, \Gamma\left(\alpha^++r^++n\right)} \sum_{l_2=0}^{n} {j_2 \choose l_2} {n \choose l_2} \,  \theta_2^{\, l_2} \sum_{l_1=0}^{n-\, l_2} {j_1 \choose l_1} {n-\, l_2 \choose l_1}   \,  \theta_1^{\, l_1} \, .
\label{eq:mom.dim3}
\end{eqnarray}
Since ${j_i \choose l_i}=0$ for $l_i>j_i$, $i=1,2$, one has:
$$
l_2 \leq j_2 \leq n \; \Rightarrow \; \left\{\begin{array}{l} l_2=0, \, \ldots \, ,j_2 \\ \\ n=l_2, \, \ldots \, ,+\infty \end{array} \right. \, \quad 
l_1 \leq j_1 \leq n-l_2 \; \Rightarrow \; \left\{\begin{array}{l} l_1=0, \, \ldots \, ,j_1 \\ \\ n=l_1+l_2, \, \ldots \, ,+\infty \end{array} \right.
$$
so that Eq.~(\ref{eq:mom.dim3}) is equivalent to:
\begin{eqnarray*}
\lefteqn{\mathbb{E}\left[\left(X''_1\right)^{r_1} \left(X''_2\right)^{r_2}\right]=\frac{\Gamma\left(\alpha^+\right)}{_0F_1\left(\alpha^+;\frac{\lambda^+}{4}\right)} \cdot}\\ 
& \cdot & \sum_{j_1=0}^{r_1} {r_1 \choose j_1} j_1 ! \left(\alpha_1-1\right)_{r_1-j_1} \, \sum_{l_1=0}^{j_1} \frac{\theta_1^{\, l_1}}{l_1 !} \, {j_1 \choose l_1} \, \sum_{j_2=0}^{r_2} {r_2 \choose j_2} j_2 ! \left(\alpha_2-1\right)_{r_2-j_2} \cdot \\
& \cdot & \sum_{l_2=0}^{j_2} \frac{\theta_2^{\, l_2}}{l_2 !} \, {j_2 \choose l_2} \sum_{n=l_1+l_2}^{+\infty} \frac{\left(\frac{\lambda^+}{4}\right)^n}{\left(n-\, l_1-\, l_2\right) ! \; \Gamma\left(\alpha^++r^++n\right)} \, ;
\end{eqnarray*}
then, by Eqs.~(\ref{eq:poch.symb}) and~(\ref{eq:f01}), setting $k=n-(l_1+l_2) \Leftrightarrow n=k+(l_1+l_2)$ yields:
\begin{eqnarray*}  
& = & \Gamma\left(\alpha^+\right) \, \sum_{j_1=0}^{r_1} {r_1 \choose j_1} j_1 ! \left(\alpha_1-1\right)_{r_1-j_1} \, \sum_{l_1=0}^{j_1} \frac{\left(\frac{\lambda_1}{4}\right)^{l_1} {j_1 \choose l_1}}{l_1 !}  \, \sum_{j_2=0}^{r_2} {r_2 \choose j_2} j_2 ! \left(\alpha_2-1\right)_{r_2-j_2} \cdot \\
& \cdot & \sum_{l_2=0}^{j_2} \frac{\left(\frac{\lambda_2}{4}\right)^{l_2} \, {j_2 \choose l_2}}{l_2 ! \; \Gamma\left(\alpha^++r^++l_1+l_2\right)} \,  \sum_{k=0}^{+\infty} \frac{\left(\frac{\lambda^+}{4}\right)^k}{k ! \, \left(\alpha^++r^++l_1+l_2\right)_k} \frac{1}{_0F_1\left(\alpha^+;\frac{\lambda^+}{4}\right)}= \\
& = & \Gamma\left(\alpha^+\right) \sum_{j_1=0}^{r_1} {r_1 \choose j_1} j_1 ! \left(\alpha_1-1\right)_{r_1-j_1} \, \sum_{l_1=0}^{j_1} \frac{\left(\frac{\lambda_1}{4}\right)^{l_1} {j_1 \choose l_1}}{l_1 !}  \, \sum_{j_2=0}^{r_2} {r_2 \choose j_2} j_2 ! \left(\alpha_2-1\right)_{r_2-j_2} \cdot \\
& \cdot & \sum_{l_2=0}^{j_2} \frac{\left(\frac{\lambda_2}{4}\right)^{l_2} \, {j_2 \choose l_2}}{l_2 ! \; \Gamma\left(\alpha^++r^++l_1+l_2\right)} \,  \frac{_0F_1\left(\alpha^++r^++l_1+l_2;\frac{\lambda^+}{4}\right)}{_0F_1\left(\alpha^+;\frac{\lambda^+}{4}\right)}
\end{eqnarray*}
and, by observing that:
$$
\left\{\begin{array}{l} j_i=0, \, \ldots \, , r_i \\ \\ l_i=0, \, \ldots \, , j_i \end{array} \right. \quad\Rightarrow \quad 0 \, \leq \, l_i \, \leq \, j_i \, \leq \, r_i \quad \Rightarrow \quad 	\left\{\begin{array}{l} l_i=0, \, \ldots \, ,r_i \\ \\ j_i=l_i, \, \ldots \, ,r_i\end{array} \right. \, ,
$$
one can obtain: 
\begin{eqnarray}
\lefteqn{\mathbb{E}\left[\left(X''_1\right)^{r_1} \left(X''_2\right)^{r_2}\right]=} \nonumber \\
& = & \Gamma\left(\alpha^+\right) \sum_{l_1=0}^{r_1} \frac{\left(\frac{\lambda_1}{4}\right)^{l_1}}{\left(l_1 !\right)^2} \, \sum_{l_2=0}^{r_2} \frac{\left(\frac{\lambda_2}{4}\right)^{l_2} \, _0F_1\left(\alpha^++r^++l_1+l_2;\frac{\lambda^+}{4}\right)}{\left(l_2 !\right)^2 \, \Gamma\left(\alpha^++r^++l_1+l_2\right) \, _0F_1\left(\alpha^+;\frac{\lambda^+}{4}\right)} \cdot \nonumber \\
& \cdot & \sum_{j_1=l_1}^{r_1} {r_1 \choose j_1} \left(\alpha_1-1\right)_{r_1-j_1} \frac{\left(j_1 !\right)^2}{\left(j_1-l_1\right)!}  \, \sum_{j_2=l_2}^{r_2} {r_2 \choose j_2} \left(\alpha_2-1\right)_{r_2-j_2} \frac{\left(j_2 !\right)^2}{\left(j_2-l_2\right)!} \, .
\label{eq:mom.dim4}
\end{eqnarray}
By setting $p_i=j_i-l_i \Leftrightarrow j_i=p_i+l_i$, $i=1,2$ and by Eq.~(\ref{eq:poch.symb.binom}), each of the two final sums in Eq.~(\ref{eq:mom.dim4}) turns out to be tantamount to:
\begin{eqnarray}
\lefteqn{\sum_{j_i=l_i}^{r_i} {r_i \choose j_i} \left(\alpha_i-1\right)_{r_i-j_i} \frac{\left(j_i !\right)^2}{\left(j_i-l_i\right)!}=} \nonumber \\
& = & \sum_{p_i=0}^{r_i-l_i} {r_i \choose p_i+l_i} \left(\alpha_i-1\right)_{r_i-p_i-l_i} \frac{\left[\left(p_i+l_i\right)!\right]^2}{p_i!}=\nonumber \\
& = & \frac{r_i !}{\left(r_i-l_i\right)!}\sum_{p_i=0}^{r_i-l_i} {r_i-l_i \choose p_i} \left(\alpha_i-1\right)_{r_i-p_i-l_i} \, \left(l_i+p_i\right)!= \nonumber \\
& = & \frac{r_i ! \, l_i !}{\left(r_i-l_i\right)!}\sum_{p_i=0}^{r_i-l_i} {r_i-l_i \choose p_i} \left(\alpha_i-1\right)_{r_i-p_i-l_i} \, \left(l_i+1\right)_{p_i}=\nonumber \\
& = & \frac{r_i ! \, l_i !}{\left(r_i-l_i\right)!}\left[\left(\alpha_i-1\right)+\left(l_i+1\right)\right]_{r_i-l_i}=\frac{r_i ! \, l_i !}{\left(r_i-l_i\right)!} \left(\alpha_i+l_i\right)_{r_i-l_i} \, ;
\label{eq:mom.dim5}
\end{eqnarray}
hence, by Eq.~(\ref{eq:mom.dim5}), Eq.~(\ref{eq:mom.dim4}) can be stated in the following form:
\begin{eqnarray}
\lefteqn{\mathbb{E}\left[\left(X''_1\right)^{r_1} \left(X''_2\right)^{r_2}\right]=} \nonumber \\
& = & \Gamma\left(\alpha^+\right) \cdot \sum_{l_1=0}^{r_1} \sum_{l_2=0}^{r_2}  \left[ \frac{{r_1 \choose l_1} \, {r_2 \choose l_2} \left(\frac{\lambda_1}{4}\right)^{l_1} \left(\frac{\lambda_2}{4}\right)^{l_2} \, \left(\alpha_1+l_1\right)_{r_1-l_1} \, \left(\alpha_2+l_2\right)_{r_2-l_2}}{\Gamma\left(\alpha^++r^++l_1+l_2\right)} \cdot  \right. \nonumber \\
& \cdot & \left. \frac{_0F_1\left(\alpha^++r^++l_1+l_2;\frac{\lambda^+}{4}\right)}{_0F_1\left(\alpha^+;\frac{\lambda^+}{4}\right)} \right] \, .
\label{eq:mom.dim6}
\end{eqnarray}
Eq.~(\ref{eq:mom.dim6}) can be finally exhibited in the form reported in Eq.~(\ref{eq:cncdir.momr1r2}) by noting that:
\begin{eqnarray*}
\Gamma\left(\alpha^++r^++l_1+l_2\right) & = & \Gamma\left(\alpha^+\right) \, \left(\alpha^+\right)_{r^+} \, \left(\alpha^++r^+\right)_{l_1+l_2} \,  ,\\
\left(\alpha_i+l_i\right)_{r_i-l_i} & = & \frac{\left(\alpha_i\right)_{r_i}}{\left(\alpha_i\right)_{l_i}} \, \qquad i=1,2
\end{eqnarray*}
in light of Eqs.~(\ref{eq:poch.symb}),~(\ref{eq:poch.symb.sum}) and~(\ref{eq:poch.symb.ratio}).
\end{proof}

Hence, the formula of the mixed raw moment of order $(1,1)$ of the $\mbox{\normalfont{CNcDir}}^{\, 2}$ distribution can be obtained by taking $r_1=r_2=1$ in Eq.~(\ref{eq:cncdir.momr1r2}) as follows:
\begin{eqnarray}
\mathbb{E}\left(X''_1 \, X''_2\right) & = & \frac{\alpha_1 \, \alpha_2}{\left(\alpha^+\right)_2} \frac{_0F_1\left(\alpha^++2;\frac{\lambda^+}{4}\right)}{_0F_1\left(\alpha^+;\frac{\lambda^+}{4}\right)}+\frac{\alpha_1 \, \frac{\lambda_2}{4}+ \alpha_2 \, \frac{\lambda_1}{4}}{\left(\alpha^+\right)_3} \, \frac{_0F_1\left(\alpha^++3;\frac{\lambda^+}{4}\right)}{_0F_1\left(\alpha^+;\frac{\lambda^+}{4}\right)}+ \nonumber \\
& + & \frac{\frac{\lambda_1}{4} \, \frac{\lambda_2}{4}}{\left(\alpha^+\right)_4} \frac{_0F_1\left(\alpha^++4;\frac{\lambda^+}{4}\right)}{_0F_1\left(\alpha^+;\frac{\lambda^+}{4}\right)} \, .
\label{eq:cncdir.mom11}
\end{eqnarray}
This latter can be algebraically manipulated with the aim to reduce the number of distinct functions $\, _0F_1$ appearing in it. More precisely, in carrying out the computations, reference must be made to the recurrence identity
$$_0F_1\left(\alpha^++2;\frac{\lambda^+}{4}\right)= \, _0F_1\left(\alpha^++3;\frac{\lambda^+}{4}\right)+\frac{\frac{\lambda^+}{4}}{\left(\alpha^++2\right)\left(\alpha^++3\right)} \, _0F_1\left(\alpha^++4;\frac{\lambda^+}{4}\right) \, ,$$
which holds for consecutive values of the denominator parameter of $\, _0F_1$ and is easy to verify upon noting that, by Eq.~(\ref{eq:poch.symb.sum}), $\left(\alpha^++2\right)_i \left(\alpha^++2+i\right)=\left(\alpha^++2\right) \left(\alpha^++3\right)_i$ for every $i \in \mathbb{N} \cup \{0\}$. Specifically, the aforementioned improvement of Eq.~(\ref{eq:cncdir.mom11}) depends on three functions $_0F_1$ instead of four and is given by 
\begin{eqnarray*}
\mathbb{E}\left(X''_1 \, X''_2\right) & = & \frac{\alpha_1 \, \alpha_2+\frac{\lambda_1 \, \lambda_2}{4 \, \lambda^+}}{\left(\alpha^+\right)_2} \frac{_0F_1\left(\alpha^++2;\frac{\lambda^+}{4}\right)}{_0F_1\left(\alpha^+;\frac{\lambda^+}{4}\right)}+ \nonumber \\
& + & \frac{\alpha_1 \, \frac{\lambda_2}{4}+ \alpha_2 \, \frac{\lambda_1}{4}-\frac{\lambda_1 \, \lambda_2}{4 \, \lambda^+}\left(\alpha^++2\right)}{\left(\alpha^+\right)_3} \, \frac{_0F_1\left(\alpha^++3;\frac{\lambda^+}{4}\right)}{_0F_1\left(\alpha^+;\frac{\lambda^+}{4}\right)} \, .
\end{eqnarray*}

\subsection{Identifiability}
\label{subsec:cncdir.identif}

Before facing the applicative issues included in the next section, in the present subsection we briefly discuss the identifiability of the $\mbox{\normalfont{CNcDir}}$ model.

\begin{proposition}[Identifiability]
\label{propo:cncdir.identif}
Let $\underline{X}''=\left(X''_1, \, \ldots \, , X''_D \right) \sim \mbox{\normalfont{CNcDir}}^{\, D}\left(\underline{\alpha},\underline{\lambda}\right)$ and $\underline{\tilde{X}}''=(\tilde{X}''_{1}, \, \ldots \, , \tilde{X}''_{D}) \sim \mbox{\normalfont{CNcDir}}^{\, D}(\underline{\tilde{\alpha}},$ $\underline{\tilde{\lambda}})$ where $\underline{\alpha}=\left(\alpha_1, \, \ldots \, ,\alpha_{D+1}\right)$, $\underline{\tilde{\alpha}}=\left(\tilde{\alpha}_1, \, \ldots \, ,\tilde{\alpha}_{D+1}\right)$, $\underline{\lambda}=\left(\lambda_1, \, \ldots \, ,\lambda_{D+1}\right)$ and $\underline{\tilde{\lambda}}=(\tilde{\lambda}_1, \, \ldots \, , \tilde{\lambda}_{D+1})$. Then, $\underline{X}'' \sim \underline{\tilde{X}}''$ if and only if $\underline{\alpha}=\underline{\tilde{\alpha}}$ and $\underline{\lambda}=\underline{\tilde{\lambda}}$.
\end{proposition}
\begin{proof}
We need to show that if $\underline{X}'' \sim \underline{\tilde{X}}''$ then $\underline{\alpha}=\underline{\tilde{\alpha}}$ and $\underline{\lambda}=\underline{\tilde{\lambda}}$, the converse being obvious. Clearly, if $\underline{X}'' \sim \underline{\tilde{X}}''$, then $X''_i \sim \tilde{X}''_{i}$ for every $i=1, \, \ldots \, ,D$. By Eq.~(\ref{eq:cncdir.unidim.marg}), $X''_i \sim \mbox{\normalfont{CDNcB}}(\alpha_i,\alpha^+-\alpha_i,\lambda_i,\lambda^+-\lambda_i)$ and $\tilde{X}''_i \sim \mbox{\normalfont{CDNcB}}(\tilde{\alpha}_i,\tilde{\alpha}^+-\tilde{\alpha}_i,\tilde{\lambda}_i,\tilde{\lambda}^+-\tilde{\lambda}_i)$; therefore, the proof follows from the identifiability of the CDNcB model (\cite{Ors21}, Section 7).
\end{proof}

\section{Applications}
\label{sec:cncdir.examples}

The potential of the $\mbox{CNcDir}$ model is now illustrated by means of an application to a real data set. To this end we turned our attention to the Cartesian coordinates of the locations of 584 longleaf pine trees (\textit{Pinus palustris}) placed in a $200 \times 200$ metre region within Southern Georgia (USA). This data set was originally collected and analyzed by \cite{PlaEvaRat88}, was later quoted in \cite{RatCre94} and finally was made available by \cite{BadTurRub20} as part of the package ``spatstat.data'' in the programming environment \texttt{R}. More precisely, in a comparative perspective, the four distributions on the unit simplex considered in the present paper were fitted to the data set obtained as follows. The above region was rescaled to the unit square and cut by the diagonal joining the vertices $(1,0)$ and $(0,1)$ into two congruent triangles; hence, the aforementioned analysis was carried out on the subset of $n=346$ locations lying in the interior of the upper triangle after undergoing the transformation $(x_1,x_2) \mapsto (x'_1,x'_2)=(1-x_2,1-x_1)$ that mapped this latter into the unit simplex in $\mathbb{R}^2$ (see Figure~\ref{fig:scatter}).

\begin{figure}[htp]
  \centering
  \includegraphics[width=8cm,keepaspectratio=true]{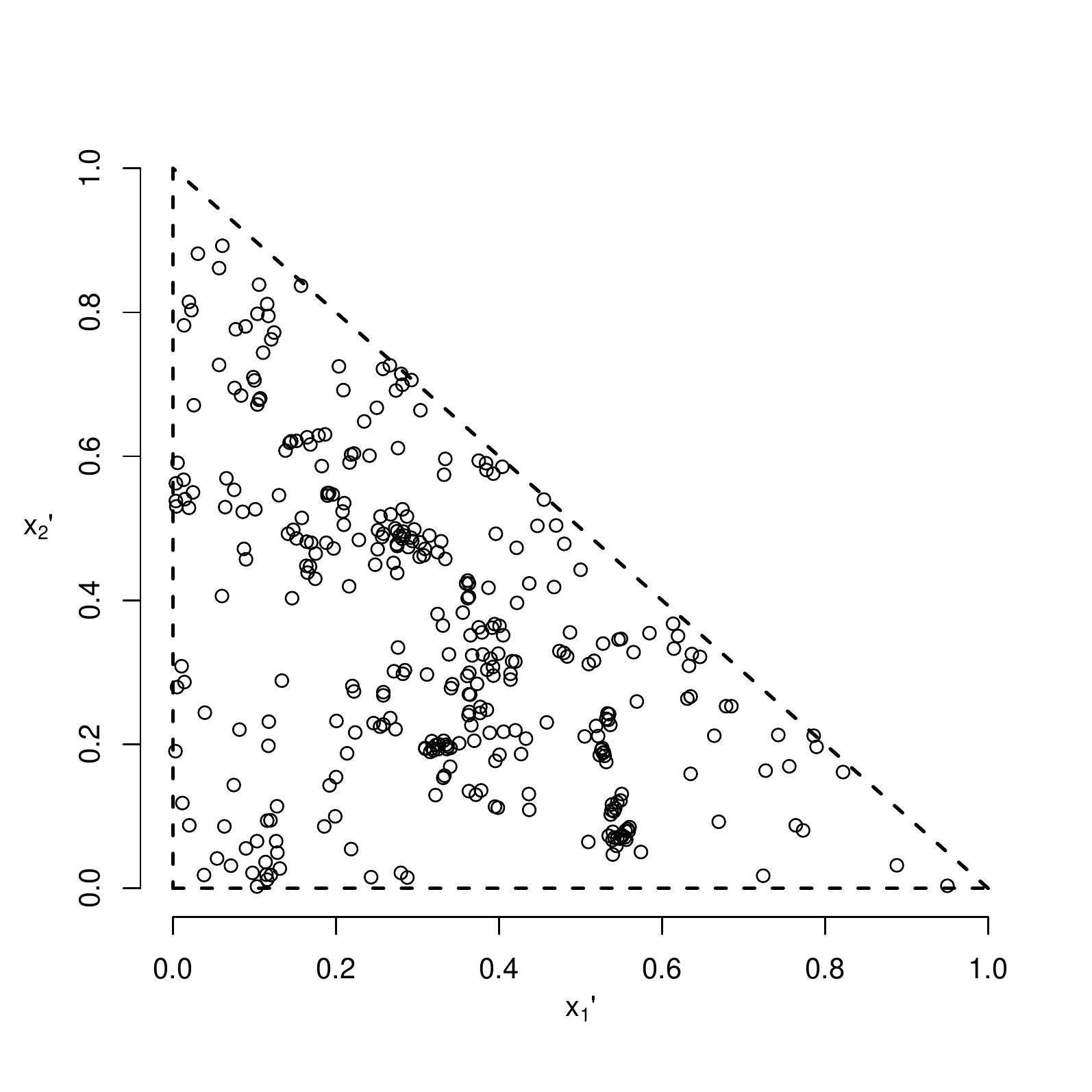}\\
\caption{Scatterplot of the locations of $n=346$ longleaf pine trees expressed in the Cartesian coordinate system $(x'_1,x'_2)=(1-x_2,1-x_1)$ mapping the upper triangle of the unit square into the unit simplex in $\mathbb{R}^2$.}\label{fig:scatter}
\end{figure}

What is of interest to us is to check if the evidence suggests that the data portions having values next to the vertices of $\mathcal{S}^2$ are large enough to let the densities of such models take on positive and finite limits as $(x'_1,x'_2)$ tends to $(1,0)$, $(0,1)$ and $(0,0)$. Hence, the problem at study can be formalized by testing the hypotheses that all three shape parameters of each model or two of them, at the very least, are unitary.

That said, the method of Maximum-Likelihood (ML) was applied in order to derive the estimates for the parameter vectors of all the above models. In this regard, let $\left\{\underline{X}_i=\left(X_{i1},X_{i2} \right)\right\}_{i=1}^{n}$ be a sequence of independent random vectors with identical distribution depending on an unknown parameter vector $\underline{\theta}$. First, suppose that $\left(X_{i1},X_{i2} \right) \stackrel{\bot}{\sim } \mbox{\normalfont{Dir}}^{\, 2}\left(\alpha_1,\alpha_2,\alpha_3\right)$, $i=1, \, \ldots \, ,n$; by Eq.~(\ref{eq:dir.dens}), given the observed sample $\left\{ \underline{x}_i=\left(x_{i1},x_{i2} \right) \right\}_{i=1}^{n}$, the log-likelihood function for the parameter vector $\left(\alpha_1,\alpha_2,\alpha_3\right)$ of the $\mbox{\normalfont{Dir}}^{\, 2}$ model is
\begin{eqnarray}
\lefteqn{l_{\, \mbox{\normalfont{Dir}}^{\, 2}}\left(\alpha_1,\alpha_2,\alpha_3;\underline{x}_1, \, \ldots \, ,\underline{x}_n\right)=} \nonumber \\
& = & n \left[\, \log \Gamma\left(\alpha^+\right)-\log \Gamma\left(\alpha_1\right)- \log \Gamma\left(\alpha_2\right)- \log \Gamma\left(\alpha_3\right) \, \right] +\left(\alpha_1-1\right) \sum_{i=1}^{n}\log x_{i1}+ \nonumber \\
& + & \left(\alpha_2-1\right) \sum_{i=1}^{n}\log x_{i2}+\left(\alpha_3-1\right) \sum_{i=1}^{n}\log\left(1-x_{i1}-x_{i2}\right) \, .
\label{eq:loglik.dir}
\end{eqnarray}
By taking $\left(X_{i1},X_{i2} \right) \stackrel{\bot}{\sim } \mbox{\normalfont{KB}}^{\, 2}\left(\alpha_1,\alpha_2,\alpha_3,\delta\right)$, $i=1, \, \ldots \, ,n$, in light of Eqs.~(\ref{eq:kb2.dens}) and~(\ref{eq:loglik.dir}), given the observed sample $\left\{\underline{x}_i=\left(x_{i1},x_{i2} \right) \right\}_{i=1}^{n}$, the log-likelihood function for the parameter vector $\left(\alpha_1,\alpha_2,\alpha_3,\delta\right)$ of the $\mbox{\normalfont{KB}}^{\, 2}$ model is
\begin{eqnarray*}
\lefteqn{l_{\, \mbox{\normalfont{KB}}^{\, 2}}\left(\alpha_1,\alpha_2,\alpha_3,\delta;\underline{x}_1, \, \ldots \, ,\underline{x}_n\right) \quad =}\\
& = & l_{\, \mbox{\normalfont{Dir}}^{\, 2}}\left(\alpha_1,\alpha_2,\alpha_3;\underline{x}_1, \, \ldots \, , \underline{x}_n\right)-\left[\sum_{i=1}^{n}\left(x_{i1}+x_{i2}\right)\right]\delta -n\log \, _1F_1\left(\alpha_1+\alpha_2;\alpha^+;-\delta\right) \, .
\end{eqnarray*}
Now let $\left(X_{i1},X_{i2} \right) \stackrel{\bot}{\sim } \mbox{\normalfont{NcDir}}^{\, 2}\left(\alpha_1,\alpha_2,\alpha_3,\lambda_1,\lambda_2,\lambda_3\right)$, $i=1, \, \ldots \, ,n$; by Eqs.~(\ref{eq:ncdir.perturb.dens}) and~(\ref{eq:loglik.dir}), given the observed sample $\left\{\underline{x}_i=\left(x_{i1},x_{i2} \right) \right\}_{i=1}^{n}$, the log-likelihood function for the parameter vector $\left(\alpha_1,\alpha_2,\alpha_3,\lambda_1,\lambda_2,\lambda_3\right)$ of the $\mbox{\normalfont{NcDir}}^{\, 2}$ model is
\begin{eqnarray*}
\lefteqn{l_{\, \mbox{\normalfont{NcDir}}^{\, 2}}\left(\alpha_1,\alpha_2,\alpha_3,\lambda_1,\lambda_2,\lambda_3;\underline{x}_1, \, \ldots \, , \underline{x}_n\right)=l_{\, \mbox{\normalfont{Dir}}^{\, 2}}\left(\alpha_1,\alpha_2,\alpha_3;\underline{x}_1, \, \ldots \, ,\underline{x}_n\right)  +}\\
& - &  \frac{n \, \lambda^+}{2} + \sum_{i=1}^{n}\log \Psi_2^{(3)}\left[\alpha^+;\alpha_1,\alpha_2,\alpha_3;\frac{\lambda_1}{2}x_{i1},\frac{\lambda_2}{2}x_{i2},\frac{\lambda_{3}}{2}\left(1-x_{i1}-x_{i2}\right)\right] \, ,
\end{eqnarray*}
where, in light of Eqs.~(\ref{eq:poch.symb.sum}) and~(\ref{eq:ncdir.perturb}), by means of simple computations, the above specified function $\Psi_2^{\, (3)}$ can be expressed as
\begin{eqnarray}
\lefteqn{\Psi_2^{\, (3)}\left[\alpha^+;\alpha_1,\alpha_2,\alpha_3;\frac{\lambda_1}{2}x_1,\frac{\lambda_2 }{2}x_2,\frac{\lambda_{3}}{2}\left(1-x_1-x_2\right)\right]=} \nonumber \\
& = & \sum_{j_1=0}^{+\infty} \frac{\left(\alpha^+\right)_{j_1}}{\left(\alpha_1\right)_{j_1}} \frac{\left(\frac{\lambda_1}{2}x_1\right)^{j_1}}{j_1!} \,  \Psi_2\left[\alpha^++j_1;\alpha_2,\alpha_3;\frac{\lambda_2}{2}x_2,\frac{\lambda_3 }{2}\left(1-x_1-x_2\right)\right]= \nonumber \\
& = & \sum_{j_1=0}^{+\infty} \frac{\left(\alpha^+\right)_{j_1}}{\left(\alpha_1\right)_{j_1}} \frac{\left(\frac{\lambda_1}{2}x_1\right)^{j_1}}{j_1!} \sum_{j_2=0}^{+\infty} \frac{\left(\alpha^++j_1\right)_{j_2}}{\left(\alpha_2\right)_{j_2}} \frac{\left(\frac{\lambda_2}{2}x_2\right)^{j_2}}{j_2!} \cdot \nonumber \\
& \cdot & _1F_1\left[\alpha^++j_1+j_2;\alpha_3;\frac{\lambda_3 }{2}\left(1-x_1-x_2\right)\right] \, ,
\label{eq:ncdir.perturb.altern}
\end{eqnarray}
i.e. as a doubly infinite sum of weighted Kummer's confluent hypergeometric functions. This formula can be usefully adopted as a natural basis for implementing such a function in any statistical package where the generalized hypergeometric function $_1F_1$ is already implemented. In this regard, the code of a possible implementation in $\texttt{R}$ of the algorithm in Eq.~(\ref{eq:ncdir.perturb.altern}) is proposed in Section~\ref{sec:r.funcs}. Similar reasonings clearly apply as far as the computation of the function $\Psi_2^{\, (m)}$ with $m>3$ is concerned.

\noindent Finally, let $\left(X_{i1},X_{i2} \right) \stackrel{\bot}{\sim } \mbox{\normalfont{CNcDir}}^{\, 2}\left(\alpha_1,\alpha_2,\alpha_3,\lambda_1,\lambda_2,\lambda_3\right)$, $i=1, \, \ldots \, ,n$; by Eqs.~(\ref{eq:cncdir.perturb.dens}) and~(\ref{eq:loglik.dir}), given the observed sample $\left\{\underline{x}_i=\left(x_{i1},x_{i2} \right) \right\}_{i=1}^{n}$, the log-likelihood function for the parameter vector $\left(\alpha_1,\alpha_2,\alpha_3,\lambda_1,\lambda_2,\lambda_3\right)$ of the $\mbox{\normalfont{CNcDir}}^{\, 2}$ model is
\begin{eqnarray*}
\lefteqn{l_{\, \mbox{\normalfont{CNcDir}}^{\, 2}}\left(\alpha_1,\alpha_2,\alpha_3,\lambda_1,\lambda_2,\lambda_3;\underline{x}_1, \, \ldots \,, \underline{x}_n\right)=}\\
& + & l_{\, \mbox{\normalfont{Dir}}^{\, 2}}\left(\alpha_1,\alpha_2,\alpha_3;\underline{x}_1, \, \ldots \, ,\underline{x}_n\right)+\sum_{i=1}^{n}\log \, _0F_1 \left(\alpha_1;\frac{\lambda_1}{4}x_{i1}\right)+\\
& + & \sum_{i=1}^{n}\log \, _0F_1 \left(\alpha_2;\frac{\lambda_2}{4}x_{i2}\right)+\sum_{i=1}^{n}\log \, _0F_1 \left[\alpha_3;\frac{\lambda_3}{4}\left(1-x_{i1}-x_{i2}\right)\right]+\\
& - & n \log \, _0F_1 \left(\alpha^+;\frac{\lambda^+}{4}\right) \, .
\end{eqnarray*}

Under well known regularity conditions, the ML estimator $\underline{\hat{\Theta}}$ for the parameter vector $\underline{\theta}$ has an asymptotic multivariate Normal distribution with mean vector $\underline{\theta}$ and covariance matrix which can be approximated by the inverse of the observed information matrix $I(\underline{\hat{\theta}};\, \underline{x})=\{-\partial^{\, 2} l(\underline{\theta};\, \underline{x}) \, / \, \partial \underline{\theta} \, \partial \underline{\theta}^{T} \}_{\underline{\theta}=\underline{\hat{\theta}}}$, $\underline{\hat{\theta}}$ being the ML estimate for $\underline{\theta}$. The maximization of the above log-likelihoods was numerically accomplished by using the built-in-function \texttt{optim} available in the statistical software \texttt{R}. The global maximum of the log-likelihoods is ensured to be actually achieved by using a wide range of starting values for the estimates of the parameters in the optimization algorithm.

The null hypotheses ($H_0$) $\alpha_1=\alpha_2=\alpha_3=1$, $\alpha_1=\alpha_2=1$, $\alpha_1=\alpha_3=1$ and $\alpha_2=\alpha_3=1$ are issued towards all the models at study with the aim to verify if the tails of the data corresponding to all three vertices of the unit simplex or to a few couples of them are valuable enough to be captured by the densities of the aforementioned distributions. If all the above hypotheses are rejected, the model with all the parameters to be estimated is chosen; otherwise, in case of non-rejection of at least one of them, the model corresponding to the highest $p$-value is selected amongst the non-rejected ones with fewest parameters. Such hypotheses can be easily verified by means of a suitable Likelihood Ratio (LR) test with asymptotic Chi-Squared distribution. The construction of such LR test requires the unconstrained maximization of the log-likelihoods as well as the constrained maximization of these latter under the hypotheses of interest. Specifically, the LR statistic is given by $w=-2(\hat{l}_0-\hat{l})$, where $\hat{l}_0$ and $\hat{l}$ are, respectively, the maximum of the log-likelihood under $H_0$ and the maximum of the unrestricted log-likelihood. This statistic can be used to check if the fit of the models with two or three shape parameters set equal to 1 is statistically superior to the one of the model containing all the parameters. According to Wilks' theorem \cite{Wil38}, on varying the sample, $w$ describes the random variable $W$, the distribution of which has been approximated by a Chi-Squared with number of degrees of freedom equal to the number of parameters fixed by $H_0$.

Details of the LR tests carried out herein are listed in Table~\ref{tab:lr.test}, which indicates that the CNcDir and the NcDir models are the only for which the null hypothesis $\alpha_1=\alpha_2=\alpha_3=1$ is not rejected with the highest $p-$value; therefore, their densities allow for the tails of the data corresponding to all three vertices of $\mathcal{S}^2$ to be captured. On the contrary, none of the four hypotheses at study is significantly supported by the data with reference to the remaining models; hence, all three shape parameters of these models are estimated. In this regard, the ML estimates for the parameters of the chosen models and the asymptotic standard errors of the corresponding estimators are reported in Table~\ref{tab:mle}.

\begin{table}[htp]
\centering
\small
\caption{Likelihood Ratio (LR) tests to evaluate the null hypotheses ($H_0$) that all three or two of the shape parameters $\alpha_1, \, \alpha_2, \, \alpha_3$ are equal to 1 for the bivariate $\mbox{\normalfont{CNcDir}}$, $\mbox{\normalfont{NcDir}}$, $\mbox{\normalfont{Dir}}$ and $\mbox{\normalfont{KB}}$ models: maximum value of the log-likelihood under $H_0$ ($\hat{l}_0$), maximum value of the unrestricted log-likelihood ($\hat{l}$), LR statistic ($w$), degrees of freedom (df) of the corresponding approximated $\chi^2$ distribution and $p$-value. The $p$-values corresponding to the chosen models are written in bold ($n=346$).}\label{tab:lr.test}
\begin{tabular}{c||c|c|c|c|c|c}
Model & $H_0$ & $\hat{l}_0$ & $\hat{l}$ & $w$ & df & $p$-value \\ \hline \hline
\multirow{4}{1.2cm}{$\mbox{\normalfont{CNcDir}}^{\, 2}$} & $\alpha_1=\alpha_2=\alpha_3=1$ & 262.0291 & \multirow{4}{1.2cm}{262.7820} & 1.5058 & 3 & \textbf{0.6809}\\
\cline{2-3} \cline{5-7} & $\alpha_1=\alpha_2=1$ & 262.3874 & & 0.7892 & \multirow{3}{0.2cm}{2} & 0.6739 \\
\cline{2-3} \cline{5-5} \cline{7-7} & $\alpha_1=\alpha_3=1$ & 262.0826 & & 1.3988 & & 0.4969 \\ 
\cline{2-3} \cline{5-5} \cline{7-7} & $\alpha_2=\alpha_3=1$ & 262.1861 & & 1.1918 & & 0.5511 \\ \hline \hline
\multirow{4}{1.0cm}{$\mbox{\normalfont{NcDir}}^{\, 2}$}  & $\alpha_1=\alpha_2=\alpha_3=1$ & 262.7941 & \multirow{4}{1.2cm}{263.8433} & 2.0984  & 3 & \textbf{0.5522} \\
\cline{2-3} \cline{5-7} & $\alpha_1=\alpha_2=1$ & 263.0525 & & 1.5816 & \multirow{3}{0.2cm}{2} & 0.4535 \\
\cline{2-3} \cline{5-5} \cline{7-7} & $\alpha_1=\alpha_3=1$ & 262.9253 & & 1.8360 & & 0.3993 \\ 
\cline{2-3} \cline{5-5} \cline{7-7} & $\alpha_2=\alpha_3=1$ & 263.1430 & & 1.4006 & & 0.4964 \\ \hline \hline
\multirow{4}{0.6cm}{$\mbox{\normalfont{Dir}}^{\, 2}$} & $\alpha_1=\alpha_2=\alpha_3=1$ & 239.8289 & \multirow{4}{1.2cm}{255.2603} & 30.8628 & 3 & \multirow{4}{1.2cm}{$<.0001$} \\
\cline{2-3} \cline{5-6} & $\alpha_1=\alpha_2=1$ & 240.7546 & & 29.0114 & \multirow{3}{0.2cm}{2} &  \\
\cline{2-3} \cline{5-5} & $\alpha_1=\alpha_3=1$ & 243.9976 & & 22.5254 & &  \\ 
\cline{2-3} \cline{5-5} & $\alpha_2=\alpha_3=1$ & 240.3968 & & 29.7270 & &  \\ \hline \hline
\multirow{4}{0.6cm}{$\mbox{\normalfont{KB}}^{\, 2}$} & $\alpha_1=\alpha_2=\alpha_3=1$ & 239.8813 & \multirow{4}{1.2cm}{255.3020} & 30.8414 & 3 & \multirow{4}{1.2cm}{$<.0001$} \\
\cline{2-3} \cline{5-6} & $\alpha_1=\alpha_2=1$ & 245.2593 & & 20.0854 & \multirow{3}{0.2cm}{2} &  \\
\cline{2-3} \cline{5-5} & $\alpha_1=\alpha_3=1$ & 244.5826 & & 21.4388 & &  \\ 
\cline{2-3} \cline{5-5} & $\alpha_2=\alpha_3=1$ & 240.4078 & & 29.7884 & &  
\end{tabular}
\end{table}  

\begin{table}[htp]
\centering
\small
\caption{Maximum-Likelihood (ML) estimates for the parameters of the bivariate $\mbox{\normalfont{CNcDir}}$, $\mbox{\normalfont{NcDir}}$, $\mbox{\normalfont{Dir}}$ and $\mbox{\normalfont{KB}}$ models selected on the basis of the LR tests carried out in Table \ref{tab:lr.test} and asymptotic standard errors (SE) of the corresponding estimators between round brackets ($n=346$).}\label{tab:mle}
\begin{tabular}{c|c|c||c|c|c}
\multicolumn{2}{c| }{\mbox{Model}} & ML estimate (SE) & \multicolumn{2}{c| }{\mbox{Model}} & ML estimate (SE) \\ \hline \hline
\multirow{6}{1.2cm}{$\mbox{\normalfont{CNcDir}}^{\, 2}$} & $\alpha_1$ & 1 & \multirow{6}{0.9cm}{$\mbox{\normalfont{NcDir}}^{\, 2}$} & $\alpha_1$ & 1\\
\cline{2-3} \cline{5-6} & $\alpha_2$ & 1 & & $\alpha_2$  & 1 \\
\cline{2-3} \cline{5-6} & $\alpha_3$ & 1 & & $\alpha_3$  & 1 \\
\cline{2-3} \cline{5-6} & $\hat{\lambda_1}$ & 42.7802 (8.5087) & & $\hat{\lambda_1}$ & 3.0478 (0.4093) \\ 
\cline{2-3} \cline{5-6} & $\hat{\lambda_2}$ & 48.7569 (9.3385) & & $\hat{\lambda_2}$ & 3.4644 (0.4382) \\
\cline{2-3} \cline{5-6} & $\hat{\lambda_3}$ & 44.1538 (8.6944) & & $\hat{\lambda_3}$ & 3.1084 (0.4121) \\ \hline
\multirow{4}{0.5cm}{$\mbox{\normalfont{Dir}}^{\, 2}$} & $\hat{\alpha_1}$ & 1.2671 (0.0717) & \multirow{4}{0.5cm}{$\mbox{\normalfont{KB}}^{\, 2}$} &  $\hat{\alpha_1}$ & 1.2810 (0.0870) \\
\cline{2-3} \cline{5-6} & $\hat{\alpha_2}$ & 1.3594 (0.0772) &  & $\hat{\alpha_2}$ & 1.3748 (0.0944) \\
\cline{2-3} \cline{5-6} & $\hat{\alpha_3}$ & 1.2818 (0.0726) &  & $\hat{\alpha_3}$ & 1.2543 (0.1187) \\
\cline{5-6} &                  &                 &  & $\hat{\delta}$ & 0.1884 (0.6519)
\end{tabular}
\end{table}

Moreover, the graphical depiction of the fitted densities of the four models under consideration is displayed in Figure~\ref{fig:EstimDens}, from which it is remarkable how much similar are the fitting of the two-dimensional CNcDir and NcDir densities on the one hand (top left and top right panels) and of the two-dimensional Dir and KB densities from the other (bottom left and bottom right panels). In particular, it is noticeable how much comparable are the abilities of the two non-central densities to identify the presence of the data portions next to the vertices of the unit simplex on one side and the behaviors of the two remaining densities on the other, which show all zero limits despite the four-parameter structure of the two-dimensional KB.     

\begin{figure}[ht]
 \centering
 \subfigure
   {\includegraphics[width=6.3cm]{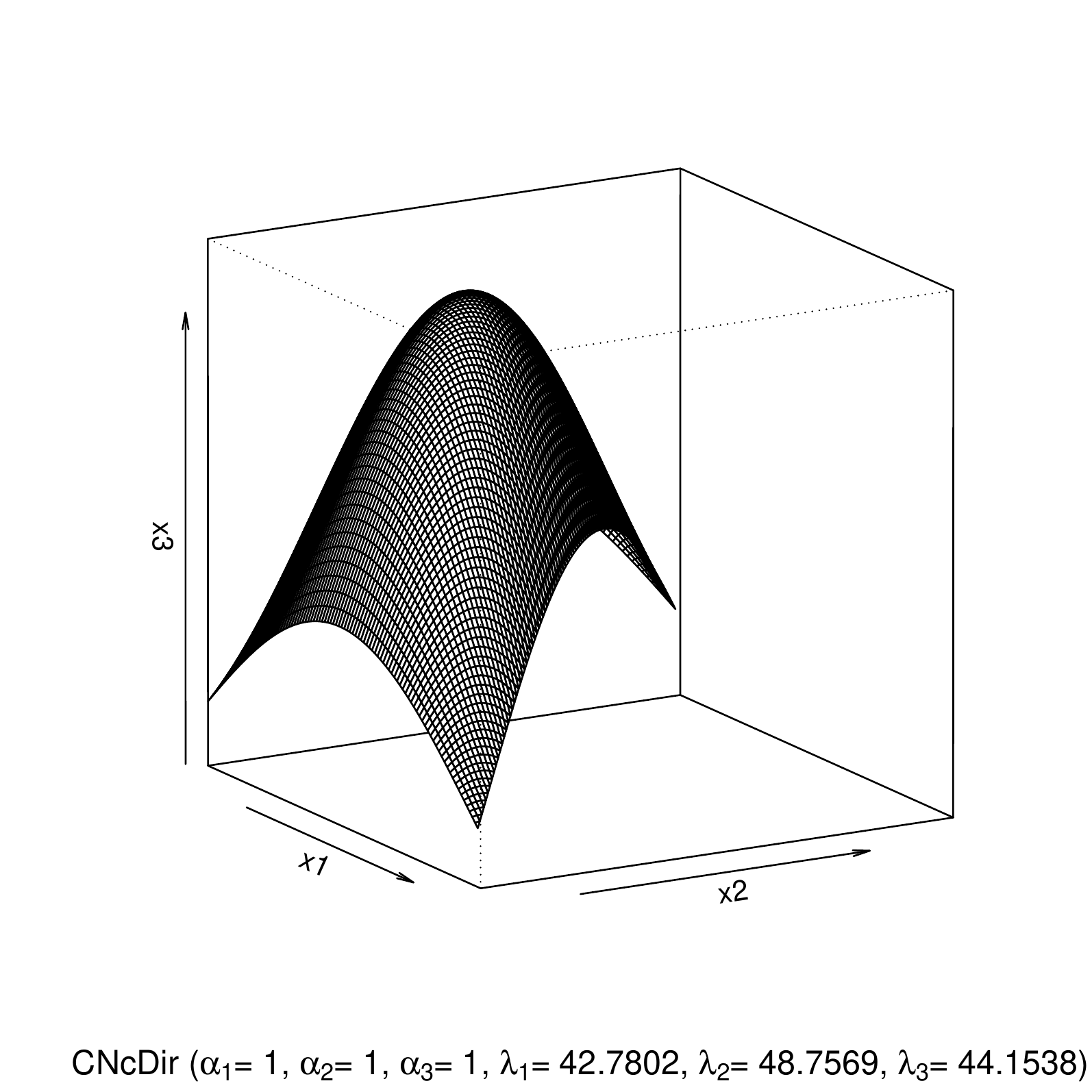}}
 \hspace{5mm}
 \subfigure
   {\includegraphics[width=6.3cm]{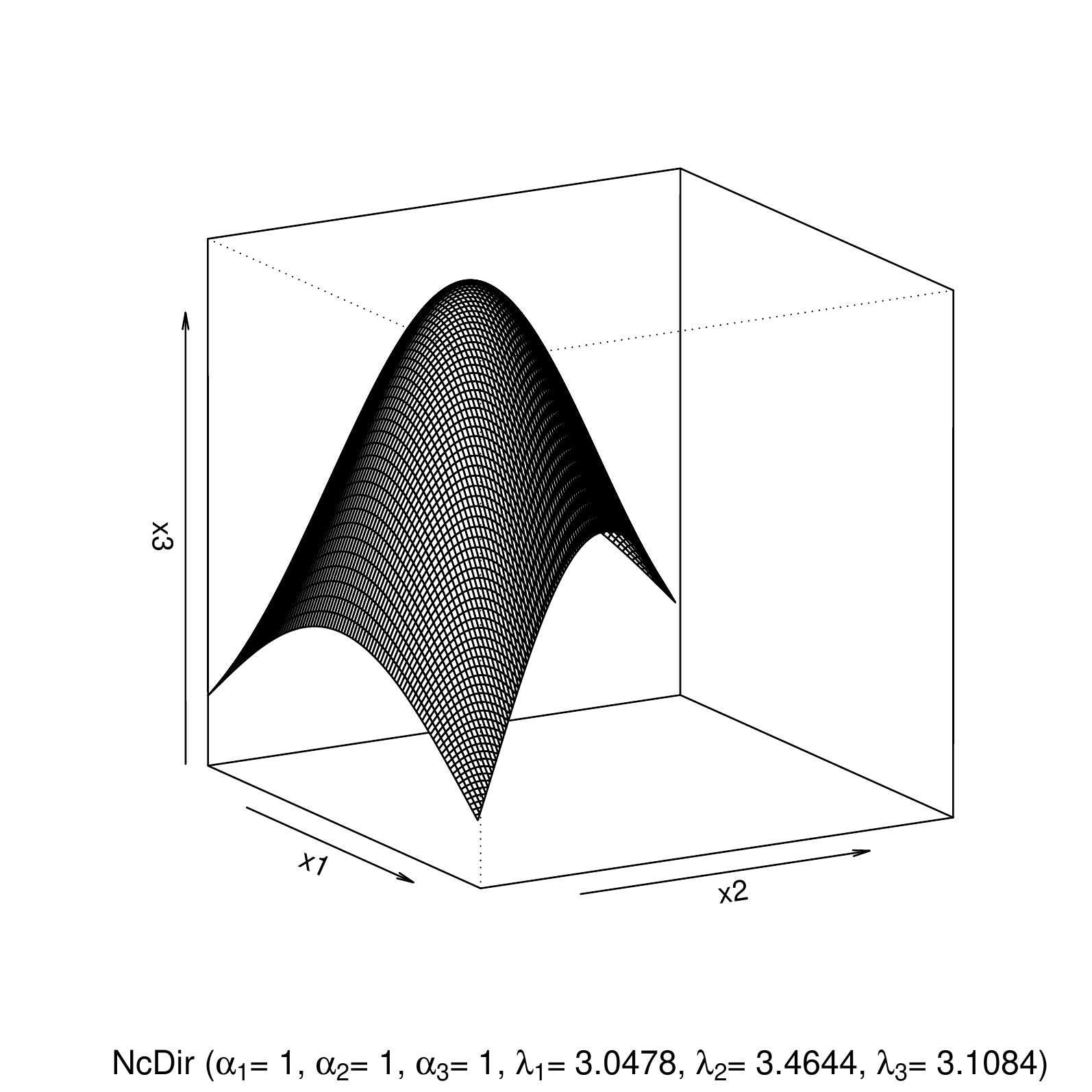}}
 \vspace{2mm}
 \subfigure
   {\includegraphics[width=6.3cm]{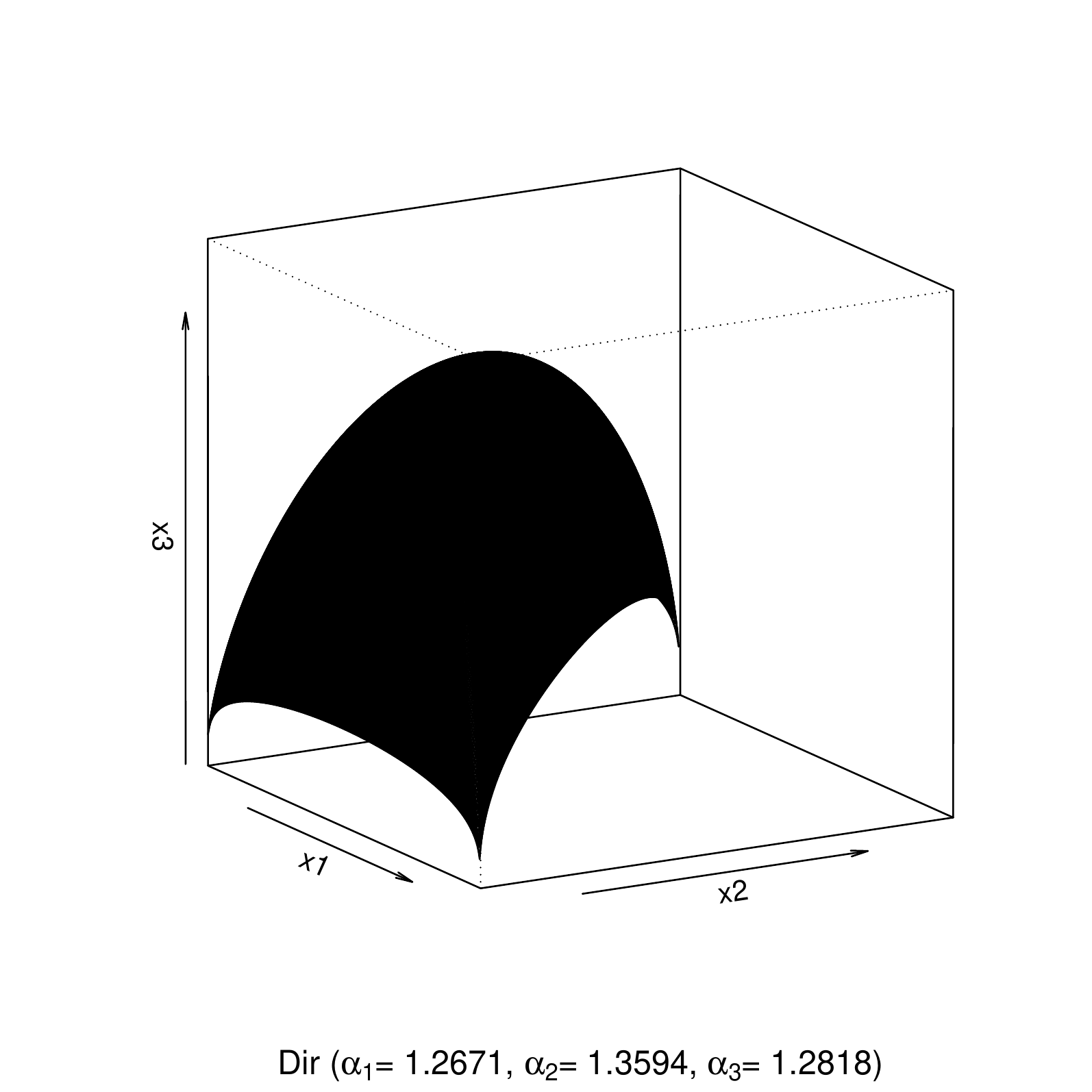}}
 \hspace{5mm}
 \subfigure
   {\includegraphics[width=6.3cm]{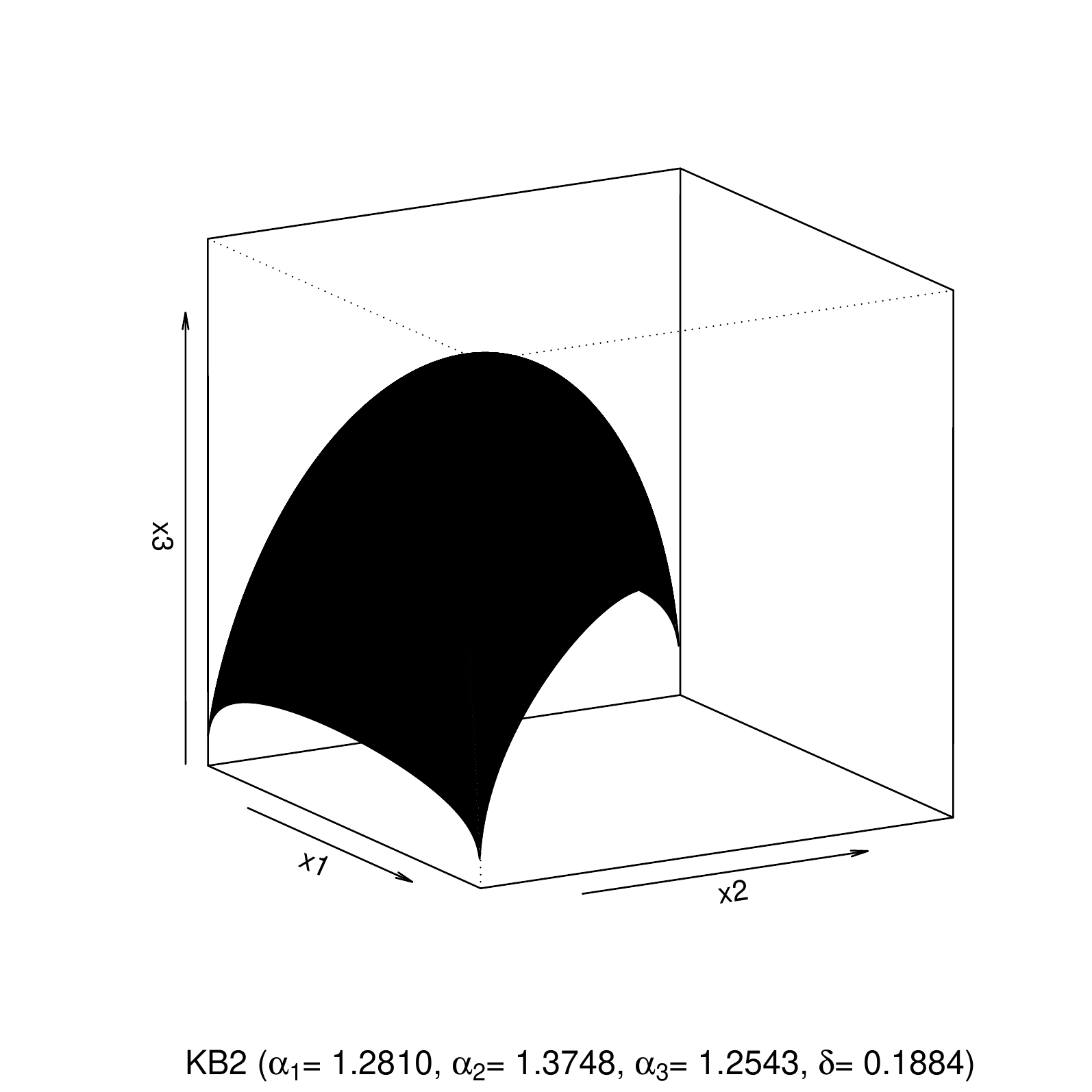}}
 \caption{Plots of the estimated densities of the chosen bivariate $\mbox{\normalfont{CNcDir}}$ (top left-hand panel), $\mbox{\normalfont{NcDir}}$ (top right-hand panel), $\mbox{\normalfont{Dir}}$ (bottom left-hand panel) and $\mbox{\normalfont{KB}}$ (bottom right-hand panel) models.}
 \label{fig:EstimDens}
\end{figure}

The main lesson to be learnt from the results reported herein is an essential equivalence between the $\mbox{\normalfont{CNcDir}}$ and the $\mbox{\normalfont{NcDir}}$ models. However, the fitting of the former is computationally less demanding than the one of the latter. Indeed, the amount of execution time that the $\mbox{\normalfont{CNcDir}}$ model enables to save with regard to the maximization of the log-likelihood is remarkable. This raises the question of further investigating the computational performances of the two models on the basis of efficiency arguments regarding the time to accomplish the Maximum-Likelihood fitting. In this respect, $n=30$ series of random values were generated from each of the two above models for selected combinations of the series size and of the parameter vector in the case $D=2$. Then, the machine-time spent to fit the two models was measured for each replication and means and standard deviations were computed for this latter within each stratum. Hence, the null hypothesis of non-inferiority of the true mean of the CNcDir fitting time with respect to the one of the NcDir is checked by using the one-tailed $Z$ test for large samples. The results of the above described simulation study are listed in Table~\ref{tab:simul.mle}, from which we conclude that the Conditional model is beyond doubt to be preferred to the existing one for the major computational efficiency this permits in carrying out inferences based on the Maximum-Likelihood theory. More precisely, the obtained founds suggest that on average the NcDir fitting is approximately of the order of 90 times slower than the one of the CNcDir model. 

\begin{table}[htp]
\centering
\small
\caption{Means and standard deviations (between round brackets) of the machine-times (in seconds) to execute the Maximum-Likelihood fitting of the $\mbox{CNcDir}^{\, 2}$ and the $\mbox{NcDir}^{\, 2}$ densities to $n=30$ series of $N$ random draws and $p$-values of the one-tailed $Z$ test to assess the null hypothesis $H_0: \, \mu_{CNcDir}-\mu_{NcDir} \, \geq\, 0$ for selected values of $N$ and of the shape parameters and the non-centrality parameters vectors.}\label{tab:simul.mle} 
\begin{tabular}{c|c||c|c|c|c}
\multirow{2}{0.2cm}{$\underline{\alpha}$} & \multirow{2}{0.2cm}{$\underline{\lambda}$} & & \multicolumn{3}{c}{Time ($''$)} \\ 
\cline{4-6} & & & \multicolumn{1}{c| }{$N=25$} & \multicolumn{1}{c| }{$N=50$} & \multicolumn{1}{c }{$N=100$} \\ \hline
1.8 & 0.7 & $\mbox{CNcDir}^{\, 2}$ & 0.9053 (0.5852)  &  1.3360 (1.0834)  &   1.5253 (1.2979) \\ 
1.2 & 1.0 & $\mbox{NcDir}^{\, 2}$ & 54.5803 (41.8910) & 72.9673 (76.5182) & 107.7123 (126.4231) \\
1.5 & 0.9 & $p$-value & $<.0001$ & $<.0001$ &	$<.0001$  \\ \hline \hline
0.7 & 4.6 & $\mbox{CNcDir}^{\, 2}$ & 1.0907 (0.8587)  &   1.7207 (1.0176)   &   1.3610 (0.9568) \\ 
1.3 & 0.5 & $\mbox{NcDir}^{\, 2}$ & 72.3550 (93.5070) & 110.6507 (138.8951) & 105.8680 (89.4877) \\
1.9 & 3.8 & $p$-value & $<.0001$ & $<.0001$ &	$<.0001$  \\ \hline \hline
2.1 & 0.8 & $\mbox{CNcDir}^{\, 2}$ & 1.1537 (0.6813)  &   1.2613 (0.9024)   &   1.4517 (0.7845) \\ 
0.2 & 2.9 & $\mbox{NcDir}^{\, 2}$ & 91.0197 (88.6285) & 142.1647 (101.2835) & 172.0060 (98.3391) \\
0.6 & 4.2 & $p$-value & $<.0001$ & $<.0001$ &	$<.0001$  \\ \hline \hline
0.8 & 2.4 & $\mbox{CNcDir}^{\, 2}$ & 0.7337 (0.3396)  &   1.0263 (0.5822)   &   1.2010 (0.6801) \\ 
0.9 & 1.7 & $\mbox{NcDir}^{\, 2}$ & 82.7857 (86.6345) & 121.7610 (71.8105) & 158.0427 (78.6970) \\
0.4 & 2.8 & $p$-value & $<.0001$ & $<.0001$ &	$<.0001$  
\end{tabular}
\end{table}

\section{Conclusions}
\label{sec:concl}

In this paper a new non-central extension of the Dirichlet model was obtained by conditioning the existing Non-central Dirichlet on the sum of the Non-central Chi-Squared random variables involved in its definition. The resulting Conditional Non-central Dirichlet distribution represents the multivariate generalization of the Conditional Doubly Non-central Beta one, namely the new non-central version of the Beta distribution which has been recently studied by \cite{Ors21}. An in-depth analysis of this new multidimensional model was carried out herein in order to highlight similarities and differences with the existing one. Specifically, despite the densities of the new and the standard Non-central Dirichlet distributions share the same complex mixture type form, from the comparison between their perturbation representations the surprisingly greater tractability and interpretability of the former over the latter emerge. Hence, the Conditional Non-central Dirichlet distribution enables to overcome the strong analytical and mathematical limitations affecting the existing Non-central Dirichlet by exhibiting a substantially simpler density function than that of the latter. At the same time, the new non-central generalization of the Dirichlet model preserves the applicative potential of the existing one by allowing its density to take on arbitrary positive and finite limits which enables to properly capture the data portions next to the vertices of the support. These are the reasons why we hope this model may attract a wide range of applications in statistics.   

\section{Appendix}
\label{sec:r.funcs}

In the present section we shall propose a possible translation into \texttt{R} code of the algorithm specified in Eq.~(\ref{eq:ncdir.perturb.altern}) to compute the generalization to three dimensions of the Humbert's confluent hypergeometric function. In light of Eq.~(\ref{eq:ncdir.perturb.dens}), this issue is instrumental in assessing the bivariate Non-central Dirichlet density. Specifically, the present implementation consists of two routines. The first routine, named \texttt{psi2.3}, implements the external infinite sum in Eq.~(\ref{eq:ncdir.perturb.altern}). This routine internally calls the other one, named \texttt{intpsi2.3}, which computes the Humbert's confluent hypergeometric function as infinite sum of Kummer's confluent hypergeometric functions. The two routines share the same following arguments:

\begin{itemize}
\item \textit{x1}, \textit{x2}: vectors of the first and of the second components of the data on $\mathcal{S}^{\, 2}$;
\item \textit{shape1}, \textit{shape2}, \textit{shape3}: the three shape parameters $\alpha_1$, $\alpha_2$, $\alpha_3$ of the bivariate Non-central Dirichlet density;
\item \textit{ncp1}, \textit{ncp2}, \textit{ncp3}: the three non-centrality parameters $\lambda_1$, $\lambda_2$, $\lambda_3$ of the bivariate Non-central Dirichlet density;
\item \textit{tol}: tolerance, where zero means keeping on the iterations until the additional terms do not change the partial sum of the infinite series under consideration;
\item \textit{maxiter}: maximum number of iterations to perform;
\item \textit{debug}: Boolean, where TRUE means returning debugging information whereas FALSE means returning only the final evaluation.
\end{itemize}

The code of the first routine is as follows:

\noindent \texttt{psi2.3<- }

\noindent \texttt{function(x1,x2,shape1,shape2,shape3,ncp1,ncp2,ncp3,tol,maxiter,debug) \hspace{0.1cm} \{ } 

\noindent \hspace{0.2cm} \texttt{L=c(shape1,shape2,shape3)}

\noindent \hspace{0.2cm} \texttt{S=sum(L)}

\noindent \hspace{0.2cm} \texttt{y1=(ncp1/2)*x1}

\noindent \hspace{0.2cm} \texttt{coef<-1}

\noindent \hspace{0.2cm} \texttt{temp<-intpsi2.3(x1=x1,x2=x2,shape1=shape1,shape2=shape2,shape3=shape3,}

\noindent \hspace{1.2cm} \texttt{ncp1=ncp1,ncp2=ncp2,ncp3=ncp3,tol=tol,maxiter=maxiter,debug=debug)}

\noindent \hspace{0.2cm} \texttt{out<-NULL}
     
\noindent \hspace{0.2cm} \texttt{for(n in seq\_len(maxiter)) \hspace{0.1cm} \{ }

\noindent \hspace{0.5cm} \texttt{coef<-coef*((S/L[1])*y1/n)}

\noindent \hspace{0.5cm} \texttt{fac<-coef*}

\noindent \hspace{1.2cm} \texttt{intpsi2.3(x1=x1,x2=x2,shape1=shape1+n,shape2=shape2,shape3=shape3,}

\noindent \hspace{1.2cm} \texttt{ncp1=ncp1,ncp2=ncp2,ncp3=ncp3,tol=tol,maxiter=maxiter,debug=debug)}

\noindent \hspace{0.5cm} \texttt{series<-temp+fac}

\noindent \hspace{0.5cm} \texttt{if(debug) \hspace{0.1cm} \{ }

\noindent \hspace{0.9cm} \texttt{out<-c(out,fac)}

\noindent \hspace{0.5cm} \texttt{\} }

\noindent \hspace{0.5cm} \texttt{if(hypergeo::isgood(series-temp,tol)) \hspace{0.1cm} \{ }

\noindent \hspace{0.9cm} \texttt{if(debug) \hspace{0.1cm} \{ }

\noindent \hspace{1.3cm} \texttt{return(list(series,out))}

\noindent \hspace{0.9cm} \texttt{\} }

\noindent \hspace{0.9cm} \texttt{else \{ }

\noindent \hspace{1.3cm} \texttt{return(series)}

\noindent \hspace{0.9cm} \texttt{\} }

\noindent \hspace{0.5cm} \texttt{\} }

\noindent \hspace{0.5cm} \texttt{temp<-series}

\noindent \hspace{0.5cm} \texttt{S<-S+1}

\noindent \hspace{0.5cm} \texttt{L[1]<-L[1]+1}

\noindent \hspace{0.2cm} \texttt{\} }

\noindent \hspace{0.2cm} \texttt{if(debug) \hspace{0.1cm} \{ }

\noindent \hspace{0.5cm} \texttt{return(list(series,out))}

\noindent \hspace{0.2cm} \texttt{\}}

\noindent \texttt{\}}

The code of the second routine is as follows:

\noindent \texttt{intpsi2.3<- }

\noindent \texttt{function(x1,x2,shape1,shape2,shape3,ncp1,ncp2,ncp3,tol,maxiter,debug) \hspace{0.1cm} \{ } 

\noindent \hspace{0.2cm} \texttt{L=c(shape1,shape2,shape3)}

\noindent \hspace{0.2cm} \texttt{S=sum(L)}

\noindent \hspace{0.2cm} \texttt{y2=(ncp2/2)*x2}

\noindent \hspace{0.2cm} \texttt{y3=(ncp3/2)*(1-x1-x2)}

\noindent \hspace{0.2cm} \texttt{coef<-1}

\noindent \hspace{0.2cm} \texttt{temp<-hypergeo::genhypergeo(U=S,L=L[3],z=y3,tol=tol,maxiter=maxiter,}

\noindent \hspace{1.2cm} \texttt{check\_mod=TRUE,polynomial=FALSE,debug=FALSE)}

\noindent \hspace{0.2cm} \texttt{out<-NULL}
     
\noindent \hspace{0.2cm} \texttt{for(m in seq\_len(maxiter)) \hspace{0.1cm} \{ }

\noindent \hspace{0.5cm} \texttt{coef<-coef*((S/L[2])*y2/m)}

\noindent \hspace{0.5cm} \texttt{fac<-coef*hypergeo::genhypergeo(U=S+1,L=L[3],z=y3,tol=tol,}

\noindent \hspace{1.2cm} \texttt{maxiter=maxiter,check\_mod=TRUE,polynomial=FALSE,debug=FALSE)}

\noindent \hspace{0.5cm} \texttt{series<-temp+fac}

\noindent \hspace{0.5cm} \texttt{if(debug) \hspace{0.1cm} \{ }

\noindent \hspace{0.9cm} \texttt{out<-c(out,fac)}

\noindent \hspace{0.5cm} \texttt{\} }

\noindent \hspace{0.5cm} \texttt{if(hypergeo::isgood(series-temp,tol)) \hspace{0.1cm} \{ }

\noindent \hspace{0.9cm} \texttt{if(debug) \hspace{0.1cm} \{ }

\noindent \hspace{1.3cm} \texttt{return(list(series,out))}

\noindent \hspace{0.9cm} \texttt{\} }

\noindent \hspace{0.9cm} \texttt{else \{ }

\noindent \hspace{1.3cm} \texttt{return(series)}

\noindent \hspace{0.9cm} \texttt{\} }

\noindent \hspace{0.5cm} \texttt{\} }

\noindent \hspace{0.5cm} \texttt{temp<-series}

\noindent \hspace{0.5cm} \texttt{S<-S+1}

\noindent \hspace{0.5cm} \texttt{L[2]<-L[2]+1}

\noindent \hspace{0.2cm} \texttt{\} }

\noindent \hspace{0.2cm} \texttt{if(debug) \hspace{0.1cm} \{ }

\noindent \hspace{0.5cm} \texttt{return(list(series,out))}

\noindent \hspace{0.2cm} \texttt{\}}

\noindent \texttt{\}}


\end{document}